\documentclass[11pt,b5paper,notitlepage]{article}
\usepackage[b5paper, margin={0.5in,0.65in}]{geometry}

\usepackage{amsmath,amscd,amssymb,amsthm,mathrsfs,amsfonts,layout,indentfirst,graphicx,caption,mathabx, stmaryrd,appendix,calc,imakeidx,upgreek} 
\usepackage[dvipsnames]{xcolor}
\usepackage{palatino}  
\usepackage{slashed} 
\usepackage{mathrsfs} 
\usepackage{extarrows} 
\usepackage{enumitem} 

\usepackage{fancyhdr} 

\usepackage{tikz-cd}
\usepackage[nottoc]{tocbibind}   

\makeindex

\usepackage{lipsum}
\let\OLDthebibliography\thebibliography
\renewcommand\thebibliography[1]{
	\OLDthebibliography{#1}
	\setlength{\parskip}{0pt}
	\setlength{\itemsep}{2pt} 
}

\allowdisplaybreaks  
\usepackage{latexsym}
\usepackage{chngcntr}
\usepackage[colorlinks,linkcolor=blue,anchorcolor=blue, linktocpage,
]{hyperref}
\hypersetup{ urlcolor=cyan,
	citecolor=[rgb]{0,0.5,0}}

\theoremstyle{definition}
\newtheorem{df}{Definition}[section]
\newtheorem{eg}[df]{Example}

\newtheorem{rem}[df]{Remark}
\newtheorem{ass}[df]{Assumption}
\newtheorem{cv}[df]{Convention}

\newtheorem{thma}[df]{Theorem}

\theoremstyle{plain}
\newtheorem{thm}[df]{Theorem}

\newtheorem{pp}[df]{Proposition}
\newtheorem{co}[df]{Corollary}
\newtheorem{lm}[df]{Lemma}

\newtheorem{Mthm}{Theorem}

\newcommand{\fk}{\mathfrak}
\newcommand{\mc}{\mathcal}
\newcommand{\wtd}{\widetilde}
\newcommand{\wht}{\widehat}
\newcommand{\wch}{\widecheck}
\newcommand{\ovl}{\overline}

\newcommand{\End}{\mathrm{End}} 
\newcommand{\id}{\mathbf{1}}
\newcommand{\Hom}{\mathrm{Hom}}
\newcommand{\Conf}{\mathrm{Conf}}
\newcommand{\Res}{\mathrm{Res}}

\newcommand{\Rep}{\mathrm{Rep}}

\newcommand{\bk}[1]{\langle {#1}\rangle}

\newcommand{\scr}{\mathscr}

\newcommand{\im}{\mathbf{i}}

\newcommand{\SX}{{S_{\fk X}}}

\newcommand{\mbb}{\mathbb}
\newcommand{\mbf}{\mathbf}
\newcommand{\bsb}{\boldsymbol}
\newcommand{\blt}{\bullet}

\newcommand{\Vbb}{\mathbb V}
\newcommand{\Ubb}{\mathbb U}
\newcommand{\Xbb}{\mathbb X}
\newcommand{\Wbb}{\mathbb W}
\newcommand{\Mbb}{\mathbb M}

\newcommand{\Cbb}{\mathbb C}
\newcommand{\Nbb}{\mathbb N}
\newcommand{\Zbb}{\mathbb Z}
\newcommand{\Pbb}{\mathbb P}
\newcommand{\Rbb}{\mathbb R}

\newcommand{\cbf}{\mathbf c}
\newcommand{\wt}{\mathrm{wt}}

\newcommand{\Sbf}{\mathbf{S}}

\newcommand{\Nbf}{\mathbf N}

\newcommand{\pr}{\mathrm {pr}}

\newcommand{\vbf}{\mathbf v}

\newcommand{\wbf}{\mathbf w}
\newcommand{\CB}{\mathrm{CB}}
\newcommand{\Perm}{\mathrm{Perm}}
\newcommand{\Orb}{\mathrm{Orb}}
\newcommand{\Lss}{{L_{0,\mathrm{s}}}}
\newcommand{\Lni}{{L_{0,\mathrm{n}}}}

\usepackage{tipa} 

\newcommand{\tipaz}{\text{\textctyogh}}
\newcommand{\tipaomega}{\text{\textcloseomega}}
\newcommand{\tipae}{\text{\textrhookschwa}}

\newcommand{\tipak}{\text{\texthtk}}

\usepackage{tipx}
\newcommand{\tipxgamma}{\text{\textfrtailgamma}}
\newcommand{\tipxcc}{\text{\textctstretchc}}
\newcommand{\tipxphi}{\text{\textqplig}}

\numberwithin{equation}{section}

\title{Genus-zero Permutation-twisted Conformal Blocks for Tensor Product Vertex Operator Algebras: The Tensor-factorizable Case}
\author{{\sc Bin Gui}
}
\date{}
\begin{document}\sloppy 
	\pagenumbering{arabic}
	\setcounter{section}{-1}

	\maketitle

\newcommand\blfootnote[1]{%
	\begingroup
	\renewcommand\thefootnote{}\footnote{#1}%
	\addtocounter{footnote}{-1}%
	\endgroup
}


\begin{abstract}
Let $\Vbb=\bigoplus_{n\in\Nbb}\Vbb(n)$ be a  vertex operator algebra (VOA),  let $E$ be a finite set, and let $G$ be a subgroup of the permutation group $\Perm(E)$ which acts on $\Vbb^{\otimes E}=\bigotimes_{e\in E}\Vbb$ in a natural way. For each $g\in G$, the $g$-twisted $\Vbb^{\otimes E}$-modules were first constructed and characterized in \cite{BDM02} when $g\curvearrowright E$ has only one orbit, i.e., $E=\bk{g}e$ for some $e\in E$. In general, if $E$ is a disjoint union of several $\bk{g}$-orbits $E=E_1\sqcup\cdots\sqcup E_k$, and if for each orbit $E_i$  one chooses a $g$-twisted $\Vbb^{\otimes E_i}$-module $\mc W_i$, then $\mc W=\mc W_1\otimes\cdots\otimes\mc W_k$ is a $g$-twisted $\Vbb^{\otimes E}$-module. A direct sum of such $\mc W$ is called  a \textbf{$\otimes$-factorizable} $g$-twisted $\Vbb^{\otimes E}$-module. It is known that  all $g$-twisted modules are $\otimes$-factorizable if $\Vbb$ is rational \cite{BDM02}.

In this article, we use the main result of \cite{Gui24b} to construct an explicit isomorphism from the space of genus-$0$ conformal blocks associated to the  $G$-twisted $\Vbb^{\otimes E}$-modules (i.e., $g$-twisted $\Vbb^{\otimes E}$-modules for some $g\in G$) that are $\otimes$-factorizable to the space of conformal blocks associated to the untwisted $\Vbb$-modules and a branched covering $C$ of the Riemann sphere $\Pbb^1$. When $\Vbb$ is CFT-type, $C_2$-cofinite, and rational, we use the above result, the (untwisted) factorization property \cite{DGT22}, and the Riemann-Hurwitz formula to completely determine the fusion rules among $G$-twisted $\Vbb^{\otimes E}$-modules. 

Furthermore, assuming $\Vbb$ is as above, we prove that the sewing/factorization of genus-$0$ $G$-twisted $\Vbb^{\otimes E}$-conformal blocks holds, and corresponds to the sewing/factorization of  untwisted $\Vbb$-conformal blocks associated to the branched coverings of $\Pbb^1$. This proves, in particular, the operator product expansion (i.e., associativity) of $G$-twisted $\Vbb^{\otimes E}$-intertwining operators (a key ingredient of the $G$-crossed braided tensor category $\Rep^G(\Vbb^{\otimes E})$ of the $G$-twisted $\Vbb^{\otimes E}$-modules) without assuming that the fixed point subalgebra $(\Vbb^{\otimes E})^G$ is $C_2$-cofinite (and rational), a condition known so far only when $G$ is solvable and remains a conjecture in the general case. More importantly, this result implies that besides the fusion rules, the associativity isomorphism of $\Rep^G(\Vbb^{\otimes E})$ is also characterized by the higher genus data of untwisted $\Vbb$-conformal blocks, which gives a new insight into the category $\Rep^G(\Vbb^{\otimes E})$.

We also discuss the applications to conformal nets, which are indeed the original motivations for the author to study the subject of this paper.
\end{abstract}

\newpage
\tableofcontents





	
	

	

\newpage

\section{Introduction}

\subsection{Motivations from fusion rule calculations}
An important problem that has long attracted people in orbifold conformal field theory is the following: given a nice (say, CFT-type, $C_2$-cofinite,  rational) vertex operator algebra (VOA) $\Vbb$ (or a completely rational conformal net $\mc A$) with a finite automorphism group $G$, once we know the tensor category $\Rep(\Vbb)$ of untwisted $\Vbb$-modules, what do we know about   the category $\Rep^G(\Vbb)$ of $G$-twisted $\Vbb$-modules?

The most studied examples of this problem are permutation orbifold VOAs, namely, the permutation action of $G$ on $\Ubb=\Vbb^{\otimes E}\equiv\bigotimes_{e\in E}\Vbb$ where $E$ is a finite set and $G$ is a subgroup of the permutation group $\Perm(E)$, cf. Def. \ref{lb72}. As shown in \cite{BHS98,BDM02} (cf. \cite{LX04,KLX05} for the conformal net version), all $G$-twisted $\Ubb$-modules (i.e., $g$-twisted $\Ubb$-modules for some $g\in G$) can be explicitly constructed from (untwisted) $\Vbb$-modules. The $n$-fold covering maps $z\mapsto z^n$ for $\Pbb^1$ (for VOAs) or $S^1$ (for conformal nets) play a central role in these constructions. Recently, \cite{DXY22} gave a construction using Zhu's algebras.

\subsubsection*{Fusion rules}

Fusion rules among $G$-twisted $\Ubb$-modules have also been investigated by physicists \cite{BHS98,Ban98,Ban02} and mathematicians. On the mathematics side, the fusion rules among $\Zbb_2$-twisted modules of $\Vbb^{\otimes 2}$ or $\mc A^{\otimes 2}$, or among two twisted modules and an untwisted one, have been completely determined in \cite{LX04,KLX05,DLXY24}. It turns out that the structure theory of completely rational conformal nets relies essentially on the idea of  cyclic permutation orbifolds \cite{KLM01,LX04}.

On the abstract tensor category/modular functor level, the computation of fusion rules of $\Rep^G(\Ubb)$ (or any modular category) is more or less complete. \cite{BS11} used topological methods to construct a $G$-crossed  braided weakly fusion category $\scr C$ extending the Deligne product $\Rep(\Vbb)^{\boxtimes E}\simeq\Rep(\Ubb)$;  the fusion rules of $\scr C$ can be described easily in terms of the higher genus fusion rules of $\Rep(\Vbb)$. \cite{BS11} gives us a hint on what the fusion ring of $\Rep^G(\Ubb)$ looks like: If the expected property that $\scr C$ is rigid were proved (\cite{BS11} proved this only for the $\Zbb_2$-permutation), then a result of \cite{ENO10} would imply that $\scr C$ and $\Rep^G(\Ubb)$ have equivalent fusion rings, even though they are not necessarily equivalent as fusion categories (cf. \cite{Bis20,EG18}).
More recently, \cite{BJ19} and \cite{Del19} provided rigorous and complete algorithms for computing the fusion ring of $\Rep^G(\Ubb)$ (or any (spherical) crossed extension of a modular category). In particular, in the case that $G$ is generated by a one cycle permutation, \cite{BJ19} explicitly calculated the fusion rings, which agree with \cite{BS11}.

However, the important progress mentioned in the previous paragraph does not mean that the task of computing permutation fusion rules for VOAs or conformal nets is complete. Indeed,  when applied to the VOA/conformal net context, the above categorical results do not directly give us the fusion rule among three twisted $\Vbb$- or $\mc A$-modules constructed explicitly in \cite{BDM02,LX04,KLX05}.  In other words, \textit{we still need to identify the objects constructed categorically in \cite{BS11,BJ19,Del19} with those constructed explicitly in the VOA/conformal net context}. On the other hand, the objects in \cite{LX04,KLX05,DLXY24} are explicit, but their results on the fusion rules are  far from complete even in the special case of cyclic permutations.

A main goal of this paper is to completely determine the fusion rules among $G$-twisted $\Vbb^{\otimes E}$-modules (as constructed in \cite{BDM02}) in terms of the fusion rules among untwisted $\Vbb$-modules. Here,  $\Vbb$ is CFT-type, $C_2$-cofinite, and rational. See Cor. \ref{lb50} for the result.\footnote{Our method is geometric. After this paper was finished, an algebraic computation of the fusion rules among $G$-twisted modules was given in \cite{DXY24} in terms of the $S$-matrix of $\Vbb$, where $G$ is a cyclic group.}

\subsubsection*{Why higher genus data appear in the fusion rules}

We should do more than just calculate the fusion rules. The (genus-$0$) fusion rules describe the dimensions of the spaces of conformal blocks associated to $\Pbb^1$ with three marked points. However, the expression of permutation-twisted fusion rules (as calculated in the previously mentioned literature) involves dimensions of untwisted conformal blocks on higher genus Riemann surfaces with arbitrary numbers of marked points. \textit{We should develop a theory that explains this higher genus phenomenon; the complete computation of permutation-twisted fusion rules should follow as a consequence.}

A rough explanation of this phenomenon is this: the permutation-twisted conformal blocks associated to a pointed compact Riemann surface $X$ correspond to untwisted conformal blocks associated to a branched covering of $X$. Bantay already noticed this fact in \cite{Ban98,Ban02}. He used this idea to compute the modular data of $\Rep^G(\Ubb)$, and to calculate the fusion rules indirectly using these modular data. But the coverings used by Bantay are mainly the unbranched coverings of tori, which are also tori. From Bantay's work, it is not clear what the correct branched coverings (called the \textbf{permutation coverings} in this article) are in general.

In \cite{BS11}, Barmeier and Schweigert have made important progress on this problem by constructing the \textit{topological} permutation coverings. Thus, to complete the story, we should not only equip the permutation covering $C$ with a complex structure (which is a standard process), but also determine the locations and the local coordiates of the marked points, and find the correspondence between these marked points and the untwisted $\Vbb$-modules. This task is nontrivial, and we will briefly explain our answer in the following section of the Introduction. 

\subsection{Main result: the twisted/untwisted correspondence}

\subsubsection*{Positively $N$-pathed Riemann spheres with local coordinates}

We are mainly interested in the case that the (complex analytic) permutation covering $C$ is the branched covering of $\Pbb^1$, but in fact many discussions also apply to other compact Riemann surfaces. 

To determine the covering $\varphi:C\rightarrow\Pbb^1$ and the marked points of $C$ with the correct labels, we need to add not only distinct marked points $x_1,\dots,x_N\in C$ and local coordinates $\eta_j$ at each $x_j$ (i.e., an analytic injective function from a neighborhood $W_j$ of $x_j$ to $\Cbb$ sending $x_j$ to $0$), but also paths $\upgamma_1,\dots,\upgamma_N$ in $\Pbb^1\setminus\Sbf$ (where $\Sbf=\{x_1,\dots,x_N\}$) with common end point $\upgamma_\blt(1)$. We assume that  each neighborhood $W_j$ is an open disc contain only $x_j$ among $x_1,\dots,x_N$. We assume that $\upgamma_j(0)$ is in the punctured disc $W_j\setminus\{x_j\}$ and satisfies
\begin{align}
\eta_j(\upgamma_j(0))>0.\label{eq101}	
\end{align}
Then the data
\begin{align}
\fk P=(\Pbb^1;x_1,\dots,x_N;\eta_1,\dots,\eta_N;\upgamma_1,\dots,\upgamma_N)=(\Pbb_1;x_\blt;\eta_\blt;\upgamma_\blt)\label{eq102}
\end{align}
is called a \textbf{positively $N$-pathed Riemann sphere with local coordinates}. See Figure \ref{fig1} for an example.
\begin{figure}[h]
	\centering
	\includegraphics[width=0.3\linewidth]{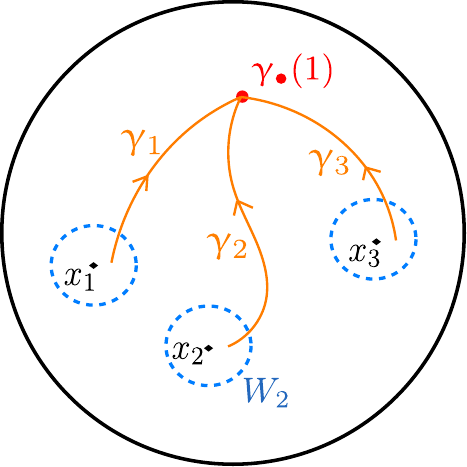}
	\caption{. A positively $3$-pathed Riemann sphere with local coordinates}
	\label{fig1}
\end{figure}

Indeed, in the main body of this article, we assume one more condition on  $\fk P$ for the sake of convenience. The fundamental group
\begin{align*}
\Gamma=\pi_1(\Pbb^1\setminus\Sbf,\upgamma_\blt(1))	
\end{align*}
is free with $N-1$ free generators. For each $j$, let $\upepsilon_j$ be an anticlockwise circle in $W_j\setminus\{x_j\}$ from and to $\upgamma_j(0)$, and let
\begin{align*}
\upalpha_j=\upgamma_j^{-1}\upepsilon_j\upgamma_j.	
\end{align*}
Our assumption is that $\Gamma$ is generated by the homotopy classes $[\upalpha_1],\dots,[\upalpha_N]$. Thus, the monodromy on $\Pbb^1\setminus\Sbf$ is determined by that around the marked points.

Now, fix a homomorphism $\Gamma\rightarrow\Perm(E)$ (where $\Perm(E)$ is the group of bijections of $E$), namely, a group action of $\Gamma$ on the finite set $E$. We let $g_j$ be the action of $[\upalpha_j]$. Then these $g_\blt$ determine the action of $\Gamma$, and we call the $g_\blt$ arising from the actions of $\Gamma$ \textit{admissible}  (with respect to $\fk P$). These group elements determine the types of twisted modules associated to the marked points.

\subsubsection*{Conformal blocks}\label{lb67}

Let us pause for a moment and discuss the meaning of conformal blocks. This will also motivate the definition of permutation coverings. \cite{FB04} defined an (infinite rank) vector bundle $\scr U_C$ (called \textbf{sheaf of VOA} for $\Ubb$ in this article) over any Riemann surface $C$ whose fibers are equivalent to the VOA $\Ubb$, and whose transition functions are described by Huang's change of coordinate formula \cite{Hua97}. If we associate untwisted $\Ubb$-modules $\mc W_1,\dots,\mc W_N$ to the marked points $x_1,\dots,x_N$, then a conformal block was defined by \cite{FB04} to be a linear functional $\uppsi:\mc W_\blt=\mc W_1\otimes\cdots\otimes\mc W_N\rightarrow\Cbb$ such that for each $w_\blt=w_1\otimes\cdots\otimes w_N$, the following condition holds: for all $j$, the expressions 
\begin{align}
u\in\Ubb\mapsto\uppsi(w_1\otimes\cdots\otimes Y(u,z)w_j\otimes\cdots\otimes w_N)	\label{eq100}
\end{align}
converge absolutely as a series of $z$ when $|z|$ is reasonably large, and (assuming the a trivialization $\mc U_\varrho(\eta_j):\scr U_{W_j}\xrightarrow{\simeq}\Ubb\otimes_\Cbb\scr O_{W_i}$ defined by the local coordinate $\eta_j$) can be extended to the same $\scr O_{\Pbb^1\setminus\Sbf}$-module morphism $\wr\uppsi(\cdot,w_\blt):\scr U_{\Pbb^1\setminus\Sbf}\rightarrow\scr O_{\Pbb^1\setminus\Sbf}$. When $\Pbb^1$ is replaced by any compact Riemann surface, one can use the same definition to define conformal blocks. See Sec. \ref{lb28} for details.

We can see that in order  to define untwisted conformal blocks, we do not need to choose paths in pointed Riemann surfaces (with local coordinates). This is not true for untwisted conformal blocks: $Y(u,z)$ is now multivalued over $z$, so \eqref{eq100} can never be extended to the same morphism. 

The correct definition is as follows. Associate a $g_j$-twisted $\Ubb$-module $\mc W_j$ (whose vertex operation $Y^{g_j}$ is temporarily also denoted by $Y$) to the marked point $x_j$, where $g_1,\dots,g_N$ arise from an action $\Gamma\curvearrowright E$. Due to the positivity condition \eqref{eq101}, when $z$ is close to $\eta_j(\upgamma_j(0))$, we may uniquely determine $Y(u,z)w_j$ by the fact that $\arg z$ is close to $0$. Then we require that \eqref{eq100} converges absolutely when $z$ is near $\eta_j(\upgamma_j(0))$, that for different $j_1,j_2$, the expression \eqref{eq100} with $j=j_1$ can be analytically continued to \eqref{eq100} with $j=j_2$ along the path $\upgamma_{j_1}\upgamma_{j_2}^{-1}$, and  can furthermore be extended to a ``multivalued" morphism $\scr U_{\Pbb^1\setminus\Sbf}\rightarrow\scr O_{\Pbb^1\setminus\Sbf}$. Cf. Subsec. \ref{lb65}.

We remark that the above definition relies on the previously mentioned assumption that $\Gamma$ is generated by all $[\upalpha_\blt]$. If $\Pbb^1$ is replaced by a higher genus $X$, then this condition is never satisfied. In this case, we just need to add one more condition relying on the data $\Gamma\curvearrowright E$. We do not explain this condition in the Introduction, and refer the readers to Rem. \ref{lb64} for details.

\subsubsection*{Permutation coverings}

We now describe the permutation covering of $\fk P$ associated to the admissible elements $g_\blt$ (equivalently, associated to the action $\Gamma\curvearrowright E$). The permutation covering  $\varphi:C\rightarrow\Pbb^1$ is unbranched outside $\Sbf$, and is determined by the restriction $\varphi:C\setminus\varphi^{-1}(\Sbf)\rightarrow\Pbb^1\setminus\Sbf$. By algebraic topology, a (resp. finite) \textit{connected} covering of $\Pbb^1\setminus\Sbf$ is described by either of the following two equivalent objects: 
\begin{enumerate}[label=(\arabic*)]
\item A conjugacy class of (resp. cofinite) subgroups of the fundamental group $\Gamma$.
\item A transitive (i.e., single-orbit) action of $\Gamma$  on a (resp. finite) set.
\end{enumerate}
For our purpose, it is more convenient to use (2) to describe $\varphi:C\setminus\varphi^{-1}(\Sbf)\rightarrow\Pbb^1\setminus\Sbf$. Then each connected component $C^\Omega$ of $C\setminus\varphi^{-1}(\Sbf)$ corresponds to a $\Gamma$-orbit $\Omega$ in $E$. (So $\Omega$ is of the form $\Gamma e$ for some $e\in E$.) The precise description is the following elegant statement (cf. Thm. \ref{lb38}) :

\textit{There exists a $\Gamma$-covariant bijection $\Psi_{\upgamma_\blt(1)}:E\rightarrow\varphi^{-1}(\upgamma_\blt(1))$.} 

By ``\textbf{$\Gamma$-covariant}", we mean that for every $e\in E$ and every closed path in $\Pbb^1\setminus\Sbf$ from and to $\upgamma_\blt(1)$, the lift of $\upmu$ to $C\setminus\varphi^{-1}(\Sbf)$ ending at $\Psi_{\upgamma_\blt(1)}(e)$ must start from $\Psi_{\upgamma_\blt(1)}([\upmu]e)$. (See Def. \ref{lb66}.) Then any two branched coverings of $\Pbb^1$ with such $\Gamma$-covariant bijections are equivalent (Thm. \ref{lb29}).

The set of marked points of our permutation covering $C$ is just $\varphi^{-1}(\Sbf)$. The untwisted $\Vbb$-conformal blocks will be defined for $C$ and these marked points. In order to  associate the correct $\Vbb$-module to each marked point, it is important to \textit{label the marked points by certain orbits of $E$}. Note that $\varphi^{-1}(\Sbf)=\bigsqcup_{1\leq j\leq N}\varphi^{-1}(x_j)$. The the labeling of the elements of $\varphi^{-1}(x_j)$ is given by a bijection
\begin{align*}
\Upsilon_j:\{\bk{g_j}\text{-orbits in }E\}\longrightarrow \varphi^{-1}(x_j)
\end{align*}
described as follows. We abbreviate $\Upsilon_j$ as $\Upsilon$. 

First of all, for each path $\uplambda$ in $\Pbb^1\setminus\Sbf$ from a point $x$ to $\upgamma_\blt(1)$, we define a bijection
\begin{align*}
\Psi_\uplambda:E\rightarrow\varphi^{-1}(x)	
\end{align*}
sending each $e\in E$ to the initial point of the lift of $\uplambda$ to $C\setminus\varphi^{-1}(\Sbf)$ ending at $\Psi_{\upgamma_\blt(1)}(e)$. Recall that $W_j$ is an open disc centered at $x_j$ and contains $\upgamma_j(0)$. For each $\bk{g_j}$-orbit $\bk{g_j}e$, we can find a unique connected component $\wtd W_j$ of $\varphi^{-1}(W_j)$ whose intersection with $\varphi^{-1}(\upgamma_j(0))$ is exactly the set of points
\begin{align*}
\Psi_{\upgamma_j}\big(\bk{g_j}e\big):=\Big\{\Psi_{\upgamma_j}(g_j^ke):k\in\Zbb\Big\}	
\end{align*}
evenly located around the center $\wtd W_j\cap\varphi^{-1}(x_j)$  of $\wtd W_j$. (The set $\wtd W_j\cap\varphi^{-1}(x_j)$ indeed has only one element.) The size of $\Psi_{\upgamma_j}\big(\bk{g_j}e\big)$ equals the size $k=|\bk{g_j}e|$ of the orbit $\bk{g_j}e$ and also equals the branching index of $\varphi$ at $\wtd W_j\cap\varphi^{-1}(x_j)$. (See Prop. \ref{lb12}.) Then $\Upsilon(\bk{g_j}e)$  is defined to be the unique point of $\wtd W_j\cap\varphi^{-1}(x_j)$.

We see that the covering $\varphi:C\rightarrow\Pbb^1$ depends only $\Gamma\curvearrowright E$, but the labeling of the marked points of $C$ depends also on the paths $\upgamma_\blt$.

Finally, we choose local coordinates of $C$ at each marked point as follows. For each $\bk{g_j}$-orbit we fix a distinct point, called the \textbf{marked point} of that orbit. Let 
\begin{align*}
	E(g_j)=\{\text{marked points of $\bk{g_j}$-orbits}\}.
\end{align*}
(So $|E(g_j)|$ equals the number of $\bk{g_j}$-orbits in $E$.) Recall that $\eta_j$ is a local coordinate defined on $W_j$ sending $x_j$ to $0$. For each marked point $\Upsilon(\bk{g_j}\tipae)$ of $C$ (where $\tipae\in E(g_j)$) which is contained in a unique connected component $\wtd W_{j,\tipae}$ of $\varphi^{-1}(W_j)$, there are $k=|\bk{g_j}\tipae|$ different injective analytic functions on $\wtd W_{j,\tipae}$ whose $k$-th power equals $\eta_j\circ\varphi$. The local coordinate we choose at $\Upsilon(\bk{g_j}\tipae)$ and denote by $\wtd\eta_{j,\tipae}$ is the one of them that satisfies
\begin{align*}
\wtd\eta_{j,\tipae}\big(\Psi_{\upgamma_j}(\tipae)\big)>0,	
\end{align*}
which exists because of $\Psi_{\upgamma_j}(\tipae)\in\varphi^{-1}(\upgamma_j(0))$ and the positivity condition \eqref{eq101}.

The above branched covering $\varphi:C\rightarrow\Pbb^1$ (or just $C$), together with the labeled marked points and local coordinates, is denoted by $\fk X$ and called the \textbf{permutation covering} of  $\fk P$ (see \eqref{eq102}) associated to the action $\Gamma\curvearrowright E$ and the set $E(g_\blt)$ of marked points of $\bk{g_\blt}$-orbits.

\subsubsection*{The correspondence of twisted/untwisted conformal blocks}




For each $\bk{g_j}$-orbit $\bk{g_j}\tipae$ in $E$ (where $\tipae\in E(g_j)$), we choose a $\Vbb$-module $\Wbb_{j,\tipae}$. Then by \cite{BDM02}, the vector space $\mc W_j=\otimes_{\tipae\in E(g_j)}\Wbb_{j,\tipae}$ is naturally equipped with a $g_j$-twisted $\Ubb$-module structure (see also Subsec. \ref{lb23}). A direct sum of such modules is called a \textbf{$\otimes$-factorizable} $g_j$-twisted $\Ubb$-module. If $\Vbb$ is rational, then any $g_j$-twisted $\Ubb$-module is $\otimes$-factorizable, cf. \cite[Thm. 6.4]{BDM02}.

It should be reminded that this twisted module structure depends not only on the $\bk{g_j}$-orbits but also on their marked points. For example, when $E=\{1,\dots,n\}$ and $g_j$ is the cycle $g=(12\cdots n)$, then for each $u=v_1\otimes\cdots \otimes v_n\in\Vbb^{\otimes n}$, setting $\tipaomega_n=e^{-2\pi\im/n}$, the twisted vertex operation $Y^g(u,z)$ at $z=1$ (with $\arg z=0$) is expressed by the untwisted ones $Y(\square v_i,\tipaomega_n^{i-1})$ (for all $i$) where $\square$ is a suitable operator. This corresponds to the marked point $1$. If we choose another marked point $l$, then $Y^g(u,1)$ should be expressed by all  $Y(\square v_i,\tipaomega_n^{i-l})$. There is no reason to assume that $1$ is superior than any other marked point.

We associate each $\Vbb$-module $\Wbb_{j,\tipae}$ to the marked point $\Upsilon(\bk{g_j}\tipae)$ of $\fk X$, and associate the twisted $\Ubb$-module $\mc W_j$ to the marked point $x_j$ of $\fk P$. Then, our first main result is:
\begin{Mthm}(Cf. Thm. \ref{lb14})\label{lb69}
A linear functional
\begin{align*}
\uppsi:\bigotimes_{1\leq j\leq N}\mc W_j=\bigotimes_{1\leq j\leq N}\bigotimes_{\tipae\in E(g_j)}\Wbb_{j,\tipae}\rightarrow\Cbb
\end{align*}
is a conformal block associated to $\fk P$ and the associated twisted $\Ubb=\Vbb^{\otimes E}$-modules if and only if it is a conformal block associated to $\fk X$ and the associated untwisted $\Vbb$-modules.
\emph{}
\end{Mthm}
In particular, we have constructed an explicit isomorphism between the two spaces of conformal blocks. The proof of this theorem relies on the main results of \cite{Gui24b}.

\subsubsection*{Sewing and factorization}

We can relate not only  the permutation-twisted and untwisted conformal blocks, but also their sewing and factorization. If we have two positively pathed Riemann spheres with local coordinates $\fk P^a$ and $\fk P^b$, we can sew these two spheres by removing one disc from $\fk P^a$ around one of its marked point $x_0$, removing another one around $y_0$ from $\fk P^b$, and gluing the remaining part. Similarly, we can sew their permutation coverings $\fk X^a,\fk X^b$. Corresponding to this geometric sewing, we have the well-known sewing of conformal blocks. (See Subsec. \ref{lb44}.)

The product and the iterate of (twisted) intertwining operators can be viewed as sewing (twisted) conformal blocks associated to $\fk P^a$ and $\fk P^b$. These two types of sewing are equivalent and related by the \textbf{operator product expansion (OPE)} (i.e., associativity) of intertwining operators. Such relation defines the associativity isomorphisms of the crossed-braided tensor category $\Rep^G(\Ubb)$ of $G$-twisted $\Ubb$-modules (where $G$ is any subgroup of $\Perm(E)$). (Cf. \cite{Hua95,McR21}.) Our second main result is this:

\begin{Mthm}\label{lb70}
The following are true.
\begin{enumerate}
\item (Thm. \ref{lb46}) There is a suitable sewing $\fk X^{a\#b}$  of $\fk X^a$ and $\fk X^b$ which is isomorphic to the permutation covering of the sewing $\fk P^{a\#b}$ of $\fk P^a$ and $\fk P^b$. 
\item (Thm. \ref{lb48}) Assume $\Vbb$ is $C_2$-cofinite. If $\uppsi^a,\uppsi^b$ are permutation-twisted $\Ubb$-conformal blocks associated respectively to $\fk P^a,\fk P^b$ (equivalently, $\Vbb$-conformal blocks associated to $\fk X^a,\fk X^b$), then their sewing as permutation-twisted $\Ubb$-conformal blocks agree with that as $\Vbb$-conformal blocks, and the result of this sewing converges absolutely to a  $\Ubb$-conformal block associated to $\fk P^{a\#b}$ (equivalently, a $\Vbb$-conformal block associated to $\fk X^{a\#b}$).

\item (Thm. \ref{lb62}) If $\Vbb$ is CFT-type, $C_2$-cofinite, and rational, then any permutation-twisted $\Ubb$-conformal block associated to $\fk P^{a\#b}$ can be expressed in a unique way (understood in a suitable sense) as the sewing of those associated to $\fk P^a$ and $\fk P^b$.
\end{enumerate}
\end{Mthm}

The first part of this theorem says that sewing and taking permutation coverings are commuting procedures. The second part says that sewing  permutation-twisted $\Ubb$-conformal blocks amounts to sewing untwisted $\Vbb$-conformal blocks. The third part is the genus-$0$ sewing/factorization property for permutation-twisted conformal blocks. Its proof relies in particular on the untwisted factorization property \cite{DGT22}.

Thus, if we assume $\Vbb$ is CFT-type, $C_2$-cofinite, and rational, then by the above theorem, \textit{the OPE of permutation-twisted $\Ubb$-intertwining operators can be interpreted as an associativity of (possibly) higher genus untwisted $\Vbb$-conformal blocks}. (Here, the associativity means the equivalence of two different ways of factoring a possibly higher genus  untwisted $\Vbb$-conformal block into the sewing of two $\Vbb$-conformal blocks).  Therefore, we may use the higher genus data in the representation theory of $\Vbb$ to study $\Rep^G(\Ubb)$, and vice versa. See Sec. \ref{lb71} for a rigorous description.

\begin{figure}[h]
	\centering
	\includegraphics[width=0.7\linewidth]{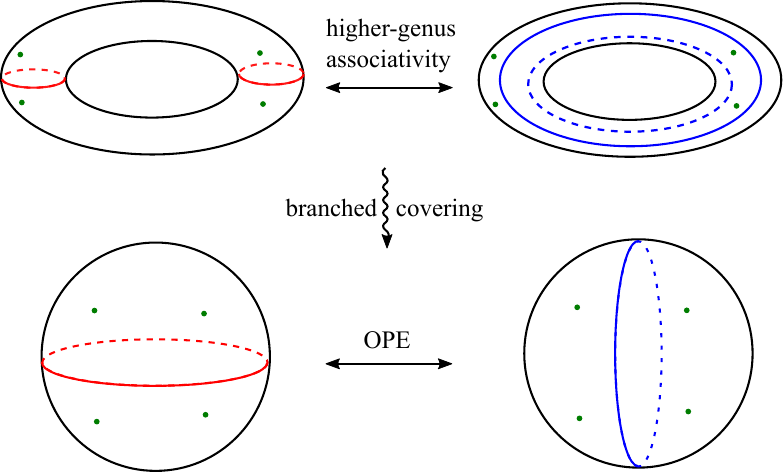}
	\caption{. The permutation-twisted OPE is equivalent to an associativity of  (possibly) higher genus conformal blocks via the permutation covering.}
	\label{fig2}
\end{figure}

An example for $E=\{1,2\}$ is illustrated in Figure \ref{fig2}. The bottom of this figure shows two ways of factoring a $4$-pointed (more precisely, positively $4$-pathed) sphere that correspond to the product and the iterate of $\Zbb_2$-permutation twisted intertwining operators. All the four marked points are associated with $(1,2)$-twisted $\Ubb=\Vbb^{\otimes 2}$-modules. The horizontal red circles and the vertical blue circles are associated with  untwisted $\Vbb^{\otimes 2}$-modules. All the four marked points are branched points with index $2$, and the permutation covering is genus $1$. In this example, the OPE of $\Zbb_2$-permutation twisted $\Vbb^{\otimes 2}$-intertwining operators amounts to an associativity of genus-$1$ untwisted $\Vbb$-conformal blocks. 

From Figure \ref{fig2}, the readers may notice the striking fact that \emph{this genus-$1$ associativity is the one that describes modular invariance}! Technically, the geometry of this associativity is \emph{topologically but not complex-analytically equivalent} to that of the modular invariance studied in \cite{Zhu96,Miy04,Hua05} and many other VOA articles: the genus-$1$ sewing in Figure \ref{fig2} is very different from sewing the boundary circles of \emph{standard} complex annuli. Nevertheless, there is good reason to believe that these two approaches will provide many of the same results. For instance, the $S$-matrices  defined originally by the invariance of genus-1 conformal blocks under the modular transform $\tau\mapsto -1/\tau$ should also be described by the genus-$1$ associativity in Figure \ref{fig2}, since it should be described by the braid/fusion matrices in the category $\Rep^{\Zbb_2}(\Vbb^{\otimes 2})$ of $\Zbb_2$-twisted $\Vbb^{\otimes 2}$-modules \cite{LX19}.

\subsection{Motivations from conformal nets}

\subsubsection*{Multi-interval Connes fusion and higher genus CFT}

The Haag-Kastler/conformal net theory (namely, the operator-algebraic approach to conformal field theory (CFT)) describes quantum fields on the Minkowski space $\Rbb^{1,1}$,  in contrast to the VOA approach where the field operators are defined on the Euclidean space $\Rbb^2=\Cbb$. The Euclidean CFT can be defined on higher genus Riemann surfaces by gluing genus-$0$ ones, since the Euclidean conformal symmetry is local.  The $\Rbb^2$-CFT (i.e., the genus-$0$ VOA theory) can be translated to the $\Rbb^{1,1}$-CFT by the Wick rotation.\footnote{Strictly speaking, the CFT one gets on $\Rbb^{1,1}$ is a Wightmann field theory. To get   Haag-Kastler CFT, one should take the smeared fields and consider certain von Neumann/$C^*$-algebras generated by them. See \cite{CKLW18}.} But it was not clear how to translate higher genus Euclidean CFT to Minkowskian CFT.

From this perspective, it is rather surprising that higher genus CFT data could be understood in the conformal net framework, which was indeed achieved by the multi-interval Connes fusion/Jones-Wassermann subfactors/Doplicher-Haag-Roberts superselection theory. (The single interval version describes genus-$0$ data; see e.g. \cite{FRS89,FRS92,Was98}.)  This observation was already made in \cite{Was94}. Later, it was proved in  \cite{KLM01,LX04} that if the index of multi-interval Jones-Wasserman subfactor (called the $\mu$-index) is finite, then the  category of semi-simple representations of the conformal net $\mc A$ is a modular tensor category; in particular, the $S$-matrices defined by the Hopf link is non-degenerate. Moreover, if the $\mu$-index is one, then $\mc A$ is holomorphic, i.e., it has only one irreducible representation: the vacuum representation.

If one notices the fact that a holomorphic VOA is described by the modular invariance of the one-point genus-$1$ conformal block defined by the $q$-trace of the vacuum vertex operation, one is surprised that holomorphic conformal nets are characterized by $\mu$ index $=1$, i.e., by the (multi-interval) Haag duality. Indeed, in full and boundary CFT, there are also equivalences of  \emph{modular invariance and "multi-region" Haag duality} (called strong Haag duality in \cite{Hen14}). (For full CFT, compare \cite[Prop. 6.6]{BKL15} with \cite[Thm. 5.7]{Kong08} and \cite[Thm. 3.4]{KR09a}. For boundary CFT, see \cite[Sec. 4]{KR09b} and the reference therein, and note that a question posed after Thm. 4.8 was solved in \cite[Prop. 4.18]{BKL15}.)

A main motivation of this article is to give a conceptual explanation of why multi-interval/multi-region Connes fusion is related to higher genus CFT. We now know the answer: it is well-known that multi-interval/multi-region Connes fusion is closely related to permutation orbifold CFT \cite{LX04,KLX05}, and from the main results of this article, we know that the latter has a higher genus CFT interpretation. (See \cite{BDH17} or \cite{LX19} for different explanations.)

\subsubsection*{Higher genus unitary VOA/conformal net correspondence}

A systematic study of the relationships between unitary VOAs and conformal nets was initiated by \cite{CKLW18}. In particular, for many unitary VOAs, the corresponding conformal nets were constructed by taking the smeared vertex operators. A different approach was given by \cite{Ten19a}. The methods in these two papers were generalized to relate the representation categories $\Rep(\Vbb),\Rep(\mc A_\Vbb)$ of a (strongly-rational) unitary VOA $\Vbb$ and the corresponding conformal net $\mc A_\Vbb$; see \cite{Ten19b,Ten19c} and \cite{Gui21,Gui20a}. 

In \cite{Gui21,Gui20a}, smeared intertwining operators are   crucial ingredients relating $\Rep(\Vbb)$ and $\Rep(\mc A_\Vbb)$: similar to the construction of smeared vertex operators in \cite{CKLW18}, we integrate $\mc Y(w,z)f(z)$ over the unit circle $I$, where $\mc Y$ is an intertwining operator of $\Vbb$, $w$ is a vector of a $\Vbb$-module, $I$ is an open non-dense interval of $\mbb S^1$, and $f$ is a smooth function supported in $I$. Since intertwining operators are $3$-pointed genus-$0$ conformal blocks, our method of comparing the tensor categories $\Rep(\Vbb)$ and $\Rep(\mc A_\Vbb)$ are essentially genus-$0$.

It is natural to think about relating higher genus aspects of unitary VOAs and conformal nets by integrating higher genus conformal blocks. However, an arbitrary integral will not give us the correct results. Note that in the genus-$0$ case, intertwining operators can only be integrated on an interval $I$ of $\mbb S^1$ ; otherwise the operators one gets will not have nice analytic properties. In arbitrary genus, one should integrate the conformal blocks on a correct real path inside the moduli space of pointed compact Riemann surfaces with (analytic) local coordinates.

Now, the twisted/untwisted correspondence proved in this article suggests how to find the correct paths and integrals: if we take $\mc Y$ to be a permutation-twisted $\Ubb=\Vbb^{\otimes E}$-intertwining operator, take $I\subset \mbb S^1$ to be an open interval, and consider the family of genus-$0$ permutation-twisted $\Ubb$-conformal blocks $z\in I\mapsto \mc Y(\cdot,z)$ as a family of untwisted (possibly higher-genus) $\Vbb$-conformal blocks, then the integral $\int_I \mc Y(w,z)f(z)dz$ is a correct integral of $\Vbb$-conformal blocks. The path in the moduli space, considered as a family of possibly higher-genus pointed compact Riemann surfaces with local coordinates, is the permutation covering of the family of $3$-pointed spheres $z\in I\mapsto (\Pbb^1;0,z,\infty)$ with local coordinates $\zeta,\zeta-z,\zeta^{-1}$ (where $\zeta$ is the standard coordinate of of $\Cbb$). 

This observation can be summarized by the slogan: \emph{the genus-$0$ permutation-twisted VOA/conformal net correspondence (as in \cite{CKLW18} or \cite{Ten19a}) is a higher-genus untwisted VOA/conformal net correspondence.}

This observation shows the  importance of relating the VOA and the conformal net tensor categories  for orbifold CFT, especially for permutation orbifolds, in the framework of \cite{CKLW18} or \cite{Ten19a}. For instance, the equivalence of tensor categories discussed in \cite{Gui21,Gui20a} should be generalized to orbifold CFT. We plan to explore this problem in future works.

\subsection{Future directions}

\subsubsection*{$\otimes$-nonfactorizable permutation twisted/untwisted correspondence.}

When $\Vbb$ is not rational, many twisted or untwisted  $\Ubb=\Vbb^{\otimes E}$-modules are not $\otimes$-factorizable. For such modules, we need to relate the $\Ubb$-conformal blocks with certain generalized $\Vbb$-conformal blocks (which are not yet studied in the literature). A generalized $\Vbb$-conformal block associated to a $N$-pointed compact Riemann surface with local coordinates is a linear functional $\uppsi:\mc M\rightarrow\Cbb$ satisfying a similar invariant condition as ordinary conformal blocks, where $\mc M$ is a $\Vbb^{\otimes N}$-module. We call such conformal blocks \textbf{$\otimes$-nonfactorizable}. 

We plan to study $\otimes$-nonfactorizable conformal blocks in the future. We also expect that they will play an essential role in the geometric theory of logarithmic CFT.

\subsubsection*{Higher genus twisted conformal blocks}

This article only discusses the $\Vbb$-conformal blocks that correspond to the genus-$0$ permutation-twisted $\Ubb$-conformal blocks. Higher genus twisted-conformal blocks are also defined (see Rem. \ref{lb64}). A natural problem is to generalize the results of this article to higher genus permutation-twisted conformal blocks. We expect that Thm. \ref{lb69} can be generalized in a straightforward way. The generalization of Thm. \ref{lb70} (the correspondence for sewing and factorization) would be more subtle and requires careful study.

\subsubsection*{Analogous results for conformal nets}

Let $\mc A$ be a completely-rational conformal net. Conformal blocks for finite-index (untwisted) representations of $\mc A$ were defined and studied in \cite{BDH17}. In particular, the factorization property was proved in that framework. Let $\mc H_1,\dots,\mc H_N,\mc H_{N+1}$ be finite-index permutation-twisted $\mc A^{\otimes E}$-representations. Construct an explicit isomorphism from $\Hom_{\mc A^{\otimes E}}(\mc H_1\boxtimes\cdots\boxtimes\mc H_N,\mc H_{N+1})$ to a space of $\mc A$-conformal blocks defined in \cite{BDH17}. Here, $\boxtimes$ is the Connes fusion product.

\subsubsection*{Geometry of genus-$1$  CFT: tori vs. elliptic curves}

Conventionally, the genus-$1$ properties of Euclidean  CFT were studied using the \emph{torus model}: taking the $q$-trace of vertex operators corresponds geometrically to sewing a standard annulus $\{a<|z|<b\}$ along the two boundaries. Our results in this article suggest that  genus-$1$ properties can also be studied using the \emph{elliptic curve model}: elliptic curves arise naturally not as the sewing of annuli, but as the branched coverings of $\Pbb^1$. This geometric picture corresponds to  the permutation-twisted/untwisted correspondence in the VOA world. 

It would be interesting to explore the genus-$1$ aspects of VOA via the elliptic curve model. For instance, when $\Vbb$ is strongly rational, using the genus-$0$ theory for permutation orbifolds, one may try to understand or give new proofs for the rigidity and the modularity of $\Rep(\Vbb)$ (originally proved in \cite{Hua08a,Hua08b} using the torus model), and to understand the full and boundary Euclidean CFT. (It should be sufficient to just consider $\Zbb_2$-permutations.) We also hope that this new framework will shed some light on proving the rigidity and understanding the modularity when $\Vbb$ is not rational.

\subsection{Outline}

In Chapter 1, we review the definition and the basic properties of (untwisted) conformal blocks. We also define a notion of conformal blocks for twisted VOA modules, and prove some elementary facts. For twisted conformal blocks, we focus mainly on the genus-$0$ case, although a definition for higher genus ones is also given (see Rem. \ref{lb64}).

We remark that a notion of conformal blocks for twisted modules already exists in the algebro-geometric approach to VOA  (cf. \cite{FS04}). Assuming genus-$0$ for simplicity, the main difference between that notion and ours is  that \cite{FS04} considered (single-valued) functions/ morphisms on a finite Galois branched covering of $\Pbb^1$, while we consider multi-valued functions/morphisms on $\Pbb^1$. These two definitions should be equivalent, although a translation between them may take quite a few pages. We choose our definition  because it is not difficult to relate to the twisted intertwining operators in the VOA literature \cite{Hua18,McR21}. Note also that the permutation coverings of $\Pbb^1$ are not necessarily Galois, thus they are in general not the same as the coverings considered in \cite{FS04}.

In Chapter 2, we describe the permutation coverings in details, and relate the permutation-twisted genus-$0$ conformal blocks and untwisted conformal blocks. In particular, we prove Main Thm. \ref{lb69}.

In Chapter 3, we show that the sewing procedure commutes with taking permutation coverings. We then show the permutation-twisted/untwisted correspondence for the sewing and factorization of conformal blocks. Namely, we prove Main Thm. \ref{lb70}.

In Chapter 4, we relate our definition of twisted conformal blocks with twisted intertwining operators. We then prove the OPE for permutation-twisted intertwining operators, and show that they correspond to the associativity of possibly higher genus untwisted conformal blocks (as indicated in Figure \ref{fig2}).

\subsection*{Acknowledgement}

I would like to thank Yi-Zhi Huang, Zhengwei Liu, Robert McRae, Jiayin Pan, Nicola Tarasca for helpful discussions.

\section{General results}

\subsection{The setting}

We set $\Nbb=\{0,1,2,\dots\}$ and $\Zbb_+=\{1,2,3,\dots\}$. Let $\Cbb^\times=\Cbb\setminus\{0\}$. \index{C@$\Cbb^\times$} For each $r>0$, we let $\mc D_r=\{z\in\Cbb:|z|<r\}$ and $\mc D_r^\times=\mc D_r\setminus\{0\}$. \index{Dr@$\mc D_r,\mc D_r^\times$} For any topological space $X$, we define the configuration space $\Conf^n(X)=\{(x_1,\dots,x_N)\in X^n:x_i\neq x_j~\forall 1\leq i<j<n\}$. \index{Conf@$\Conf^n(X)$, configuration space}

For a complex manifold $X$, $\scr O_X$ denotes the sheaf of (germ of) holomorphic functions on $X$, and $\scr O(X)=\scr O_X(X)$ is the space of (global) holomorphic functions on $X$.

Given a Riemann surface $C$ and a point $x\in C$, a \textbf{local coordinate} $\eta\in\scr O(U)$ of $C$ at $x$ means that $\eta$ is a holomorphic injective function on a neighborhood $U\subset C$ of $x$, and that $\eta(x)=0$.

For any ($\Cbb$-)vector space $W$, we define four spaces of formal series \index{z@$[[z]],[[z^{\pm 1}]],((z)),\{z\}$}
\begin{gather*}
W[[z]]=\bigg\{\sum_{n\in\mathbb N}w_nz^n:\text{each }w_n\in W\bigg\},\\
W[[z^{\pm 1}]]=\bigg\{\sum_{n\in\mathbb Z}w_nz^n:\text{each }w_n\in W\bigg\},\\
W((z))=\Big\{f(z):z^kf(z)\in W[[z]]\text{ for some }k\in\mbb Z \Big\},\\
W\{z\}=\bigg\{\sum_{n\in\mathbb \Cbb}w_nz^n:\text{each }w_n\in W\bigg\}.
\end{gather*}

If $X$ is a locally compact Hausdorff space, and if $\sum_{n\in\Cbb} f_n$ is a series of continuous functions $f_n$ on $X$, we say this series \textbf{converges absolutely and locally uniformly (a.l.u.)} on $X$, if for each compact subset $K\subset X$,
\begin{align}\label{eq10}
\sup_{x\in K}\sum_{n\in\Cbb}|f_n(x)|<+\infty.	
\end{align}

Let $\Vbb$ be a  VOA with vertex operator $Y(v,z)=\sum_{n\in\Zbb}Y(v)_nz^{-n-1}$, vacuum vector $\id$, and Virasoro operators $L_n=Y(\cbf)_{n+1}$. We assume throughout this article that a VOA $\Vbb$ has $L_0$-grading $\Vbb=\bigoplus_{n\in\Nbb}\Vbb(n)$ where each $\Vbb(n)$ is finite-dimensional.

As in \cite{Gui24b},  a $\Vbb$-module $\Wbb$ always means a \textbf{finitely-admissible $\Vbb$-module}. This means that $\Wbb$ is a weak $\Vbb$-module in the sense of \cite{DLM97}, that $\Wbb$ is equipped with a diagonalizable operator $\wtd L_0$ satisfying \index{L0@$\wtd L_0$}  
\begin{align}\label{eq34}
	[\wtd L_0,Y_\Wbb(v)_n]=Y_\Wbb(L_0 v)_n-(n+1)Y_\Wbb(v)_n,	
\end{align}
that the eigenvalues of $\wtd L_0$ are in $\Nbb$, and that each eigenspace $\Wbb(n)$ is finite-dimensional. Let \index{W@$\Wbb(n),\Wbb_{(n)}$}
\begin{align*}
	\Wbb=\bigoplus_{n\in\Nbb}\Wbb(n)	
\end{align*}
be the grading given by $\wtd L_0$. We set
\begin{align*}
	\Wbb^{\leq n}=\bigoplus_{0\leq k\leq n}	\Wbb{(k)}
\end{align*}
We choose  
\begin{align*}
\wtd L_0=L_0\qquad\text{on }\Vbb.	
\end{align*}
Any $\wtd L_0$-eigenvector of $\Wbb$  is called \textbf{$\wtd L_0$-homogeneous}.

If $\Wbb$ is finitely-admissible, then as $[\wtd L_0,L_0]=0$, each $\Wbb(n)$ is $L_0$-invariant, and hence is a sum of generalized eigenvectors of $L_0$. It follows that $\Wbb=\bigoplus_{\lambda\in\Cbb}\Wbb[\lambda]$ where $(L_0-\lambda)^k|_{\Wbb[n]}=0$ for some $k\in\Zbb_+$. Let $\Lss$ be the diagonalizable operator on $\Wbb$ which equals $\lambda$ when restricted to each $\Wbb[\lambda]$.  Clearly, for each $\alpha\in\Cbb$, $\Wbb_\alpha=\bigoplus_{n\in\Zbb}\Wbb[\alpha+n]$ is $\Vbb$-invariant. 

\begin{cv}
We assume that for each $\alpha\in\Cbb$, $\wtd L_0$ equals $\Lss$ plus a constant when acting on $\Wbb_\alpha$.
\end{cv}

We can define the \textbf{contragredient $\Vbb$-module} $\Wbb'$ of $\Wbb$ as in \cite{FHL93}. We choose $\wtd L_0$-grading to be
\begin{align*}
	\Wbb'=	\bigoplus_{n\in\Nbb}\Wbb'{(n)},\qquad \Wbb'{(n)}=\Wbb{(n)}^*.
\end{align*}
Therefore, if we let $\bk{\cdot,\cdot}$ be the pairing between $\Wbb$ and $\Wbb'$, then $\bk{\wtd L_0 w,w'}=\bk{w,\wtd L_0 w'}$ for each $w\in\Wbb,w'\in\Wbb'$. $\Wbb$ is the contragredient of $\Wbb'$. The vertex operation $Y_{\Wbb'}$ will be described in Example \ref{lb36}.

The vertex operation $Y_\Wbb$ of $\Wbb$ will be denoted by $Y$ when no confusion arises. Moreover, it can be extended $\Cbb((z))$-linearly to a map
\begin{align}
Y:\Big(\Vbb\otimes\Cbb((z))\Big)\otimes\Wbb\rightarrow\Wbb\otimes\Cbb((z)),\qquad v\otimes w\rightarrow Y(v,z)w,	\label{eq44}
\end{align}
and similarly
\begin{align}
Y:\Big(\Vbb\otimes\Cbb((z))dz\Big)\otimes\Wbb\rightarrow\Wbb\otimes\Cbb((z))dz,\qquad v\otimes w\rightarrow Y(v,z)wdz \label{eq45}	
\end{align}
so that we can take residue $\Res_{z=0}Y(v,z)wdz$. Let $(\Wbb')^*$ be the dual space of $\Wbb'$. Then for each $n\in\Nbb$ we can define a projection \index{Pn@$P_n,P_n^g$, the projection onto the $n$-eigenspace of $\wtd L_0$ resp. $\wtd L_0^g$}
\begin{align}
P_n:(\Wbb')^*=\prod_{n\in\Nbb}\Wbb(n)\rightarrow\Wbb(n).\label{eq46}
\end{align} 
If $v\in\Vbb$ and $w\in\Wbb$ are homogeneous with $\wtd L_0$-weights $\wt v,\wtd\wt w$ respectively, then  by \eqref{eq34}, 
\begin{align}
P_nY(v,z)w=Y(v)_{-n-1+\wt v+\wtd\wt w}	w\cdot z^{n-\wt v-\wtd\wt w}.\label{eq48}
\end{align}

A \textbf{family of transformations} over a complex manifold $X$ is by definition a holomorphic function $\rho$ on a neighborhood of $0\times X\subset \Cbb\times X$ sending each $(z,x)$ to $\rho_x(z)$, such that $\rho_x(0)=0$ and $(\partial_z\rho_x)(0)\neq 0$ for all $x\in X$. Let $c_0,c_1,\dots\in\scr O(X)$ be determined by
\begin{align*}
	\rho_x(z)=c_0(x)\cdot\exp\Big(\sum_{n>0}c_n(x)z^{n+1}\partial_z\Big)z
\end{align*}
on the level of $\scr O(X)[[z]]$. Then we necessarily have $c_0(x)=(\partial_z\rho_x)(0)$. On each $\Wbb$,   we set \index{U@$\mc U(\rho)$}
\begin{align*}
	\mc U(\rho)=(\partial_z\rho)(0)^{\wtd L_0}\cdot\exp\Big(\sum_{n>0}c_nL_n\Big)n
\end{align*}
as an automorphism of $\Wbb\otimes_\Cbb\scr O_X$, i.e., an ``$\End(\Wbb)$-valued holomorphic function" on $X$.

As an example, if $X$ is a Riemann surface, and if $\eta,\mu\in\scr O(X)$ are both locally injective (equivalenly, if $d\eta,d\mu$ vanish nowhere), we can define a family of transformations $\varrho(\eta|\mu)$ \index{zz@$\varrho(\eta\lvert\mu)$} on $X$ defined by
\begin{align}
\eta(y)-\eta(x)=\varrho(\eta|\mu)_x\Big(\mu(y)-\mu(x)\Big)	
\end{align}
for any $x\in X$ and any $y$ close to $x$.

For any two families of transformations $\rho_1,\rho_2$, if we let $\rho_1,\rho_2$ be their pointwise multiplication, then (cf. \cite[Sec. 4.2]{Hua97})
\begin{align*}
\mc U(\rho_1\rho_2)=\mc U(\rho_1)\mc U(\rho_2).	
\end{align*}

\begin{eg}\label{lb36}
Consider $\zeta^{-1}\in\scr O(\Cbb^\times)$ where $\zeta$ is the standard coordinate of $\Cbb$. Then the value of $\tipxgamma:=\varrho(\zeta^{-1}|\zeta)$ at each $z\in\Cbb^\times$ is \index{zz@$\tipxgamma$}
\begin{align*}
\tipxgamma_z(t)=\frac 1{~z+t~}-\frac 1{~z~}.	
\end{align*}
On any $\Vbb$-module $\Wbb$ we have (cf. for instance \cite[Ex. 1.4]{Gui24a})
\begin{align}
\mc U(\tipxgamma_z)\equiv\mc U(\tipxgamma)_z=e^{z L_1}(-z^{-2})^{\wtd L_0}.	
\end{align}
The vertex operator for the contragredient module $\Wbb'$ is determined by the fact that for each $v\in\Vbb,w\in\Wbb,w'\in\Wbb'$,
\begin{align*}
\bk{Y(v,z)w,w'}=\bk{w,Y(\mc U(\tipxgamma_z)v,z^{-1})w'}.
\end{align*}
\end{eg}

The \textbf{sheaf of VOA} $\scr V_C$ \index{VC@$\scr V_C,\scr V_C^{\leq n}$, sheaves of VOAs} associated to $\Vbb$ and a (non-necessarily compact) Riemann surface $C$ is an $\scr O$-module which associates to each connected open $U\subset C$ with locally injective $\eta\in\scr O(U)$ a trivialization (i.e., an $\scr O_U$-module isomorphism) \index{U@$\mc U_\varrho(\eta)$, trivialization for sheaves of VOAs}
\begin{align}
	\mc U_\varrho(\eta):\scr V_C|_U\xrightarrow{\simeq}\Vbb\otimes_\Cbb\scr O_U\label{eq5}	
\end{align}
such that for another similar $V\subset\mc C,\mu\in\scr O(V)$, the transition function is
\begin{align}
	\mc U_\varrho(\eta)\mc U_\varrho(\mu)^{-1}=\mc U(\varrho(\eta|\mu)):\Vbb\otimes_\Cbb\scr O_{U\cap V}\xrightarrow{\simeq}\Vbb\otimes_\Cbb\scr O_{U\cap V}. \label{eq65}
\end{align} 
For each $n\in\Nbb$, this transition function restricts to an automorphism of $\Vbb^{\leq n}\otimes_\Cbb\scr O_{U\cap V}$. Thus, we have a finite-rank locally free sheaf $\scr V_C^{\leq n}\subset\scr V_C$ having trivialization
\begin{align*}
\mc U_\varrho(\eta):\scr V_C^{\leq n}|_U\xrightarrow{\simeq}\Vbb^{\leq n}\otimes_\Cbb\scr O_U	
\end{align*}
being the restriction of \eqref{eq5}. $\scr V_C$ can be regarded as an infinite-rank (holomorphic) vector bundle which is the direct limit of the finite ones $\scr V_C^{\leq n}$.

The vacuum section $\id\in\scr V_C(C)$ \index{1@The vacuum section $\id$} denotes the one that is sent under any trivialization \eqref{eq5} to the constant vacuum vector $\id\in\Vbb\subset\Vbb\otimes_\Cbb\scr O(U)$.

\subsection{Conformal blocks for untwisted modules}\label{lb28}

\subsubsection{Conformal blocks and propagation}

By an \textbf{$N$-pointed compact Riemann surface with local coordinates}, we mean the following data
\begin{align}
\fk X=(C;x_1,\dots,x_N;\eta_1,\dots,\eta_N)	
\end{align}
where $C$ is a (non-necessarily connected) compact Riemann surface, $x_1,\dots,x_N$ are distinct points on $C$ such that each connected component of $C$ contains at least one of these points, and $\eta_1,\dots,\eta_N$ are local coordinates at $x_1,\dots,x_N$ respectively. We set \index{SX@$\SX$}
\begin{align*}
\SX=\{x_1,\dots,x_N\}.	
\end{align*}
For each $i$, choose a neighborhood $W_i$ of $x_i$ on which $\eta_i$ is defined, such that $W_i\cap\SX=\{x_i\}$.

Suppose $\Vbb$-modules $\Wbb_1,\dots,\Wbb_N$ are associated to the marked points $x_1,\dots,x_N$ respectively. Write
\begin{align*}
\Wbb_\blt=\Wbb_1\otimes\cdots\otimes \Wbb_N.	
\end{align*}
By $w\in\Wbb_\blt$, we mean a vector in $\Wbb_\blt$; by $w_\blt\in\Wbb_\blt$, we mean a vector of the form
\begin{align*}
w_\blt=w_1\otimes\cdots\otimes w_N	
\end{align*}
where each $w_i\in \Wbb_i$. Depending on the context, sometimes we also understand $\Wbb_\blt$ as the tuple $(\Wbb_1,\dots,\Wbb_N)$.

Let $\omega_C$ be the (holomorphic) cotangent bundle for $C$. Then $\scr V_C\otimes\omega_C(\star\SX)=\varinjlim_{n\in\Nbb}\scr V_C\otimes\omega_C(n\SX)$ is the sheaf of meromorphic sections of $\scr V_C\otimes\omega_C$ whose poles are only in $\SX$. The space of global sections is $H^0(C,\scr V_C\otimes\omega_C(\star\SX))$, which acts on $\Wbb_\blt$ in the following way. For each $v$ in this space, the restriction $v|_{W_i}$ can be regarded as a section of $\Vbb\otimes_\Cbb\omega_{W_i}$ through the trivialization $\mc U_\varrho(\eta_i):\scr V_C|_{W_i}\xrightarrow{\simeq}\Vbb\otimes\omega_{W_i}$. By taking series expansion with respect to the variable $\eta_i$ at $x_i$, $v|_{W_i}$ can furthermore be regarded as an element of $\Vbb\otimes\Cbb((z))dz$. Recall \eqref{eq45}, we define the action of $v$ on each $w_\blt\in\Wbb_\blt$ to be
\begin{align*}
v\cdot w_\blt=\sum_{i=1}^N w_1\otimes\cdots\otimes \Res_{z=0}Y(v,z)w_i \otimes\cdots\otimes w_N.	
\end{align*}

A \textbf{conformal block} $\upphi$ associated to $\fk X$ and $\Wbb_\blt$ is a linear functional
\begin{align*}
\upphi:\Wbb_\blt\rightarrow\Cbb	
\end{align*}
vanishing on $v\cdot w_\blt$ for all $w_\blt\in\Wbb_\blt$ and $v\in H^0(C,\scr V_C\otimes\omega_C(\star\SX))$.

We need the following propagation property of conformal blocks. Cf. \cite[Sec. 8]{Gui24b}. For each open $U_1,\dots,U_n\subset C$, we set \index{Conf@$\Conf(U_\blt\setminus\SX)$}
\begin{gather*}
\Conf(U_\blt\setminus\SX)=(U_1\times\cdots\times U_n)\cap\Conf^n(C\setminus\SX).	
\end{gather*}

\begin{thm}\label{lb7}
For each $n\in\Nbb$, we have the \textbf{$n$-propagation} of $\upphi$, which associates to each open $U_1,\dots,U_n\subset C$ a linear functional \index{zz@$\wr^n\upphi,\wr\uppsi$, propagation of conformal blocks}
\begin{gather*}
\wr^n\upphi:\scr V_C(U_1)\otimes\cdots\otimes\scr V_C(U_n)\otimes\Wbb_\blt\rightarrow\scr O(\Conf(U_\blt\setminus\SX))\\
v_1\otimes\cdots\otimes v_n\otimes w\mapsto \wr^n\upphi(v_1,\dots,v_n,w)
\end{gather*}
which is compatible with restriction to open subsets; namely, if $V_1\subset U_1,\dots,V_n\subset U_n$ are open, then
\begin{align*}
\wr^n\upphi(v_1|_{V_1},\dots,v_n|_{V_n},w)=\wr^n\upphi(v_1,\dots,v_n,w)\big|_{\Conf(V_\blt\setminus\SX)}.	
\end{align*}
$\wr^n\upphi$ intertwines the actions of $\scr O_C$, namely, for each $f_1\in\scr O(U_1),\dots,f_n\in\scr O(U_n)$,
\begin{align*}
\wr^n\upphi(f_1v_1,\dots,f_nv_n,w)=(f_1\circ\pr_1)\cdots (f_n\circ\pr_n)\wr^n\upphi(v_1,\dots,v_n,w)	
\end{align*}
where $\pr_i:C^n\rightarrow C$ is the projection onto the $i$-th component.

Choose any $w_\blt\in\Wbb_\blt$. For each $1\leq i\leq n$, choose an open subset  $U_i$  of $C$ equipped with an injective $\mu_i\in\scr O(U_i)$. Identify 
\begin{align*}
	\scr V_C\big|_{U_i}=\Vbb\otimes_\Cbb\scr O_{U_i}	\qquad\text{via }\mc U_\varrho(\mu_i).
\end{align*}
Choose $v_i\in\scr V_C(U_i)=\Vbb\otimes_\Cbb\scr O(U_i)$, and choose $(y_1,\dots,y_n)\in\Conf(U_\blt\setminus\SX)$. Then the following are true.
\begin{enumerate}[label=(\arabic*)]
\item If $U_1=W_j$ (where $1\leq j\leq N$) and contains only $y_1,x_j$ of all the points $x_\blt,y_\blt$, if $\mu_1=\eta_j$, and if $U_1$ contains the closed disc with center $x_j$ and radius $|\eta_j(y_1)|$ (under the coordinate $\eta_j$), then
\begin{align}
&\wr^n\upphi(v_1,v_2,\dots,v_n,w_\blt)\big|_{y_1,y_2,\dots,y_n}\nonumber\\
=&\wr^{n-1}\upphi\big(v_2,\dots,v_n,w_1\otimes\cdots\otimes Y(v_1,z)w_j\otimes\cdots\otimes w_N\big)\big|_{y_2,\dots,y_n}~\big|_{z=\eta_j(y_1)}\label{eq8}
\end{align}
where the series of $z$ on the right hand side converges absolutely, and $v_1$ is considered as an element of $\Vbb\otimes\Cbb((z))$ by taking Taylor series expansion with respect to the variable $\eta_j$ at $x_j$ (cf. \eqref{eq44}).
\item If $U_1=U_2$ and contains only $y_1,y_2$ of all the  points $x_\blt,y_\blt$, if $\mu_1=\mu_2$, and if $U_2$ contains the closed disc with center $y_2$ and radius $|\mu_2(y_1)-\mu_2(y_2)|$ (under the coordinate $\mu_2$), then
\begin{align}
&\wr^n\upphi(v_1,v_2,v_3,\dots,v_n,w_\blt)\big|_{y_1,y_2,\dots,y_n}\nonumber\\
=&\wr^{n-1}\upphi\big(Y(v_1,z)v_2,v_3,\dots,v_n,w_\blt\big)\big |_{y_2,\dots,y_n}~\big |_{z=\mu_2(y_1)-\mu_2(y_2)}\label{eq9}
\end{align}
where the series of $z$ on the right hand side converges absolutely, and $v_1$ is considered as an element of $\Vbb\otimes\Cbb((z))$ by taking Taylor series expansion with respect to the variable $\mu_2-\mu_2(y_2)$ at $y_2$.
\item  We have
\begin{align}
\wr^n\upphi(\id,v_2,v_3,\dots,v_n,w_\blt)=\wr^{n-1}\upphi(v_2,\dots,v_n,w_\blt).
\end{align}
\item For any permutation $\uppi$ of the set $\{1,2,\dots,n\}$, we have
\begin{align}
\wr^n\upphi(v_{\uppi(1)},\dots,v_{\uppi(n)},w_\blt)\big|_{y_{\uppi(1)},\dots,y_{\uppi(n)}}=\wr^n\upphi(v_1,\dots,v_n,w_\blt)\big|_{y_1,\dots,y_n}.
\end{align}
\end{enumerate}
\end{thm}

In the above theorem, $\wr^0\upphi$ is understood as $\upphi$.

The existence of $\wr\upphi$ that is compatible with restriction to open subsets, that intertwines the actions of $\scr O_C$ and that satisfies condition (1) of Theorem \ref{lb7}, can be regarded as an equivalent definition of a conformal block $\upphi:\Wbb_\blt\rightarrow\Cbb$; see \cite[10.1.2]{FB04}. We present the precise statement in a form that is closely related to the definition of twisted conformal blocks in Sec. \ref{lb22}.

\begin{thm}\label{lb24}
A linear functional $\upphi:\Wbb_\blt\rightarrow\Cbb$ is a conformal block associated to $\fk X$ and $\Wbb_\blt$ if and only if the following are satisfied
\begin{enumerate}
\item For each $v\in\Vbb$ and $1\leq j\leq N$, the series 
\begin{align}
&\upphi(w_1\otimes\cdots\otimes Y(v,z)w_j\otimes\cdots\otimes w_N)\nonumber\\
:=&\sum_{n\in \Nbb}\upphi(w_1\otimes\cdots\otimes P_nY(v,z)w_j\otimes\cdots\otimes w_N)\label{eq47}
\end{align}
of functions of $z$ converges a.l.u. on $\eta_j(W_j\setminus\{x_j\})$ in the sense of \eqref{eq10}.

\item There exists an operation $\wr\upphi$ which associates to each $w_\blt\in\Wbb_\blt$ an $\scr O_{C\setminus\SX}$-module morphism $\wr\upphi:\scr V_{C\setminus\SX}\rightarrow\scr O_{C\setminus\SX}$ satisfying the following conditions:

For each $1\leq j\leq N$, identify
\begin{align*}
	\scr V_{W_i}=\Vbb\otimes_\Cbb\scr O_{W_i}\qquad\text{via }\mc U_\varrho(\eta_j).	
\end{align*}
Choose any $v\in\Vbb\otimes_\Cbb\scr O(W_i\setminus\{x_i\})$. Then the function $\wr\upphi(v,w_\blt)\in\scr O(W_i\setminus\{x_i\})$ satisfies
\begin{align}
\wr\upphi(v,w_\blt)_x=\upphi\big(w_1\otimes\cdots\otimes Y\big(v(x),\eta_j(x)\big)w_j\otimes\cdots\otimes w_N\big)	\label{eq49}
\end{align}
for each $x\in U$, where the right hand side is understood as the limit of the series \eqref{eq47} by replacing $v$ by $v(x)$ and substituting $z=\eta_j(x)$.
\end{enumerate}
\end{thm}

Recall that the projection $P_n$ was defined in \eqref{eq46}. Also, note that by linearity, in the second condition it suffices to verify \eqref{eq49} when $v$ is a constant section, i.e., $v\in\Vbb\subset\Vbb\otimes_\Cbb\scr O(W_i\setminus\{x_i\})$.

\begin{proof}
``Only if": Assume $\upphi$ is a conformal block, and let $\wr\upphi$ be its propagation. Then Thm. \ref{lb7} implies the two conditions of this theorem  whenever $v,w_j$ are homogeneous: indeed, by \eqref{eq48}, the convergences of the right hand side of \eqref{eq47} and that of \eqref{eq8} are in the same sense. The general case follows from $\Cbb$- or $\scr O(W_i)$-linearity and the triangle inequality.

``If": Let $\wr\upphi$ be as in Condition 2 of this theorem. Choose $v\in H^0(C,\scr V_C\otimes\omega_C(\star\SX))\subset H^0(C\setminus\SX,\scr V_{C\setminus\SX}\otimes\omega_{C\setminus\SX})$. Consider $\wr\upphi(\cdot,w_\blt)\otimes 1:\scr V_{C\setminus\SX}\otimes\omega_{C\setminus\SX}\rightarrow \omega_{C\setminus\SX}$, also written as $\wr\upphi(\cdot,w_\blt)$ for simplicity. Then $\wr\upphi(v,w_\blt)$ is a global holomorphic $1$-form on $C\setminus\SX$. By Stokes theorem (or residue theorem), if we choose a small circle around each $x_j$, then the sum of the integrals of $\wr\upphi(v,w_\blt)$ along each circle is $0$. Substituting \eqref{eq49} into this equality, we see $\upphi(v\cdot w_\blt)=0$. This proves that $\upphi$ is a conformal block.
\end{proof}

\subsubsection{Sewing and propagation}\label{lb44}

Let $N,M\in\Zbb_+$. Let
\begin{align}
\fk X=(C;x_1,\dots,x_N;x_1',\dots,x_M';x_1'',\dots,x_M'')	
\end{align}
be an $(N+2M)$-pointed compact Riemann surface with local coordinates $\eta_i$ at $x_i$, $\xi_j\in\scr O(W_j')$ at $x_j'$, and $\varpi_j\in\scr O(W_j'')$ at $x_j''$. We assume that each connected component of $C$ contains at least one of $\SX=\{x_1,\dots,x_N\}$.  We assume that
\begin{align*}
\xi_j(W_j')=\mc D_{r_j},\qquad \varpi_j(W_j'')=\mc D_{\rho_j}	
\end{align*}
where
\begin{align*}
r_j\rho_j>1.	
\end{align*}
Moreover, we shall also assume that:
\begin{ass}\label{lb20}
The open sets $W_1',\dots,W_M',W_1'',\dots,W_M''$ are mutually disjoint and are disjoint from $\SX$. (Thus, $W_j'$ resp. $W_j''$ contains only $x_j'$ resp. $x_j''$ among the $N+2M$ marked points.)
\end{ass}
$W_j',W_j''$ are the discs to be sewn. So the above assumption says that the marked points left after sewing (namely, $x_1,\dots,x_N$) should be away from the discs to be sewn.

The \textbf{sewing of $\fk X$ along the pairs} $x_j',x_j''$ (for all $j$) is an $N$-pointed compact Riemann surface with local coordinates \index{SX@$\scr S\fk X,\scr SC$, sewing compact Riemann surfaces}
\begin{align}
\scr S\fk X=(\scr SC;x_1,\dots,x_N;\eta_1,\dots,\eta_N)	
\end{align}
constructed as follows. Remove the closed subsets
\begin{align*}
F_j'=\{x\in W_j':|\xi_j(x)|\leq 1/\rho_j\},\qquad 	F_j''=\{x\in W_j'':|\varpi_j(x)|\leq 1/r_j\}
\end{align*}
(for all $1\leq j\leq M$) from $C$. Glue the remaining part of $C$ by gluing all $x\in W_j'\setminus F_j'$ and $y\in W_j''\setminus F_j''$ satisfying
\begin{align*}
\xi_j(x)\varpi_j(y)=1.	
\end{align*}
This gives us a new compact Riemann surface $\scr SC$. Clearly, $C\setminus\bigcup_j (F_j'\cup F_j'')$ can be identified with an open subset of $\scr SC$. In particular, each $x_i$ is also a point of $\scr SC$, and $\eta_i$ can be regarded as a local coordinate of $\scr SC$ at $x_i$.

The above sewing procedure is unchanged if for each $j$ we choose $\lambda_j>0$ and replace $\xi_j,\varpi_j$ by $\lambda_j\xi_j,\lambda^{-1}\varpi_j$, or if for each $j$ we replace $W_j',W_j''$ by new neighborhoods $\wtd W_j'\ni x_j',\wtd W_j''\ni x_j''$ on which $\xi_j,\varpi_j$ are defined respectively, such that $\xi_j(\wtd W_j')=\mc D_{\wtd r_j},\varpi_j(\wtd W_j'')=\mc D_{\wtd\rho_j}$ and $\wtd r_j\wtd \rho_j>1$.

Corresponding to this geometric sewing, we can define, for every conformal block
\begin{align*}
\upphi:\Wbb_\blt\otimes\Mbb_\blt\otimes\Mbb_\blt'=\Wbb_1\otimes\cdots\otimes \Wbb_N\otimes\Mbb_1\otimes\Mbb_1'\otimes\cdots\otimes\Mbb_M\otimes\Mbb_M'\rightarrow\Cbb	
\end{align*}
associated to $\fk X$ and the chosen local coordinates, the \textbf{sewing} $\scr S\upphi$ as follows. For each $1\leq j\leq M,n\in\Nbb$, set
\begin{align}
\bowtie_{j,n}\in\Mbb_j(n)\otimes\Mbb_j'(n)	
\end{align}
which, considered as an element of $\End(\Mbb_j(n))$, is the identity operator. (Recall that $\dim\Mbb_j(n)<+\infty$ and $\Mbb_j'(n)$ is the dual space of $\Mbb_j(n)$.) Then $\scr S\upphi$ associates to each $w\in\Wbb_\blt$ an infinite sum
\begin{align}
\scr S\upphi(w)=\sum_{n_1,\dots,n_M\in\Nbb}\upphi(w\otimes \bowtie_{1,n_1}\otimes\cdots\otimes\bowtie_{M,n_M}).	\label{eq11}
\end{align}
\begin{df}\label{lb18}
We say the sewing $\scr S\upphi$ \textbf{converges $\bsb q$-absolutely}, if there exist $R_1,\dots,R_M>1$ such that for each $w\in\Wbb_\blt$, the infinite series of functions of $(q_1,\dots,q_M)\in\mc D_{R_\blt}:=\mc D_{R_1}\times\cdots\times\mc D_{R_M}$:
\begin{align*}
\sum_{n_1,\dots,n_M\in\Nbb}\upphi(w\otimes \bowtie_{1,n_1}\otimes\cdots\otimes\bowtie_{M,n_M})q_1^{n_1}\cdots q_M^{n_M}	
\end{align*}
converges a.l.u. on $\mc D_{R_\blt}$ in the sense of \eqref{eq10}; equivalently, there exist $R_1,\dots,R_M>1$ such that for each $w\in\Wbb_\blt$,
\begin{align*}
\sum_{n_1,\dots,n_M\in\Nbb}\big|\upphi(w\otimes \bowtie_{1,n_1}\otimes\cdots\otimes\bowtie_{M,n_M})\big|\cdot R_1^{n_1}\cdots R_M^{n_M}<+\infty.
\end{align*}
\end{df}

When $\scr S\upphi$ converges $q$-absolutely, it can be regarded as a linear functional on $\Wbb_\blt$ sending each $w\in\Wbb_\blt$ to the limit of \eqref{eq11}.

\begin{thm}[{\cite[Thm. 11.3]{Gui24a}}]\label{lb37}
If $\scr S\upphi$ converges $q$-absolutely, then $\scr S\upphi$ is a conformal block associated to $\scr S\fk X$ and $\Wbb_\blt$.
\end{thm}


\begin{thm}[{\cite[Thm. 13.1]{Gui24a}}]\label{lb39}
If $\Vbb$ is $C_2$-cofinite, and if all the modules $\Wbb_1,\dots,\Wbb_N,\Mbb_1,\dots,\Mbb_M$ are finitely-generated, then $\scr S\upphi$ converges $q$-absolutely.
\end{thm}

Let $\mc E$ be a complete list of irreducible $\Vbb$-modules, namely, every irreducible $\Vbb$-module is equivalent to a unique element of $\mc E$. If $\Vbb$ is CFT-type, $C_2$-cofinite, and rational, then $\mc E$ is a finite set. Let
\begin{align*}
\CB_{\fk X}(\Wbb_\blt\otimes\Mbb_\blt\otimes\Mbb'_\blt)\qquad\text{resp.}\qquad \CB_{\scr S\fk X}(\Wbb_\blt)	
\end{align*}
denote the space of conformal blocks associated to $\fk X$ and $\Wbb_\blt\otimes\Mbb_\blt\otimes\Mbb'_\blt$ (resp. $\scr S\fk X$ and $\Wbb_\blt$). Fix finitely-generated $\Wbb_1,\dots,\Wbb_N$. Then,  by Thm. \ref{lb37} and \ref{lb39}, we have a linear map
\begin{gather}
\begin{array}{c}
\scr S:\bigoplus_{\Mbb_1,\dots,\Mbb_M\in\mc E}\CB_{\fk X}(\Wbb_\blt\otimes\Mbb_\blt\otimes\Mbb'_\blt)\rightarrow\CB_{\scr S\fk X}(\Wbb_\blt)	\\[0.8ex]
\bigoplus_{\Mbb_1,\dots,\Mbb_M\in\mc E}\upphi_{\Mbb_\blt}\mapsto\sum _{\Mbb_1,\dots,\Mbb_M\in\mc E}\scr S\upphi_{\Mbb_\blt}
\end{array}	\label{eq71}
\end{gather}
where $\Mbb_\blt$ denotes also the tuple $(\Mbb_1,\dots,\Mbb_M)$.

\begin{thm}\label{lb40}
If $\Vbb$ is CFT-type, $C_2$-cofinite, and rational, and if $\Wbb_1,\dots,\Wbb_N$ are semi-simple $\Vbb$-modules, then the linear map $\scr S$ defined by \eqref{eq71} is bijective.
\end{thm}

\begin{proof}
When $M=1$, this follows from \cite[Thm. 12.1]{Gui24a}. (Note that the surjectivity of $\scr S$ follows from the remarkable factorization property proved by \cite{DGT22}.) Apply this result inductively, we can prove this theorem for a general $M$.
\end{proof}

The following theorem is crucial to the main result of this article. Cf. \cite[Thm. 9.1]{Gui24b}.

\begin{thm}\label{lb19}
Assume $\scr S\upphi$ converges $q$-absolutely. Let $U_1,\dots U_n\subset C$ be open and disjoint from $W_j',W_j''$ (for all $1\leq j\leq N$), which can also be viewed as open subsets of $\scr SC$. Then there exist $R_1,\dots,R_M>1$ such that for each $v_i\in\scr V_C(U_i)=\scr V_{\scr SC}(U_i)$ and $w\in\Wbb_\blt$, the following infinite series
\begin{align*}
\wtd{\mc S}\wr^n\upphi(v_1,\dots,v_n,w)=\sum_{n_1,\dots,n_M\in\Nbb}q_1^{n_1}\cdots q_M^{n_M}\cdot\wr^n\upphi(v_1,\dots,v_n,w\otimes\bowtie_{1,n_1}\otimes\cdots\otimes\bowtie_{M,n_M})	
\end{align*}
of holomorphic functions on $\mc D_{R_1}\times\cdots\times\mc D_{R_M}\times\Conf(U_\blt\setminus\SX)$ converges a.l.u. in the sense of \eqref{eq10}. Moreover, let $\scr S\wr^n\upphi$ be the limit of the above series at $q_1=\dots=q_M=1$.  Then
\begin{align}
\scr S\wr^n\upphi(v_1,\dots,v_n,w)=\wr^n\scr S\upphi(v_1,\dots,v_n,w).	
\end{align}
\end{thm}

\subsection{Conformal blocks for twisted modules}\label{lb22}

Let $\Ubb=\bigoplus_{n\in\Nbb}\Ubb(n)$ ($\dim\Ubb(n)<+\infty$) be a  VOA. An automorphism $g$ of $\Ubb$ is a linear map preserving the vacuum and the conformal vector of $\Ubb$, and satisfying $gY(u)_nv=Y(gu)_ngv$ for each $u,v\in\Ubb,n\in\Zbb$.

We let $G$ be a finite group of automorphisms of $\Ubb$.

\subsubsection{Twisted modules}\label{lb34}

For any $g\in G$ with order $|g|$, a (finitely-admissible)  \textbf{$g$-twisted $\Ubb$-module} is a vector space $\mc W$ together with a diagonalizable operator $\wtd L_0^g$, and an operation
\begin{gather*}
Y^g:\Ubb\otimes\mc W\rightarrow	\mc W[[z^{\pm 1/k}]]\\
u\otimes w\mapsto Y^g(u,z)w=\sum_{n\in \frac 1{|g|} \Zbb}Y^g(u)_nw\cdot z^{-n-1}
\end{gather*}
satisfying the following conditions:
\begin{enumerate}
\item $\mc W$ has $\wtd L_0^g$-grading \index{L0g@$\wtd L_0^g$} $\mc W=\bigoplus_{n\in\frac 1{|g|}\Nbb}\mc W(n)$, each eigenspace $\mc W(n)$ is finite-dimensional, and for any $u\in\Ubb$ we have
\begin{align}
[\wtd L_0^g,Y^g(u)_n]=Y^g(L_0u)_n-(n+1)Y^g(u)_n.\label{eq12}	
\end{align}
In particular, for each $w\in\mc W$ the lower truncation condition follows: $Y^g(u)_nw=0$ when $n$ is sufficiently small.
\item $Y^g(\id,z)=\id_{\mc W}$.
\item ($g$-equivariance) For each $u\in\Ubb$,
\begin{align}
Y^g(gu,z)=Y^g(u,e^{-2\im\pi}z):=\sum_{n\in \frac 1{|g|} \Zbb}Y^g(u)_nw\cdot e^{2(n+1)\im\pi}z^{-n-1}.\label{eq21}
\end{align}
\item (Analytic Jacobi identity)\footnote{In \cite[Def. 3.1]{Hua18}, this property is called the duality property.}  For each $u,v\in\Ubb,w\in\mc W,w'\in\mc W'$, and for each $z\neq \tipaz$ in $\Cbb^\times$ with chosen $\arg z,\arg \tipaz$, the following series of single-valued functions of $\log z,\log z,\log(z-\tipaz)$
\begin{gather}
\bk{Y^g(u,z)Y^g(v,\tipaz)w,w'}:=\sum_{n\in\frac 1{|g|}\Nbb}	\bk{Y^g(u,z)P_n^gY^g(v,\tipaz)w,w'}\label{eq13}\\
\bk{Y^g(v,\tipaz)Y^g(u,z)w,w'}:=\sum_{n\in\frac 1{|g|}\Nbb}	\bk{Y^g(v,\tipaz)P_n^gY^g(u,z)w,w'}\label{eq14}\\
\bk{Y^g(Y(u,z-\tipaz)v,\tipaz)w,w'}:=\sum_{n\in\Nbb}	\bk{Y^g(P_nY(u,z-\tipaz)v,\tipaz)w,w'}\label{eq15}
\end{gather}
(where $\tipaz$ is fixed) converge a.l.u.  on $|z|>|\tipaz|$, $|z|<|\tipaz|$, $|z-\tipaz|<|\tipaz|$ respectively (in the sense of \eqref{eq10}). Moreover, for any fixed $\tipaz\in\Cbb^\times$ with chosen argument $\arg \tipaz$, let $R_{\tipaz}$ be the ray with argument $\arg \tipaz$ from $0$ to $\infty$, but with $0,\tipaz,\infty$ removed. Any point on $R_{\tipaz}$ is assumed to have argument $\arg \tipaz$. Then the above three expressions, considered as functions of $z$ defined on $R_{\tipaz}$ satisfying the three mentioned inequalities respectively, can be analytically continued to the same holomorphic function on the open set
\begin{align*}
\Updelta_\tipaz=\Cbb\setminus\{\tipaz,-t\tipaz:t\geq 0\},
\end{align*}
which can furthermore be extended to a multivalued holomorphic function $f_{\tipaz}(z)$ on $\Cbb^\times\setminus\{\tipaz\}$ (i.e., a holomorphic function on the universal cover of $\Cbb^\times\setminus\{\tipaz\}$).
\end{enumerate}
In the above analytic Jacobi identity, let $\mc W'=\bigoplus_{n\in\frac 1{|g|}\Nbb}\mc W(n)^*$, then  $P_n^g$ \index{Pn@$P_n,P_n^g$, the projection onto the $n$-eigenspace of $\wtd L_0$ resp. $\wtd L_0^g$} is defined to be the projection
\begin{align}
	P_n^g:(\mc W')^*=\prod_{n\in\frac 1{|g|}\Nbb}\mc W(n)\rightarrow\mc W(n).\label{eq72}
\end{align}

The above vector space $\mc W'$ can be equipped with a $g^{-1}$-twisted $\Ubb$-module structure with $\wtd L_0$-grading  $\mc W'=\bigoplus_{n\in\frac 1{|g|}\Nbb}\mc W'(n)=\bigoplus_{n\in\frac 1{|g|}\Nbb}\mc W(n)^*$. Then for each $w\in\mc W,w'\in\mc W'$, 
\begin{align}
\bk{Y^{g^{-1}}(v,z)w',w}=&\bk{w',Y^g(\mc U(\tipxgamma_z)v,z^{-1})w}\nonumber\\
	=& \bk{w',Y^g(e^{zL_{-1}}(-z^{-2})^{L_0}v,z^{-1})w}.\label{eq64}
\end{align}
$\mc W'$ is called the \textbf{contragredient module} of $\mc W$. Cf. \cite[Prop. 3.3]{Hua18}. Using the easy fact that $\tipxgamma_{z^{-1}}\cdot \tipxgamma_z=1$ (and hence $\mc U(\tipxgamma_{z^{-1}})\cdot \mc U(\tipxgamma_z)=1$), one sees that $\mc W$ is the contragredient of $\mc W'$. By choosing $v$ to be the conformal vector $\cbf$ and setting
\begin{align*}
L_n=Y^g(\cbf)_{n+1}	
\end{align*}
when acting on  $\mc W$, we conclude
\begin{align}
\bk{L_n w,w'}=\bk{w,L_{-n}w'}.\label{eq85}	
\end{align}

The above analytic Jacobi identity is equivalent to its well-known algebraic form (cf. \cite[Sec. 10]{Gui24b}). It follows easily from that algebraic Jacobi identity that \eqref{eq12} holds if $\wtd L_0^g$ is replaced by $L_0$. Thus, $\wtd L_0^g-L_0$ commutes with the actions of vertex operators. In particular, if $\mc W$ is an irreducible twisted $\Ubb$-module, then $\wtd L_0^g-L_0$ is a constant.

\subsubsection{$\fk P$ is a positively $N$-pathed Riemann spheres with local coordinates}\label{lb30}

Let \index{00@Positively $N$-pathed Riemann sphere with local coordinates}
\begin{align}
\fk P=(\Pbb^1;x_1,\dots,x_N;\eta_1,\dots,\eta_N;\upgamma_1,\dots,\upgamma_N)=(\Pbb_1;x_\blt;\eta_\blt;\upgamma_\blt)\label{eq16}
\end{align}
where $(\Pbb^1;x_1,\dots,x_N;\eta_1,\dots,\eta_N)$ is an $N$-pointed Riemann sphere with local coordinates. Each $\eta_j$ is defined analytically on an open disc $W_j$ \index{W@$W_j$, the open disc centered at $x_j$} centered at $x_j$, i.e., $\eta_j(W_j)$ is an open disc centered at $0$. We assume $W_j\cap \{x_1,\dots,x_N\}=x_j$. We assume the (continuous) paths $\upgamma_1,\dots,\upgamma_N:[0,1]\rightarrow \Pbb^1$ have common end point in $\Pbb^1\setminus\{x_1,\dots,x_N\}$ (denoted by $\upgamma_\blt(1)$), \index{zz@$\upgamma_j$, a fixed path inside $\Pbb^1\setminus\Sbf$  starting from inside $W_j$} \index{zz@$\upgamma_\blt(1)$, the common end point of $\upgamma_1(1),\dots,\upgamma_N(1)$} and the initial point $\upgamma_j(0)$ for each $1\leq j\leq N$ satisfies $\upgamma_j(0)\in W_j$. (See Figure \ref{fig1}.)

\begin{cv}
	Unless otherwise stated, for any path $\uplambda$ in $\Pbb^1\setminus\Sbf$, its homotopy class $[\uplambda]$ denotes the class of all paths $\wtd{\uplambda}$ in $\Pbb^1\setminus\Sbf$ having the same initial and end points as $\uplambda$, and is homotopic (assuming the initial and end points are always fixed) in $\Pbb^1\setminus\Sbf$ to $\uplambda$. \index{zz@$[\uplambda]$} 
\end{cv} 

For each $x\in\Pbb^1\setminus\Sbf$, let \index{zz@$\Lambda_x$, the set of paths inside $\Pbb^1\setminus\Sbf$ from $x$ to $\upgamma_\blt(1)$}
\begin{align*}
	\Lambda_x=\big\{\text{continuous maps }\uplambda:[0,1]\rightarrow\Pbb^1\setminus\Sbf,\uplambda(0)=x,\uplambda(1)=\upgamma_\blt(1)\big\}.	
\end{align*}
Namely, $\Lambda_x$ is the set of all paths in $\Pbb^1\setminus\Sbf$ going from $x$ to $\upgamma_\blt(1)$. 

For each $j$, we let \index{zz@$\upepsilon_j$, the anticlockwise circle around $x$ from and to $\upgamma_j(0)$} 
\begin{align*}
	\upepsilon_j:[0,1]\rightarrow W_j\setminus\{x_j\}	
\end{align*}
be the anticlockwise circle (defined using $\eta_j$) centered at $x_j$ whose initial and end point is $\upgamma_\blt(0)$. We write
\begin{align*}
	\Gamma:=\pi_1(\Pbb^1\setminus \Sbf,\upgamma_\blt(1))	
\end{align*}
(the fundamental group of $\Pbb^1\setminus\Sbf$ with basepoint $\upgamma_\blt(1)$). \index{zz@$\Gamma=\pi_1(\Pbb^1\setminus \Sbf,\upgamma_\blt(1))$} Equivalently, $\Gamma=\{[\upmu]:\upmu\in\Lambda_{\upgamma_\blt(1)}\}$. Then $\Gamma$ is isomorphic to the free group $\mbb F_{N-1}$. Set \index{zz@$\upalpha_j=\upgamma_j^{-1}\upepsilon_j\upgamma_j$}
\begin{align}
	\upalpha_j:=\upgamma_j^{-1}\upepsilon_j\upgamma_j.	\label{eq1}
\end{align}
Then the homotopy class $[\upalpha_j]$ (in $\Pbb^1\setminus\Sbf$) belongs to $\Gamma$. We assume
\begin{align}
\Gamma=\bk{[\upalpha_1],\dots,[\upalpha_N]},\label{eq87}	
\end{align}
namely, these $N$ elements generate $\Gamma$.

Such data $\fk P$ is called an \textbf{$N$-pathed Riemann sphere with local coordinates}. If, moreover, for each $j$ we have
\begin{align}
	\begin{array}{c}
		\eta_j\circ\upgamma_j(0)\in (0,+\infty)\\[0.8ex]
		\arg \eta_j\circ\upgamma_j(0)=0
	\end{array}	\label{eq22}
\end{align} 
We say $\fk P$ is \textbf{positively $N$-pathed}. We set \index{S@$\Sbf=\{x_1,\dots,x_N\}$ in $\Pbb^1$}
\begin{align}
	\Sbf=\{x_1,\dots,x_N\}\qquad\subset\Pbb^1.\label{eq20}
\end{align}

\index{00@Positively $N$-pathed Riemann sphere with local coordinates}

\subsubsection{Conformal blocks}\label{lb65}

In this subsection, we define genus-$0$ twisted conformal blocks. For a comparison of our definition with that in algebraic geometry (cf. \cite{FS04}), see Introduction-Outline.

We let \index{P@$\mc P$, the universal cover of $\Pbb^1\setminus\Sbf$}
\begin{align*}
\mc P=\text{the universal cover of }\Pbb^1\setminus\Sbf.	
\end{align*}
Let $\scr U_C$ \index{UC@$\scr U_C$, the sheaf of VOA for $\Ubb$ and $C$} be the sheaf of VOA for $\Ubb$ associated to any Riemann surface $C$. Then $\scr U_{\mc P}$ can be identified naturally with the pullback of $\scr U_{\Pbb^1\setminus\Sbf}$ along the covering map $\mc P\rightarrow\Pbb^1\setminus\Sbf$.

\begin{lm}\label{lb17}
There is a one-to-one correspondence between: 
\begin{enumerate}
\item An $\scr O_{\mc P}$-module morphism $\text{\textOlyoghlig}:\scr U_{\mc P}\rightarrow\scr O_{\mc P}$.
\item An operation $\uppsi$ which associates to each simply-connected open subset $U\subset\Pbb^1\setminus\Sbf$, each path $\uplambda\subset\Pbb^1\setminus\Sbf$ from a point of $U$ to $\upgamma_\blt(1)$, and each section $v\in\scr U_{\Pbb^1}(U)$, an element $\uppsi(\uplambda,v)\in\scr O(U)$   satisfying the following properties:
\begin{enumerate}[label=(\alph*)]
	\item If $V\subset U$ is open, simply-connected, and contains $\uplambda(0)$, then $\uppsi(\uplambda,v|_V)=\uppsi(\uplambda,v)|_V$.
	\item If $f\in\scr O(U)$ then $\uppsi(\uplambda,fv)=f\uppsi(\uplambda,v)$.
	\item If $\uplambda'$ is another path in $\Pbb^1\setminus\Sbf$ with the same initial and end points as $\uplambda$, and if $[\uplambda]=[\uplambda']$, then $\uppsi(\uplambda,v)=\uppsi(\uplambda',v)$. Therefore, we may write $\uppsi(\uplambda,v)$ as $\uppsi([\uplambda],v)$.
	\item If $l$ is a path in $U$ ending at the initial point of $\uplambda$, then $\uppsi(\uplambda,v)=\uppsi(l\uplambda,v)$.
\end{enumerate}
\end{enumerate}
\end{lm}

We call any $\uppsi$ in part 2 a \textbf{multivalued ($\scr O_{\Pbb^1\setminus\Sbf}$-module) morphism} $\scr U_{\Pbb^1\setminus\Sbf}\rightarrow\scr O_{\Pbb^1\setminus\Sbf}$.

\begin{proof}
Fix a lift $p\in\mc P$ of $\upgamma_\blt(1)$. If $\text{\textOlyoghlig}$ is as in part 1, we lift $\uplambda$ to a path $\wtd\uplambda$ in $\mc P$ ending at $p$, and lift $U$ to a unique simply-connected $\wtd U_\uplambda$ containing the initial point of the $\wtd\uplambda$. Identify $\wtd U_\uplambda$ with $U$ via the covering map. Then $\uppsi(\uplambda,v):=\text{\textOlyoghlig}(v)$ defines $\uppsi$.

Conversely, choose $\uppsi$ as in part 2. For each open simply-connected $U\subset\Pbb^1\setminus\Sbf$ and any lift $\wtd U$, we choose a path $\uplambda\subset\Pbb^1\setminus\Sbf$ from a point of $U$ to $\upgamma_\blt(1)$ such that $\wtd U=\wtd U_\uplambda$.  Identify $\wtd U$ with $U$ via the covering map, and set $\text{\textOlyoghlig}(v)=\uppsi(\uplambda,v)$ for each $v\in\scr U_{\wtd U}(\wtd U)$. This defines $\text{\textOlyoghlig}$ on any simply connected open set, and can be easily extended onto all open subsets of $\mc P$.
\end{proof}

Now, for each $j$, choose $g_j\in G$, and choose a $g_j$-twisted $\Ubb$-module $\mc W_j$ associated to the marked point $x_j$.  Let $\mc W_\blt=\mc W_1\otimes\cdots\otimes\mc W_N$. We are ready to define genus $0$ conformal blocks for twisted modules.

\begin{df}\label{lb8}
A \textbf{conformal block} associated to $\fk P$ and $\mc W_\blt$ is a linear functional $\uppsi:\mc W_\blt\rightarrow\Cbb$ satisfying the following conditions:
\begin{enumerate}
\item For each $u\in\Ubb$ and each $1\leq j\leq N$, the series
\begin{align}
&\uppsi(w_1\otimes\cdots\otimes Y^{g_j}(u,z)w_j\otimes\cdots\otimes w_N)\nonumber\\
:=&\sum_{n\in \frac 1{|g_j|}\Nbb}\uppsi(w_1\otimes\cdots\otimes P_n^{g_j}Y^{g_j}(u,z)w_j\otimes\cdots\otimes w_N)\label{eq17}
\end{align}
of single-valued functions of $\log z$ converges a.l.u. on $\exp^{-1}\big(\eta_j(W_j\setminus\{x_j\})\big)$  in the sense of \eqref{eq10}.
\item There exists $\wr\uppsi$ associating to each $w_\blt\in\mc W_\blt$ a multivalued $\scr O_{\Pbb^1\setminus\Sbf}$-module morphism \index{zz@$\wr^n\upphi,\wr\uppsi$, propagation of conformal blocks}
\begin{align*}
\wr\uppsi(\cdot,\cdot,w_\blt):\scr U_{\Pbb^1\setminus\Sbf}\rightarrow\scr O_{\Pbb^1\setminus\Sbf}	
\end{align*}
satisfying the following condition: 

For each $1\leq j\leq N$,  identify 
\begin{align}
\scr U_{W_j}=\Ubb\otimes_\Cbb\scr O_{W_j}  \qquad\text{via }\mc U_\varrho(\eta_j).\label{eq38}
\end{align}
Choose   any open simply-connected subset $U\subset W_j\setminus\{x_j\}$ containing $\upgamma_j(0)$ and equipped with a continuous $\arg$ function on $\eta_j(U)$ whose value at $\eta_j\circ\upgamma_j(0)$ is $0$. Choose any $u\in\Ubb\otimes_\Cbb\scr O(U)$.  Then the function $\wr\uppsi(\upgamma_j,u,w_\blt)\in\scr O(U)$ satisfies   
\begin{align}
\wr\uppsi(\upgamma_j,u,w_\blt)_x=\uppsi\big(w_1\otimes\cdots\otimes Y^{g_j}\big(u(x),\eta_j(x)\big)w_j\otimes\cdots\otimes w_N\big)\label{eq18}	
\end{align}
for each $x\in U$, where the right hand side is understood as the limit of the series \eqref{eq17} by replacing $u$ by $u(x)$, substituting $z=\eta_j(x)$, and defining $\arg\eta_j(x)$ using the $\arg$ function on $\eta_j(U)$.
\end{enumerate}
\end{df}

Note that in \eqref{eq18}, we understand $u$ as an $\Ubb$-valued function whose value at each $x\in U$ is $u(x)\in\Ubb$.

We give some comments on this definition.


\begin{rem}
It suffices to check \eqref{eq18} for any constant section $u\in\Ubb\simeq\Ubb\otimes_\Cbb\id\subset\Ubb\otimes_\Cbb\scr O(U)$.
\end{rem}

\begin{rem}\label{lb9}
If for some open simply-connected $U\subset W_j$, the relation \eqref{eq18} holds for any $x\in U$, then due to the uniqueness of analytic continuation, for every open simply-connected $U\subset W_j$ this is also true.

Moreover, given one simply-connected open $U\subset W_j$, if $I$ is subset of $U$ with at least one accumulation point in $U$, then by complex analysis, \eqref{eq18} holds for all $x\in U$ if it holds for all $x\in I$.
\end{rem}

\begin{rem}\label{lb10}
Suppose that  $u$ is $\wtd L_0=L_0$-homogeneous with weight $\wt u$, and each $w_j$ is $\wtd L_0$-homogeneous with weight $\wtd\wt w_j$. Then by \eqref{eq12},
\begin{align*}
P_n^{g_j}Y^{g_j}(u,z)w_j=Y^{g_j}(u)_{-n-1+\wt u+\wtd\wt w_j}w_j\cdot	z^{n-\wt u-\wtd\wt w_j}.
\end{align*}
It follows that for homogeneous vectors, the following statements are equivalent:
\begin{enumerate}[label=(\arabic*)]
\item \eqref{eq17} converges a.l.u. as a function of $\log z$ on $\exp^{-1}\big(\eta_j(W_j\setminus\{x_j\})\big)$.
\item The laurent series
\begin{align}
\sum_{n\in \frac 1{|g_j|}\Zbb}\uppsi(w_1\otimes\cdots\otimes Y^{g_j}(u)_{n}w_j\otimes\cdots\otimes w_N)z^{-n-1}	\label{eq19}
\end{align}
of $z^{1/|g_j|}$ (which is a power series when multiplied by a power of $z^{1/|g_j|}$) converges absolutely for any  $z^{1/|g_j|}$ on the punctured disc $\mc D_{r_j^{1/|g_j|}}^\times$ (if $\eta_j(W_j)$ has radius $r_j$).
\item For each $z\in\eta_j(W_j\setminus\{x_j\})$ and every argument $\arg z$, \eqref{eq17} converges absolutely.
\item For each $z\in\eta_j(W_j\setminus\{x_j\})$ with one argument $\arg z$, \eqref{eq17} converges absolutely.
\end{enumerate}
Moreover, by linearity and triangle inequality, any of these four statements holds for all vectors provided that it holds for homogeneous vectors.
\end{rem}

\begin{rem}\label{lb21}
It follows from the previous two remarks that the definition for $\uppsi:\mc W_\blt\rightarrow\Cbb$ to be a conformal block is independent of (the sizes of) the open discs $W_j$ under $\eta_j$. 

Indeed, suppose we can verify the two conditions for  $W_j$ with radius $r_j$. Let $\wht W_j$ be a larger one with radius $R_j$ centered at $x_j$ on which $\eta_j$ is still defined.  Then condition 2 for $W_j$ implies that \eqref{eq19}, which converges to a holomorphic function of $z^{1/|g_j|}$ on $\mc D_{r_j^{1/|g_j|}}^\times$, can be analytically continued (by $\wr\upphi(\upgamma_j,u,w)$) to one on $\mc D_{R_j^{1/|g_j|}}^\times$. Thus \eqref{eq19}, whose coefficients are given by those of the series expansion of $\wr\upphi(\upgamma_j,u,w)$,  converges absolutely on this larger punctured disc, i.e., condition 1 holds on $\wht W_j$. By Rem. \ref{lb9}, condition 2 also holds on $\wht W_j$.
\end{rem}

\begin{rem}\label{lb2}
Assume a positively $N$-pathed $\fk P'$ is $\Pbb^1$ with the same marked points $x_\blt$ and local coordinates $\eta_\blt$, but with different paths $\upgamma_1',\dots,\upgamma_N'$ (in $\Pbb^1\setminus\Sbf$ ending at a common point $\upgamma'_\blt(1)$). We say that $\upgamma_\blt$ and $\upgamma_\blt'$ are \textbf{equivalent (positive) paths} (or that $\fk P$ a and $\fk P'$ are equivalent) \index{00@Equivalent (positive) paths for pointed Riemann spheres with local coordiates} if there exists a path $\upsigma$ in $\Pbb^1\setminus\Sbf$ with initial point $\upgamma_\blt(1)$, and for each $1\leq j\leq N$ there is a path $l_j$ in $W_j$ from $\upgamma'(0)$ to $\upgamma_j(0)$ satisfying 
\begin{gather}
	\mathrm{Range}(\eta_j\circ l_j)\subset(0,+\infty),\label{eq53}\\
	[\upgamma_j']=[l_j\upgamma_j\upsigma].	\label{eq66}
\end{gather}

If $\upgamma_\blt$ and $\upgamma_\blt'$ are equivalent, then it is clear that a conformal block associated to $\fk P$ and $\mc W_\blt$ is also one associated to $\fk P'$ and $\mc W_\blt$.
\end{rem}

\begin{lm}
If $\uppsi$ is a conformal block, then the $\wr\uppsi$ satisfying condition 2 of Def. \ref{lb8} are unique. 
\end{lm}

\begin{proof}
For two such operations $\wr_1\uppsi,\wr_2\uppsi$ (considered as morphisms $\scr U_{\mc P}\rightarrow\scr O_{\mc P}$), let $\Omega$ be the open set of all $y\in\mc P$ on a neighborhood of which these two agree (for all $w_\blt$). By \eqref{eq18}, $\Omega$ intersects a lift of $W_j$ in $\mc P$. If $U$ is a simply-connected open subset of $\Pbb^1\setminus\Sbf$ such that there is an injective $\eta\in\scr O(V)$, and if $\wtd U$ is a lift of $U$ in $\mc P$ (namely, $\wtd U$ is represented by $U$ and a path $\upgamma$ from inside $U$ to $\upgamma_\blt(1)$), then $\scr U_{\wtd U}=\scr U_U\simeq\Ubb\otimes_\Cbb\scr O_U$. Hence, by complex analysis, if $\wtd U$ intersects $\Omega$,  then $\wtd U\subset\Omega$. So $\Omega$ is closed, and hence must be $\mc P$.
\end{proof}


The following results are not expected to hold for higher genus conformal blocks.

\begin{lm}\label{lb53}
There exists a set $\fk U$ of elements of $\scr U_{\Pbb^1}(\Pbb^1\setminus\Sbf)$ that generates freely the $\scr O_{\Pbb^1\setminus\Sbf}$-module $\scr U_{\Pbb^1\setminus\Sbf}$, namely, any section of $\scr U_{\Pbb^1\setminus\Sbf}$ on an open set $U\subset\Pbb^1\setminus\Sbf$ can be written in a unique way as a finite sum $f_1v_1+f_2v_2+\cdots$ where $v_1,v_2,\dots\in\fk U$ and $f_1,f_2,\dots\in\scr O(U)$.
\end{lm}

\begin{proof}
Assume with out loss of generality that $x_N\in\Sbf$ is $\infty$. Let $\fk U_0$ be a basis of $\Ubb$. Let $\zeta$ be the standard coordinate of $\Cbb$. Then one can set $\fk U=\{\mc U_\varrho(\zeta)^{-1}u:u\in\fk U_0\}$.
\end{proof}

\begin{pp}\label{lb11}
Let $\fk U$ be a set of elements of $\scr U_{\Pbb^1}(\Pbb^1\setminus\Sbf)$ generating freely  $\scr U_{\Pbb^1\setminus\Sbf}$. Then a linear functional $\uppsi:\mc W_\blt\rightarrow\Cbb$ is a conformal block if and only if the following are true:
\begin{enumerate}
\item Condition 1 of Def. \ref{lb8} is satisfied.
\item There is $\wr\uppsi$ associating to each $w_\blt\in\mc W_\blt$ and $u\in\fk U$ a multivalued holomorphic function $\wr\uppsi(\cdot,u,w_\blt)$ on $\Pbb^1\setminus\Sbf$ (single-valued on any simply-connected subset $U$ if a path $\uplambda$ from inside $U$ to $\upgamma_\blt(1)$ is specified)  satisfying the following condition: 

For each $1\leq j\leq N$,  identify $\scr U_{W_i}$ with $\Ubb\otimes_\Cbb\scr O_{W_i}$ through $\mc U_\varrho(\eta_i)$.  Choose   any open simply-connected subset $U\subset W_j\setminus\{x_j\}$ containing $\upgamma_j(0)$ and equipped with a continuous $\arg$ function on $\eta_j(U)$ whose value at $\eta_j\circ\upgamma_j(0)$ is $0$. Then the function $\wr\uppsi(\upgamma_j,u,w_\blt)\in\scr O(U)$ satisfies \eqref{eq18} for all $x\in U$.
\end{enumerate}

\end{pp}

\begin{proof}
The ``only if" part is obvious. Now assume $\wr\uppsi$ satisfies the two conditions of the present proposition. Then  $\wr\uppsi(\cdot,\cdot,w_\blt)$ can be extended uniquely to a multivalued homomorphism $\scr U_{\Pbb^1\setminus\Sbf}\rightarrow\scr O_{\Pbb^1\setminus\Sbf}$, which satisfies \eqref{eq18} for all $u\in\fk U$, and hence all $u\in\scr U_U(U)$.
\end{proof}

\section{Relating untwisted and permutation-twisted conformal blocks}

\subsection{Permutation branched coverings of $\Pbb^1$}

Let $C$ be a (non-necessarily connected) compact Riemann surface. A \textbf{branched covering} $\varphi:C\rightarrow\Pbb^1$ is by definition a holomorphic map which is non-constant on each connected component of $C$. Then $\varphi$ is surjective on each component of $C$ since the image of $\varphi$ is both open and compact. By complex analysis, $y\in C$ has a neighborhood $V$ such that $\varphi$ on $V$ is equivalent to the holomorphic map $z\mapsto z^n$. $n$ is called the \textbf{branching index} of $\varphi$ at $y$, which is $0$ precisely when $d\varphi$ is not $0$ at $y$. The (necessarily finite) set of  \textbf{branch points} is $\Sigma:=\{x\in C:d\varphi=0\}$, which is also the set of points with non-zero branching indexes. Let $\Delta=\varphi(\Sigma)$ be the \textbf{critical locus}. Then the restricted map $\varphi:C\setminus\varphi^{-1}(\Delta)\rightarrow\Pbb^1\setminus\Delta$ is a finite (unbranched) covering map, since $\varphi$ is proper (cf. \cite[Sec. 4.2.1]{Don}).

\subsubsection{Actions  $\Gamma=\pi_1(\Pbb^1\backslash\Sbf,\upgamma_\blt(1))\curvearrowright E$ and admissible group elements}\label{lb52}

Let $\fk P=\eqref{eq16}$ be an $N$-pathed Riemann spheres with local coordinates. We use the notations in Subsection \ref{lb30}. We do not assume \eqref{eq22}.

Let $E$ be a finite set. Let $\Perm(E)$ be the permutation group of $E$. \index{ESE@$E,\Perm(E)$} An action of $\Gamma$ on $E$ is equivalently a homomorphism $\Gamma\rightarrow\Perm(E)$. 

\begin{df}\label{lb61}
We say that $g_1,\dots,g_N\in\Perm(E)$ are \textbf{admissible (with respect to $\fk P$)} \index{00@Admissible group elements} if there is a (necessarily unique) action $\Gamma\curvearrowright E$ sending
\begin{align}
	[\upalpha_j]\mapsto g_j	\label{eq2}
\end{align}
for each $j=1,\dots, N$. The action $\Gamma\curvearrowright E$ is called \textbf{the action arising from $g_1,\dots,g_N$}.
\end{df}

Assume the setting of Def. \ref{lb8}, in which $\uppsi:\mc W_\blt\rightarrow\Cbb$ is a conformal block. We say $\uppsi$ is \textbf{separating} on $\mc W_j$, if the only $w_j\in\mc W_j$ satisfying $\uppsi(w_1\otimes\cdots\otimes w_N)=0$ for all $w_1\in\mc W_1,\dots,w_{j-1}\in\mc W_{j-1},w_{j+1}\in\mc W_{j+1},\dots,w_N\in\mc W_N$ is $w_j=0$. The following proposition is similar to \cite[Thm. 4.7]{Hua18}. By this proposition, it is reasonable to consider only twisted conformal blocks associated to $g_1,\dots,g_N$-twisted modules where $g_1,\dots,g_N$ are admissible. And we will actually consider in this section  twisted conformal blocks only of these types, which correspond well to untwisted conformal blocks associated to the permutation covering of $\fk P$. (However, the other results of this article do not logically rely on the following proposition.)

\begin{pp}
Assume that $\uppsi$ is separating on $\mc W_i$ for some $1\leq i\leq N$, and the map $u\in\Ubb\mapsto Y^{g_j}(u,z)\in\End(\mc W_j)[[z]]$ is injective. Then $g_1,\dots,g_N$ are admissible with respect to $\fk P$.
\end{pp}

\begin{proof}
Let $\fk U$ be as in Prop. \ref{lb11}. Then for each $u\in\fk U$, $\wr\uppsi(\cdot,u,w_\blt)$ is a multivalued holomorphic function on $\Pbb^1\setminus\Sbf$.  Then for each $\uplambda\in\Lambda_{\upgamma_\blt(1)}$, \eqref{eq18} and the $g_j$-equivariance \eqref{eq21} show that for  a small open disc $U$ centered at $\upgamma_j(0)$,
	\begin{align*}
		\wr\uppsi(\upepsilon_j^{\pm 1}\upgamma_j,u,w_\blt)|_U=	\wr\uppsi(\upgamma_j,g_j^{\pm 1}u,w_\blt)|_U
	\end{align*}
	since the left hand side is the analytic continuation of $\wr\uppsi(\upgamma_j,u,w_\blt)$ from $U$ along $\upepsilon_j^{\mp 1}$ to $U$ (i.e.,  multiplying $z$ by $e^{\mp2\im\pi}$   in the expression $\uppsi\big(w_1\otimes\cdots\otimes Y^{g_j}\big(u,z\big)w_j\otimes\cdots\otimes w_N\big)$). By analytic continuation along $\upgamma_j$ from $U$ to a simply-connected neighborhood $V$ of $\upgamma_\blt(1)$, we have (noticing \eqref{eq1})
	\begin{align}
		\wr\uppsi(\upalpha_j^{\pm1},u,w_\blt)|_V=	\wr\uppsi(\upgamma_\blt(1),g_j^{\pm1}u,w_\blt)|_V\label{eq42}
	\end{align}
	where $\upgamma_\blt(1)=\upgamma_j^{-1}\upgamma_j$ denotes the constant path pointing at $\upgamma_\blt(1)$. Apply this relation successively, we get
	\begin{align*}
		\wr\uppsi(\upgamma_\blt(1),f(g_1,\dots,g_N)u,w_\blt)|_V=\wr\uppsi(f(\upalpha_1,\dots,\upalpha_N),u,w_\blt)|_V,
	\end{align*}
for any word $f$.

Now suppose $g_1,\dots,g_N$ are not admissible. Then there exists a word $f$ such that $f([\upalpha_1],\dots,[\upalpha_N])=1$ but $g:=f(g_1,\dots,g_N)\neq 1$. Thus, the above equation implies 	\begin{align*}
\wr\uppsi(\upgamma_\blt(1),gu,w_\blt)|_V=\wr\uppsi(\upgamma_\blt(1),u,w_\blt)|_V.
\end{align*}
Consider $\wr\uppsi(\upgamma_\blt(1),gu-u,w_\blt)$ as a multivalued holomorphic functions on $\Pbb^1\setminus\Sbf$, i.e. single-valued on $\mc P$. Since $g\neq 1$, we can choose $v=gu-u\neq 0$. Then $\uppsi\big(w_1\otimes\cdots\otimes Y^{g_j}\big(v,z\big)w_j\otimes\cdots\otimes w_N\big)=0$ for some nonzero $u\in\Ubb$ and all $w_1,\dots,w_N$. Since $\uppsi$ is separating on $\mc W_j$, we conclude $Y^{g_j}(v,z)w_j=0$ for all $w_j\in\mc W_j$, and hence $v=0$. This gives a contradiction. 
\end{proof}

\begin{rem}
	Note that if a $g$-twisted $\mc W$ ($g\in G$) is non-trivial, then $u\in\Ubb\mapsto Y^g(u,z)$ is injective whenever $\Ubb$ is simple (as a $\Ubb$-module). Indeed, assume $Y^g(u,z)=0$. Choose any $v\in \Ubb$, and assume without loss of generality that $gv=e^{2\im a\pi/k}v$. Then the algebraic Jacobi identity for $Y^g$ (cf. for instance \cite[Sec. 10]{Gui24b}) shows that for each $m,n\in\Zbb$ we have
	\begin{align*}
		\sum_{l\in\Nbb}{\frac ak+m\choose l} Y^g\big(Y(v)_{n+l}u,z\big)\cdot z^{\frac ak+m-l}=0.
	\end{align*}
	Since $Y(v)_nu=0$ for sufficiently large $n$, one can easily show, by induction on $n$ starting from sufficiently large numbers, that $Y^g(Y(v)_nu,z)=0$ for all $n$. A similar argument shows $Y^g(Y(v_1)_{n_1}\cdots Y(v_l)_{n_l}u,z)=0$. If $u\neq 0$, as $\Ubb$ is simple, we have $Y^g(v,z)=0$ for all $v\in\Ubb$. This is impossible.
\end{rem}

\begin{rem}
	Suppose $U\subset\Pbb^1\setminus\Sbf$ is open and simply-connected, and $\uplambda$ is a path in $\Pbb^1\setminus\Sbf$ from inside $U$ to $\upgamma_\blt(1)$. Then, by analytic continuation of \eqref{eq42} from $V$ to $U$ along $\uplambda$, we see
	\begin{align}
		\wr\uppsi(\uplambda\upalpha_j^{\pm1},u,w_\blt)\big|_U=	\wr\uppsi(\uplambda,g_j^{\pm1}u,w_\blt)\big|_U\label{eq43}
	\end{align}
	for all $u\in\fk U$, and hence all $u\in\scr U_{\Pbb^1}(U)$. Formula \eqref{eq43} will be used in Subsec. \ref{lb49}.
\end{rem}

\begin{rem}\label{lb64}
Formula \eqref{eq43} suggests a definition of twisted conformal blocks when $\fk P$ is replaced by an arbitrary positively $N$-pathed compact Riemann surface $\fk X=(X;x_\blt;\eta_\blt;\upgamma_\blt)$, but  $\Gamma=\pi_1(X\setminus\Sbf,\upgamma_\blt(1))$ (where $\Gamma=\pi_1(X\setminus\Sbf,\upgamma_\blt(1))$) is not necessarily generated by $[\upalpha_1],\dots,[\upalpha_N]$. (This will happen even if $X=\Pbb^1$.) Fix an action of $\Gamma$ on $E$, and assume each $\mc W_j$ associated to $x_j$ is $g_j$-twisted where $g_j$ is the image of $[\upalpha_j]$. Then a  linear functional $\uppsi:\mc W_\blt\rightarrow\Cbb$ is called a  conformal block  if it satisfies, in addition to the conditions in Def. \ref{lb8}, that for each path $\updelta$ in $X\setminus\Sbf$ from and to $\upgamma_\blt(1)$, the relation
\begin{align}
\wr\uppsi(\updelta,u,w)\big|_U=	\wr\uppsi(\upgamma_\blt(1),[\updelta]u,w)\big|_U	
\end{align}
holds for any open simply-connected $U\subset X\setminus\Sbf$ containing $\upgamma_\blt(1)$, any $u\in\scr U_X(U)$, and any $w\in\mc W_\blt$.
\end{rem}

\subsubsection{The permutation covering of $\Pbb^1$ associated to $\Gamma\curvearrowright E$}

Fix an action $\Gamma\curvearrowright E$, and let $g_j\in\Perm(E)$ be the image of $[\upalpha_j]$. Then $E$ is the disjoint union of orbits of $\Gamma$. We choose one element for each orbit, called the marked point of that $\Gamma$-orbit. The set of all these marked points of $\Gamma$-orbits is denoted by $E(\Gamma)$, \index{EG@$E(\Gamma)$, the set of marked points of the $\Gamma$-orbits} which is a subset of $E$. Then
\begin{align*}
E=\bigsqcup_{\mbf e\in E(\Gamma)}\Gamma \mbf e.	
\end{align*}

The following is well-known; see  \cite[Sec. 4.2.2, Thm. 2]{Don} or \cite[Sec. 19b]{Ful}.

\begin{pp}\label{lb3}
There is a compact Riemann surface $C=\bigsqcup_{\mbf e\in E(\Gamma)}C^{\mbf e}$ whose connected components $\{C^{\mbf e}:\mbf e\in E(\Gamma)\}$ are in one-to-one correspondence with the points of $E(\Gamma)$, and a branched covering $\varphi:C\rightarrow\Pbb^1$ whose restriction $\varphi:C\setminus\varphi^{-1}(\Sbf)\rightarrow\Pbb^1\setminus\Sbf$ is unbranched,  such that the following condition is satisfied: (Note that $\varphi^{-1}(\Sbf)$ must be a discrete and hence finite subset of $C$.) 

For each $\mbf e\in E(\Gamma)$, there is an element $p^{\mbf e}\in C^{\mbf e}\cap\varphi^{-1}(\upgamma_\blt(1))$ satisfying the following condition: for any path $\upnu\in\Lambda_{\upgamma_\blt(1)}$ in $\Pbb^1\setminus\Sbf$, if $\wtd\upnu$ is its lift to $C$ ending at $p^{\mbf e}$, then the initial point of $\wtd\upnu$ is $p^{\mbf e}$ if and only if $[\upnu]\mbf e=\mbf e$.
\end{pp}

\begin{proof}
For each $\mbf e\in E(\Gamma)$, the existence of such a topological (and hence analytic) unbranced covering $\varphi:C^{\mbf e}\setminus\varphi^{-1}(\Sbf)\rightarrow\Pbb^1\setminus\Sbf$ follows from basic algebraic topology. This covering is finite, since for any $x\in\Pbb^1\setminus\Sbf$ the set $C^{\mbf e}\cap\varphi^{-1}(y)$ is bijective to $\Gamma \mbf e$. One checks easily that for any compact set $K$, every sequence of $\varphi^{-1}(K)$ has a subsequence converging to a point in $C^{\mbf e}\cap\varphi^{-1}(\Sbf)$. Namely, $\varphi$ is proper on $C^{\mbf e}\setminus\varphi^{-1}(\Sbf)$.

We now extend it to $C^{\mbf e}$. The (finite) covering $\varphi:\varphi^{-1}(W_j\setminus\{x_j\})\rightarrow W_j\setminus\{x_j\}$ restricts to a map $\varphi:V \rightarrow W_j\setminus\{x_j\}$ for each connected component $V$ of $\varphi^{-1}(W_j\setminus\{x_j\})$, which is (easy to see) surjective and proper, hence a covering map. This covering map is (topologically and hence analytically) equivalent to $\mc D_r^\times\xrightarrow{z^n}\mc D_r^\times$ for some $r>0$. We may thus add a point $y$ to $V$ such that $\varphi:V\cup\{y\}\rightarrow W_j$ is analytically equivalent to $\mc D_r\xrightarrow{z^n}\mc D_r$. By adding all such $y$, we get a new Riemann  surface $C^{\mbf e}$ and a holomorphic $\varphi:C^{\mbf e}\rightarrow\Pbb^1$. It is clear that $\varphi$ is proper on each $\varphi^{-1}(W_j)$. Thus it is proper on $C$. In particular, $C$ is compact.
\end{proof}

Let us formulate the above proposition in a way independent of $E(\Gamma)$.

\begin{df}\label{lb66}
Let $\varphi:C\rightarrow\Pbb^1$ be a branched covering which is unbranched outside $\Sbf$. A map $\Psi_{\upgamma_\blt(1)}:E\rightarrow\varphi^{-1}(\upgamma_\blt(1))$ is called \textbf{$\Gamma$-covariant} \index{00@$\Gamma$-covariant bijection} if for every $e\in E$ and $\upmu\in\Lambda_{\upgamma_\blt(1)}$, the lift of $\upmu$ to $C$ (or more precisely, in $C\setminus\varphi^{-1}(\Sbf)$) ending at $\Psi_{\upgamma_\blt(1)}(e)$ must start from $\Psi_{\upgamma_\blt(1)}([\upmu]e)$.
\end{df}

\begin{thm}\label{lb38}
There exists a compact Riemann surface $C$, a branched covering $\varphi:C\rightarrow\Pbb^1$ unbranched outside $\Sbf$, and a $\Gamma$-covariant bijection $\Psi_{\upgamma_\blt(1)}:E\rightarrow \varphi^{-1}(\upgamma_\blt(1))$.
\end{thm}

\begin{proof}
Let $\varphi:C\rightarrow\Pbb^1$ be as in Prop. \ref{lb3}. For each $e\in E$, we may find $\upmu\in\Lambda_{\upgamma_\blt(1)}$ and $\mbf e\in E(\Gamma)$ such that $e=[\upmu]\mbf e$. Define $\Psi_{\upgamma_\blt(1)}(e)$ to be the initial point of the lift of $\upmu$ ending at $p^{\mbf e}$, which is clearly inside the connected component of $C$ containing $p^{\mbf e}=\Psi_{\upgamma_\blt}(\mbf e)$.

We have $\Psi_{\upgamma_\blt(1)}([\upmu_1]\mbf e)=\Psi_{\upgamma_\blt(1)}([\upmu_2]\mbf e)$ iff the lifts of $\upmu_1,\upmu_2$ ending at $p^{\mbf e}$ have the same initial point, iff the lift of $\upmu_2^{-1}\upmu_1$ ending at $p^{\mbf e}$ must start at $p^{\mbf e}$, iff (by the statements in Prop. \ref{lb3}) $[\upmu_2^{-1}\upmu_1]\mbf e=\mbf e$, iff $[\upmu_1]\mbf e=[\upmu_2]\mbf e$. This proves that $\Psi_{\upgamma_\blt(1)}$ is well-defined.

Suppose $e_1,e_2\in E$ and $\Psi_{\upgamma_\blt(1)}(e_1)=\Psi_{\upgamma_\blt(1)}(e_2)$. Write $e_1=[\upmu_1]\mbf e_1,e_2=[\upmu_2]\mbf e_2$ for some $\upmu_1,\upmu_2\in\Lambda_{\upgamma_\blt(1)}$ and $\mbf e_1,\mbf e_2\in E(\Gamma)$. Since $\Psi_{\upgamma_\blt(1)}(e_1)$ belongs to $C^{\mbf e_1}$ and $\Psi_{\upgamma_\blt(1)}(e_2)$ belongs to $C^{\mbf e_2}$,  $\mbf e_1$ and $\mbf e_2$ are equal, which we denote by $\mbf e$. Then the above paragraph shows $[\upmu_1]\mbf e=[\upmu_2]\mbf e$, i.e. $e_1=e_2$. This proves that $\Psi_{\upgamma_\blt(1)}$ is injective. The $\Gamma$-covariance is obvious.

Finally, for each $x\in \varphi^{-1}(\upgamma_\blt(1))$, we choose $\mbf e\in E(\Gamma)$ such that $x\in C^{\mbf e}$. Choose a curve $\wtd\upmu$ in $C^{\mbf e}\setminus\Sbf$ from $x$ to $\mbf e$, and let $\upmu=\varphi\circ\wtd\upmu$. Then we clearly have $\Psi_{\upgamma_\blt(1)}([\upmu]\mbf e)=x$. So $\Psi_{\upgamma_\blt(1)}$ is surjective.
\end{proof}

We explore some properties of this branched covering. 

\begin{thm}\label{lb4}
For each $x\in\Pbb^1\setminus\Sbf$ and $\uplambda\in\Lambda_x$, there is a unique bijection \index{zz@$\Psi_\uplambda$, the trivialization of $\varphi:C\rightarrow\Pbb^1$  determined by $\Psi_{\upgamma_\blt(1)}$}
\begin{gather*}
\Psi_\uplambda:E\longrightarrow \varphi^{-1}(x),
\end{gather*}
satisfying the following properties (a) and (b):
\begin{enumerate}[label=(\alph*)]
\item By considering $\upgamma_\blt(1)$ as the constant path at this point, $\Psi_{\upgamma_\blt(1)}$ is the $\Gamma$-covariant bijection given in Thm. \ref{lb38}.
\item Suppose $\uplambda_1,\uplambda_2$ are paths in $\Pbb^1\setminus\Sbf$, $\uplambda_1$ ends at the initial point of $\uplambda_2$, and $\uplambda_2$ ends at $\upgamma_\blt(1)$. Let $\wtd\uplambda_1$ be the lift to $C$ of $\uplambda_1$ ending at $\Psi_{\uplambda_2}(e)$. Then 
\begin{align}
	\text{$\wtd\uplambda_1$ goes from $\Psi_{\uplambda_1\uplambda_2}(e)$ to $\Psi_{\uplambda_2}(e)$}.	\label{eq67}
\end{align}
\end{enumerate}
$\Psi_\uplambda$ depends only on the homotopy class $[\uplambda]$. Moreover, for each $\upmu\in\Lambda_{\upgamma_\blt(1)}$ and $e\in E$ we have
\begin{align}
	\Psi_\uplambda([\upmu]e)=\Psi_{\uplambda\upmu}(e).\label{eq7}	
\end{align}
\end{thm}

We call $\Psi$ \textbf{the trivialization of $\varphi:C\rightarrow\Pbb^1$  determined by $\Psi_{\upgamma_\blt(1)}$}. This name is justified by the fact that, by varying $x$ in a simply-connected open set $U\subset\Pbb^1\setminus\Sbf$ and multiplying $\lambda$ from the left by a curve $l$ in $U$ ending at the initial point of $\lambda$, we obtain (using $\Psi$) an equivalence between the projection  $E\times U\rightarrow U$ and the covering $\varphi:\varphi^{-1}(U)\rightarrow U$.

\begin{proof}
Uniqueness: $\Psi_{\upgamma_\blt(1)}$ is unique. By (b), $\Psi_\uplambda(e)$ is the initial point of the lift of $\uplambda$ to $C$ ending at $\Psi_{\upgamma_\blt(1)}(e)$, which is unique.

Existence: We already have $\Psi_{\upgamma_\blt(1)}$. For each $x\in\Pbb^1\setminus\Sbf,\uplambda\in\Lambda_x$, we define the map $\Psi_\uplambda$ sending each $e$ to the initial point of the lift of $\uplambda$ to $C$ ending at $\Psi_{\upgamma_\blt(1)}(e)$. It is clear that $\Psi$ satisfies (a) and (b). 

Obviously, $\Psi_\uplambda$ relies only on $[\uplambda]$; it is a bijection since $\Psi_{\upgamma_\blt(1)}$ is so. By $\Gamma$-covariance, $\Psi_{\upgamma_\blt(1)}([\upmu]e)=\Psi_{\upmu}(e)$. Denote this point by $x$. Then by (b), the left and the right of \eqref{eq7} are both the initial point of the lift of $\uplambda$ to $C$ ending at $x$. This proves \eqref{eq7}.	
\end{proof}

We shall investigate the local behavior of $\varphi$ near $\varphi^{-1}(\Sbf)$. Recall that $W_i$ is a disc centered at $x_i$ (with respect to the coordinate $\eta_i$) which does not contain any other points of $\Sbf$.

\begin{pp}\label{lb5}
Each connected component $\wtd W_j$ of $\varphi^{-1}(W_j)$ contains exactly one point of $\varphi^{-1}(x_j)$. Let $R>0$ such that the disc $\eta_j(W_j)$ equals $\mc D_R$. Then there is $n\in\Zbb_+$ and a bi-holomorphic function $\wtd\eta_j:\wtd W_j\rightarrow\mc D_r$ (where $r=\sqrt[n]R$) such that the following diagram commutes: 
\begin{equation}
\begin{tikzcd}
\wtd W_j \arrow[r, "\wtd\eta_j","\simeq"'] \arrow[d, "\varphi"]
& \mc D_r \arrow[d, "z^n"] \\
W_j \arrow[r,  "\eta_j","\simeq"']
&  \mc D_{r^n}
\end{tikzcd}	\label{eq4}
\end{equation}
\end{pp}

\begin{proof}
Since $\varphi$ is locally equivalent to $z\mapsto z^n$ for some $n\in\Zbb_+$, it is clear that $\varphi(\wtd W_j)$ contains at least one point $x$ of $W_j\setminus\Sbf=W_j\setminus\{x_j\}$. For any $x'\in W_j\setminus\{x_j\}$, we can lift a path in $W_j\setminus\{x_j\}$ from $x$ to $x'$ into $C$ (and hence in $\varphi^{-1}(W_j)$), and the end point of that lifted path must be in $\varphi^{-1}(x')$. Since $\wtd W_j$ is a connected component,  the end point must be in $\wtd W_j$. So the locally biholomorphic map $\varphi:\wtd W_j\setminus\varphi^{-1}(\Sbf)\rightarrow W_j\setminus\{x_j\}=W_j\setminus\Sbf$ is surjective. We show that it is proper, and therefore  a covering map. Let $K\subset W_j\setminus\{x_j\}$ be compact, and let $(y_k)$ be a sequence in $\wtd W_j\cap\varphi^{-1}(K)\subset\wtd W_j\setminus\varphi^{-1}(\Sbf)$. By passing to a subsequence, $y_k$ converges to $y\in C$. Since $\varphi(y_k)\rightarrow \varphi(y)$, we have $\varphi(y)\in K\subset W_j$ and hence $y\in\varphi^{-1}(W_j)$. Since any connected component of $\varphi^{-1}(W_j)$ is its closed subset, we have $y\in\wtd W_j$ and hence $y\in\wtd W_j\cap\varphi^{-1}(K)$. This shows $\wtd W_j\cap\varphi^{-1}(K)$ is compact. 

Since $W_j$  is biholomorphic to $\mc D_R$ via $\eta_j$, we identify these two spaces via $\eta_j$. In particular, $W_j\setminus\{x_j\}$ equals $\mc D_R^\times$, and hence we have a holomorphic covering $\varphi:\wtd W_j\setminus\varphi^{-1}(\Sbf)\rightarrow\mc D_R^\times$. Since all connected topological (and hence  analytic) coverings of $\mc D_R^\times$ are equivalent topologically (and hence analytically) to $\mc D_r^\times\xrightarrow{z^n}\mc D_R^\times$ where $n\in\Zbb_+$ and $r=\sqrt[n]{R}$, we conclude that there is a biholomorphic $\wtd\eta_j:\wtd W_j\setminus\varphi^{-1}(\Sbf)\rightarrow\mc D_r^\times$ such that \eqref{eq4} commutes when restricted to $\wtd W_j\setminus\varphi^{-1}(\Sbf)$. 

Note that $\wtd W_j\cap\varphi^{-1}(\Sbf)$ is a discrete  subset of $\wtd W_j$, and hence is finite because $C$ is compact. We set $\wtd\eta_j$ to be $0$ on $\wtd W_j\cap\varphi^{-1}(\Sbf)$. To check that $\wtd\eta_j$ is analytic, by Morera's theorem, it suffices to check that $\wtd\eta_j$ is continuous at any $y\in\wtd W_j\cap\varphi^{-1}(\Sbf)=\wtd W_j\cap\varphi^{-1}(x_j)$. Choose a sequence $y_k\in\wtd W_j\setminus\varphi^{-1}(\Sbf)$ converging to $y$. Then $(\wtd\eta_j(y_k))^n=\eta_j\circ\varphi(y_k)\rightarrow\eta_j\circ\varphi(y)=\eta_j(x_j)=0$. This proves the continuity. 

We have proved that the diagram \eqref{eq4} commutes. To finish the proof, we shall show that $\wtd W_j\cap\varphi^{-1}(\Sbf)$ contains precisely one element of $\varphi^{-1}(\Sbf)$. This will also imply that $\wtd\eta_j:\wtd W_j\rightarrow\mc D_r$ is bijective.

A similar argument as above shows that $\varphi:\wtd W_j\rightarrow W_j$ is proper. Thus, if we choose a sequence $y_k$ in $W_j\setminus\varphi^{-1}(\Sbf)$ such that $\varphi(y_k)\rightarrow x_j$, by passing to a subsequence, we see $y_k\rightarrow y\in W_j$ and $\varphi(y)=x_j$. So $\wtd W_j$ contains at least one point of $\varphi^{-1}(x_j)$.

Now assume $y_1,y_2\in\wtd W_j\cap\varphi^{-1}(x_j)$. For each $i=1,2$, we have proved that $\wtd\eta_j$ is holomorphic near $y_i$, and sends $y_i$ to $0$. Since $\wtd\eta_j$ sends nearby points of $y_i$ to $\mc D_r^\times$, it is not constant near $y_i$, and hence it is open near $y_i$.  Therefore, if $y_1\neq y_2$, we may find $y_1',y_2'\in\wtd W_j\setminus\varphi^{-1}(\Sbf)$ close to $y_1$ and $y_2$ respectively, such that $\wtd\eta_j(y_1')=\wtd\eta_j(y_2')$. This is impossible, since we know that $\wtd\eta_j:\wtd W_j\setminus\varphi^{-1}(\Sbf)\rightarrow\mc D_r^\times$ is bijective.
\end{proof}

The branching index $n$ in the previous Proposition can be calculated explicitly:

\begin{pp}\label{lb12}
Let $\wtd W_j$ be a connected component of $\varphi^{-1}(W_j)$, and let $n$ be the branching index in Prop. \ref{lb5}.  Then $\wtd W_j\cap\varphi^{-1}(\upgamma_j(0))$ has precisely $n$ elements, and there exists $e\in E$ such that
\begin{align*}
\wtd W_j\cap\varphi^{-1}(\upgamma_j(0))=&\Psi_{\upgamma_j}\big(\bk{g_j}e\big)\\
:=&\Big\{\Psi_{\upgamma_j}(g_j^ke):k\in\Zbb\Big\}.	
\end{align*}
In particular,  for any $e\in E$ such that $\Psi_{\upgamma_j}(e)\in \wtd W_j\cap\varphi^{-1}(\upgamma_j(0))$, 
\begin{align*}
n=\big|\bk{g_j}e\big|,	
\end{align*}
the number of elements in the orbit $\bk{g_j}e$.
\end{pp}

Recall that the action of $[\upalpha_j]$ on $E$ equals that of $g_j$. $\bk{g_j}$ is the cyclic subgroup generated by $g_j$.

\begin{proof}
By Prop. \ref{lb5}, we identify $\varphi:\wtd W_j\rightarrow W_j$ with $\mc D_r\xrightarrow{z^n}\mc D_{r^n}$. Since we have assumed $\upgamma_j(0)\in W_j\setminus\{x_j\}=\mc D_{r^n}^\times$,  there are $n$ elements in $\wtd W_j\cap\varphi^{-1}(\upgamma_j(0))$. By Thm. \ref{lb4}, any of them is of the form $\Psi_{\upgamma_j}(e)$ for some $e\in E$. Hence $\wtd W_j\cap\varphi^{-1}(\upgamma_j(0))=\{e^{2k\im\pi/n}\Psi_{\upgamma_j}(e):0\leq k\leq n-1\}$.

Recall $\upepsilon_j$ is an anticlockwise circle in $W_j$ from and to $\upgamma_j(0)$. By Thm. \ref{lb4}-(b), $\Psi_{\upgamma_j}(e)$ is the initial point of the lift of $\upgamma_j$ in $C\setminus\varphi^{-1}(\Sbf)$ ending at $\Psi_{\upgamma_\blt(1)}(e)$. Thus, for each $k\in\Zbb$, the lift of $\upepsilon_j^k\upgamma_j=\upgamma_j\upalpha_j^k$ (recall \eqref{eq1}) ending at $\Psi_{\upgamma_\blt(1)}(e)$ and the lift of $\upepsilon_j^k$ ending at $\Psi_{\upgamma_j}(e)$ have the same initial point, which is $e^{-2k\im\pi/n}\Psi_{\upgamma_j}(e)$. Thus, we conclude $\Psi_{\upgamma_j}(g_j^ke)=\Psi_{\upgamma_j}([\upalpha_j^k] e)=e^{-2k\im\pi/n}\Psi_{\upgamma_j}(e)$. This finishes the proof. (The formula for $n$ follows immediately since $\Psi_{\upgamma_j}$ is one-to-one.)
\end{proof}

The above propositions immediately show:

\begin{co}\label{lb6}
For each $1\leq j\leq N$, there is a (necessarily unique) bijective map \index{zz@$\Upsilon$}
\begin{align*}
\Upsilon:\{\bk{g_j}\emph{-orbits in }E\}\longrightarrow \varphi^{-1}(x_j)
\end{align*}
such that for each $e\in E$, the point $\Upsilon(\bk{g_j} e)$    and the set $\Psi_{\upgamma_j}\big(\bk{g_j} e\big)$ are contained in the same connected component $\wtd W_j$ of $\varphi^{-1}(W_j)$. In particular, the domain and the codomain of $\Upsilon$ are both bijective to the set of connected components of $\varphi^{-1}(W_j)$.
\end{co}

Pictorially, $\Upsilon(\bk{g_j}e)$ is the center of the disc $\wtd W_j$, and  $\Psi_{\upgamma_j}(\bk{g_j}e)$ is a set of $n$ points on $\wtd W_j$ surrounding that center.

\begin{rem}
We have suppressed the subscript $j$ and write $\Upsilon_j$ as $\Upsilon$ for simplicity. This means that if $i\neq j$,  $\Upsilon_i$ and $\Upsilon_j$ are different maps. In particular, even if (say) $g_i$ equals $g_j$, $\Upsilon(\bk{g_i}e)$ and $\Upsilon(\bk{g_j}e)$ are allowed to be different.	
\end{rem}

\begin{thm}\label{lb29}
Any two branched coverings satisfying Thm. \ref{lb38} are equivalent. More precisely, assume $\varphi:C\rightarrow\Pbb^1$ and $\varphi':C'\rightarrow\Pbb^1$ are branched coverings with $\Gamma$-covariant bijections $\Psi_{\upgamma_\blt(1)},\Psi_{\upgamma_\blt'(1)}$ as in Thm. \ref{lb38}. Then there is a unique holomorphic map $F:C\rightarrow C'$ such that  the following diagrams commute:
\begin{equation}\label{eq6}
\begin{tikzcd}
C \arrow[rr, "F","\simeq"'] \arrow[dr, "\varphi"'] && C' \arrow[dl, "\varphi'"] \\
&  \Pbb^1
\end{tikzcd}	
\end{equation}
\begin{equation}\label{eq68}
\begin{tikzcd}
&E \arrow[dl, "\Psi_{\upgamma_\blt(1)}"'] \arrow[dr,  "{\Psi'_{\upgamma_\blt'(1)}}"]\\
\varphi^{-1}(\upgamma_\blt(1))\arrow[rr, "F", "\simeq"'] && (\varphi')^{-1}(\upgamma_\blt'(1))
\end{tikzcd}
\end{equation}
$F$ is a bi-holomorphism. Moreover, for each path $\uplambda$ in $\Pbb^1\setminus\Sbf$ ending at $\upgamma_\blt(1)$, $e\in E$, and $1\leq j\leq N$, $F$ satisfies 
\begin{gather*}
F\big(\Psi_\uplambda(e)\big)=\Psi'_\uplambda(e),\\
F\big(\Upsilon(\bk{g_j}e)\big)=\Upsilon'(\bk{g_j}e).
\end{gather*}
\end{thm}

Note that $\Psi_\uplambda'$ is the trivialization for $\varphi'$ defined by $\Psi'_{\upgamma_\blt'(1)}$, which defines $\Upsilon'$  defined for $\varphi'$ as in Cor. \ref{lb6}.

\begin{proof}
Uniqueness: By basic algebraic topology, if we fix a point for each connected component of $C$, the  continuous maps $F$ satisfying \eqref{eq6} are uniquely determined by their values at these points. Therefore, the holomorphic $F$ satisfying the two commuting diagrams are unique when  restricted to $C\setminus\varphi^{-1}(\Sbf)$. Since $\varphi^{-1}(\Sbf)$ is discrete, the values of $F$ on 	$\varphi^{-1}(\Sbf)$ are also unique.
	
Existence: Define a map $F:C\setminus\varphi^{-1}(\Sbf)\rightarrow C'\setminus(\varphi')^{-1}(\Sbf),\Psi_\uplambda(e)\mapsto \Psi'_\uplambda(e)$ for each path $\uplambda$ in $\Pbb^1\setminus\Sbf$ and each $e\in E$. Once we have shown that $F$ is well-defined, then $F$ clearly satisfies \eqref{eq6} and \eqref{eq68} (outside $\varphi^{-1}(\Sbf)$), and is a local homeomorphism by Thm. \ref{lb4}-(b) (applied to the two trivializations).

Suppose $\Psi_{\uplambda_1}(e_1)=\Psi_{\uplambda_2}(e_2)$. Let $x$ denote their image under $\varphi$. Then $\uplambda_1,\uplambda_2\in\Lambda_x$, and hence $\upmu:=\uplambda_2^{-1}\uplambda_1$ belongs to $\Lambda_{\upgamma_\blt(1)}$. By \eqref{eq7}, $\Psi_{\uplambda_2}(e_2)=\Psi_{\uplambda_1\upmu}(e_2)=\Psi_{\uplambda_1}([\upmu]e_2)$. So $\Psi_{\uplambda_1}(e_1)=\Psi_{\uplambda_2}(e_2)$ iff $\Psi_{\uplambda_1}(e_1)=\Psi_{\uplambda_1}([\upmu]e_2)$ iff $e_1=[\upmu]e_2$. Similarly, $\Psi'_{\uplambda_1}(e_1)=\Psi'_{\uplambda_2}(e_2)$ iff $e_1=[\upmu]e_2$. So $F$ is well-defined and is a bijection.

We now extend $F$ to $\varphi^{-1}(\Sbf)$: We let $F$ send each $\Upsilon(\bk{g_j}e)$ to $\Upsilon'(\bk{g_j}e)$ (for all $e\in E$). Now we have a bijective $F:C\rightarrow C'$, and \eqref{eq6} commutes. To finish the proof, it remains to show that $F$ is holomorphic at each point of the finite set $\varphi^{-1}(\Sbf)$. By Morera's theorem, it suffices to check the continuity.

Let $\wtd W_j$ be the connected component of $\varphi^{-1}(W_j)$ containing $y=\Upsilon(\bk{g_j}e)$, and also containing a set $\Psi_{\upgamma_j}(\bk{g_j}e)$ of $n$ points surrounding $\Upsilon(\bk{g_j}e)$. By \eqref{eq4}, $\wtd W_j\setminus\varphi^{-1}(\Sbf)=\wtd W_j\setminus\{y\}$ is  the set of all $\Psi_{l\upgamma_j}(e)$ where $l$ is a path in $W_j\setminus\{x_j\}$ ending at $\upgamma_j(0)$. (Note that $\Psi_{l\upgamma_j}(e)$ is the initial point of the lift of $l$ ending at $\Psi_{\upgamma_j}(e)$.) Similarly, if we let $\wtd W_j'$ be the connected component of $(\varphi')^{-1}(W_j)$ containing $y'=\Upsilon'(\bk{g_j}e)$ and hence containing the set $\Psi_{\upgamma_j}'(\bk{g_j}e)$, then $\wtd W_j'\setminus\{y'\}$ is the set of all $\Psi'_{l\upgamma_j}(e)$. This shows that $F$ restricts to a holomorphic map from $\wtd W_j\setminus\{y\}$ to $\wtd W_j'\setminus\{y'\}$. Choose any sequence $y_k$ in $\wtd W_j\setminus\{y\}$ converging to $y$. Then $\varphi'(Fy_k)=\varphi(y_k)\rightarrow \varphi(y)=x_j$. By Prop. \ref{lb5}, $\varphi':\wtd W_j'\rightarrow W_j$ is equivalent to $\mc D_r\xrightarrow{z^n}\mc D_{r^n}$, and $x_j$ is equivalent to $0\in\mc D_{r^n}$. This shows that $Fy_k$ converges to $0\in\mc D_r$, namely, converges to $y'=Fy$. This proves the continuity.
\end{proof}

We call the data \index{zz@$\varphi:C\rightarrow\Pbb^1$}
\begin{align*}
\big(\varphi:C\rightarrow\Pbb^1;\Psi_{\upgamma_\blt(1)}\big)
\end{align*}
the \textbf{permutation covering} of $\Pbb^1$ associated to the permutations $g_1,\dots,g_N$ (equivalently, the action $\Gamma\curvearrowright E$). Its isomorphism class (in the sense of Thm. \ref{lb29}) depends only on the action $\Gamma\curvearrowright E$ (note that $\Gamma$ also contains the information $\gamma_\blt(1)$) but not on the other information of $\upgamma_\blt$. However, $\Upsilon$ does rely on the homotopy class of each $\upgamma_j$.

\begin{rem}
The $\Gamma$-orbits in $E$ are in one-to-one correspondence with the connected components of $C$: given a $\Gamma$-orbit $\Omega\subset E$, the corresponding connected component $C^\Omega$ is the one containing the set $\Psi_{\upgamma_\blt(1)}(\Omega)$, and hence containing $\Psi_{\uplambda}(\Omega)$ for any path $\uplambda$ in $C\setminus\varphi^{-1}(\Sbf)$ ending at $\upgamma_\blt(1)$. $C^\Omega\setminus\varphi^{-1}(\Sbf)$ is an $|\Omega|$-fold covering of $\Pbb^1\setminus\Sbf$ where $|\Omega|$ is the cardinality of $\Omega$. It is clear that for each $\bk{g_j}$-orbit $\omega=\bk{g_j}e$, the branched point $\Upsilon(\bk{g_j}e)$ is on $C^\Omega$ if and only if $\omega\subset\Omega$. We know that the branching index at $\Upsilon(\omega)$ is $|\omega|$. Therefore, we may use the Riemann-Hurwitz formula \cite[17.14]{For} to conclude:
\end{rem}

\begin{co}
For each $\Gamma$-orbit $\Omega$, let $\Orb_\Omega(g_j)$ be the set of $\bk{g_j}$-orbits inside $\Omega$. Then  the genus $g(C^\Omega)$ of $C^\Omega$ equals
\begin{align}\label{eq88}
g(C^\Omega)=1-|\Omega|+\frac 12\sum_{j=1}^N~\sum_{\omega\in\Orb_\Omega(g_j)}(|\omega|-1).
\end{align}
\end{co}

In this formula, the sum $\sum|\omega|$ clearly equals $|\Omega|$. Therefore
\begin{align}
	g(C^\Omega)=1+\Big(\frac N2-1\Big)|\Omega|-\frac 12\sum_{j=1}^N|\Orb_\Omega(g_j)|\label{eq89}
\end{align}
where $|\Orb_\Omega(g_j)|$ is the number of $\bk{g_j}$-orbits in $\Omega$. When $N=3$, Eq. \eqref{eq89} agrees with \cite[Lemma 8]{BS11}.

\begin{rem}\label{lb25}
Suppose that we have new paths $\upgamma_1',\dots,\upgamma_N'$ with common end point  satisfying that  $[\upgamma_j']=[l_j\upgamma_j\upsigma]$, $l_j$ is a path in $W_j$ from $\upgamma_j'(0)$ to $\upgamma_j(0)$, and $\upsigma$ is a path in $\Pbb^1\setminus\Sbf$ with end point $\upgamma_\blt(1)$. (In particular, we assume $\upgamma_j'(0)=l_j(0)$ and $\upgamma_j'(1)=\upsigma(1)$.) 
	
(Namely, we assume $\upgamma_\blt'$ satisfy the conditions in Rem. \ref{lb2} except \eqref{eq53}.)
	
Similar to \eqref{eq1}, we define $\upalpha_j'=(\upgamma_j')^{-1}\upepsilon_j'\upgamma_j'$ where $\upepsilon_j'$ is an anticlockwise loop in $W_j$ from and to $\upgamma_j'(0)$. This defines an action of $\Gamma'=\pi_1(\Pbb^1\setminus\Sbf,\upgamma_\blt'(1))$ by sending each $[\upalpha_j']$ to $g_j$. We then have an isomorphism $\Gamma\xrightarrow{\simeq}\Gamma'$ defined by $[\upmu]\mapsto[\upsigma^{-1}\upmu\upsigma]$, and this isomorphism sends $[\upalpha_j]$ to $[\upalpha_j']$. This isomorphism intertwines the actions of $\Gamma,\Gamma'$ on $E$, namely, we have commuting diagram
\begin{equation}\label{eq69}
\begin{tikzcd}
\Gamma\arrow[rr, "\simeq"',"{[\upmu]\mapsto [\upsigma^{-1}\upmu\upsigma]}"] \arrow[dr] && \Gamma'\arrow[dl]\\
&\mathrm{Aut}(E)
\end{tikzcd}	
\end{equation}
The following proposition allows us to change  the base point $\upgamma_\blt(1)$.
\end{rem}

\begin{pp}\label{lb26}
Let $(\varphi:C\rightarrow\Pbb^1;\Psi_{\upgamma_\blt(1)})$ be a branched covering associated to $\Gamma\curvearrowright E$ and the paths $\upgamma_\blt$. Assume the setting of Rem. \ref{lb25}. Let $\Psi$ be the trivialization determined by $\Psi_{\upgamma_\blt(1)}$, which defines $\Upsilon$ as in Cor. \ref{lb6}. Then the map
\begin{gather}\label{eq70}
\begin{array}{c}
	\Psi'_{\upgamma_\blt'(1)}:E\rightarrow \varphi^{-1}(\upgamma_\blt'(1)),\\[0.8ex]
	\Psi'_{\upgamma_\blt'(1)}(e)=\Psi_{\upsigma^{-1}}(e)
\end{array}	
\end{gather}
is a $\Gamma'$-covariant bijection, and $(\varphi:C\rightarrow\Pbb^1;\Psi'_{\upgamma'_\blt(1)})$ is a branched covering associated to $\Gamma'\curvearrowright E$ and the paths ${\upgamma'_\blt(1)}$.

Moreover, let $\Psi'$ be the trivialization defined by $\Psi'_{\upgamma_\blt'(1)}$, which defines $\Upsilon'$ as in Cor. \ref{lb6}. Then
\begin{align*}
\Upsilon=\Upsilon',	
\end{align*}
and for each $e\in E$ and each path $\uplambda$ in $\Pbb^1\setminus\Sbf$ ending at $\upgamma_\blt(1)$, 
\begin{gather}
\Psi_{\uplambda}(e)=\Psi'_{\uplambda\upsigma}(e).	\label{eq52}
\end{gather}
\end{pp}

\begin{proof}
For each path $\uplambda'$ in $\Pbb^1\setminus\Sbf$ ending at $\upgamma_\blt'(1)$, define a bijection $\Psi'_{\uplambda'}=\Psi_{\uplambda'\upsigma^{-1}}:E\rightarrow \varphi^{-1}(\uplambda'(0))$. Then $\Psi'$ satisfies conditions (a) and (b) of Thm. \ref{lb4}. Let us show that, by considering $\upgamma_\blt'(1)$ as a constant path, $\Psi'_{\upgamma_\blt'(1)}=\Psi_{\sigma^{-1}}$ is $\Gamma'$-covariant. Choose any path $\upmu'$ in $\Pbb^1\setminus\Sbf$ from and to $\upgamma_\blt'(1)$. For each $e\in E$, by \eqref{eq69}, $[\upmu']e=[\upsigma\upmu'\upsigma^{-1}]e$. Therefore
\begin{align*}
\Psi'_{\upgamma'_\blt(1)}([\upmu']e)=\Psi_{\upsigma^{-1}}([\upsigma\upmu'\upsigma^{-1}]e)\xlongequal{\eqref{eq7}} \Psi_{\upmu'\upsigma^{-1}}(e),
\end{align*}
which, by Thm. \ref{lb4}-(b), is the initial point of the lift of $\upmu'$ ending at 	$\Psi_{\upsigma^{-1}}(e)=\Psi'_{\upgamma'_\blt(1)}(e)$.

We have proved that $\Psi'_{\upgamma_\blt'(1)}$ is a $\Gamma'$-covariant bijection,  that the above defined $\Psi'$ is the corresponding trivialization, and that \eqref{eq52} holds. We now show $\Upsilon=\Upsilon'$. We know that $\Psi_{\upgamma_j}(\bk{g_j}e)$ are points around $\Upsilon(\bk{g_j}e)$, and that they are on the same connected component of $\varphi^{-1}(W_j)$. The same can be said for $\Psi_{\upgamma_j'}'(\bk{g_j}e)$ and $\Upsilon'(\bk{g_j}e)$. Since $\upgamma'_j=l_j\upgamma_j\upsigma$, from \eqref{eq52}, $\Psi_{\upgamma_j'}'(\bk{g_j}e)=\Psi_{\upgamma_j'\upsigma^{-1}}(\bk{g_j}e)=\Psi_{l_j\upgamma_j}(\bk{g_j}e)$.  Since $\Psi_{l_j\upgamma_j}(g_j^ke)$ is the initial point of the lift of $l_j$ ending at $\Psi_{\upgamma_j}(g_j^ke)$, and since $l_j$ is in $W_j$, $\Psi_{l_j\upgamma_j}(g_j^ke)$ and $\Psi_{\upgamma_j}(g_j^ke)$ are in the same connected component of  $\varphi^{-1}(W_j)$. The same can be said for $\Upsilon(\bk{g_j}e)$ and $\Upsilon'(\bk{g_j}e)$. Therefore, we must have $\Upsilon(\bk{g_j}e)=\Upsilon'(\bk{g_j}e)$. 
\end{proof}

\subsubsection{The permutation covering of $\fk P$ associated to $\Gamma\curvearrowright E$ and $E(g_\blt)$}\label{lb31}

We now assume $\fk P$ is positively $N$-pathed, namely, \eqref{eq22} is also true. We have constructed the permutation covering $(C;\varphi^{-1}(\Sbf))$ of $(\Pbb^1;\Sbf)$. In this subsection, we shall add local coordinates to $C$ at these marked points. These local coordinates should be compatible with those of $\fk P$ and the branched covering map $\varphi$. Moreover, to uniquely determine these local coordinates, we have to fix an element $\tipae$ for each $\bk{g_j}$-orbit, called the \textbf{marked point} of that $\bk{g_j}$-orbit. We let \index{Eg@$E(g_j)$, the set of marked points of the $\bk{g_j}$-orbits}
\begin{align*}
E(g_j)=\{\text{all marked points $\tipae$ of the $\bk{g_j}$-orbits of $E$}\}.	
\end{align*}

For each point of $\varphi^{-1}(x_j)$, necessarily represented by $\Upsilon(\bk{g_j}\tipae)$ where $\tipae\in E(g_j)$, we let $\wtd W_{j,\tipae}$ be the unique connected component of $\varphi^{-1}(W_j)$ containing this point (equivalently, containing $\Psi_{\upgamma_j}(\bk{g_j}\tipae)$). By Prop. \ref{lb5}, we can choose a local coordinate $\wtd\eta_{j,\tipae}\in\scr O(\wtd W_{j,\tipae})$ of $C$ at $\Upsilon(\bk{g_j}\tipae)$ such that (for $k=|\bk{g_j}\tipae|$ and for some $r_i>0$) the diagram \index{Wg@$\wtd W_{j,\tipae}$} \index{zz@$\wtd\eta_{j,\tipae}$}
\begin{equation}\label{eq25}
	\begin{tikzcd}
		\wtd W_{j,\tipae} \arrow[r, "\wtd\eta_{j,\tipae}","\simeq"'] \arrow[d, "\varphi"]
		& \mc D_{\sqrt[k]r_j} \arrow[d, "z^k"] \\
		W_j \arrow[r,  "\eta_j","\simeq"']
		&  \mc D_{r_j}
	\end{tikzcd}	
\end{equation}
commutes. 

Such $\wtd\eta_{j,\tipae}$ is unique up to multiplication by a power of \index{zz@$\tipaomega_k=e^{-\frac{2\im\pi}k}$}
\begin{align}
	\tipaomega_k=e^{-\frac{2\im\pi}k}.	
\end{align}
However, since $\eta_j$ sends $\upgamma_j(0)$ to a positive number (by \eqref{eq22}), we may (and shall) choose $\wtd\eta_{j,\tipae}$ such that
\begin{align}
	\wtd\eta_{j,\tipae}\circ\Psi_{\upgamma_j}(\tipae)>0.	\label{eq28}
\end{align}
Then
\begin{align}
	\boxed{~\fk X=	\big(C;\Upsilon(\bk{g_j}\tipae);\wtd\eta_{j,\tipae}\in\scr O(\wtd W_{j,\tipae})\text{ for all }1\leq j\leq N,\tipae\in E(g_j)\big)~}\label{eq24}
\end{align}
is a pointed compact Riemann surface with local coordinates. The data
\begin{align*}
\big(\fk X,\varphi:C\rightarrow\Pbb^1,\Psi_{\upgamma_\blt(1)}\big)	
\end{align*}
(or simply just $\fk X$) is  called the \textbf{permutation covering} of the positively $N$-pathed \index{00@Permutation coverings of the positively $N$-pathed Riemann sphere with local coordinates}
\begin{align*}
	\boxed{~\fk P=(\Pbb^1;x_1,\dots,x_N;\eta_1,\dots,\eta_N;\upgamma_1,\dots,\upgamma_N)~}
\end{align*}
associated to the action $\Gamma\curvearrowright E$ and the set  $E(g_j)$ (for all $j$) of marked points of $\bk{g_j}$-orbits. $\wtd\eta_{j,\tipae}$ is the local coordinate at $\Upsilon(\bk{g_j}\tipae)$ defined on an open disc $\wtd W_{j,\tipae}$. The set of all marked points is $\varphi^{-1}(\Sbf)$.

\begin{rem}\label{lb27}
Suppose we have positively $N$-pathed $\fk P'=(\Pbb^1;x_\blt;\eta_\blt;\upgamma_\blt')$ where $\upgamma_1',\dots,\upgamma_N'$ are equivalent to $\upgamma_\blt$ (cf. Rem. \ref{lb2}). Define the action of $\Gamma'=\pi_1(\Pbb^1\setminus\Sbf,\upgamma_\blt'(1))$ on $E$ as in  Rem. \ref{lb25}. Let $\fk X$ be \eqref{eq24}. Define a $\Gamma'$-covariant bijection $\Psi'_{\upgamma_\blt'(1)}$ using \eqref{eq70}, which defines a trivialization $\Psi$. Noting \eqref{eq66}, for each $\tipae\in E(g_j)$ we have $\Psi'_{\upgamma_j'}(\tipae)=\Psi_{l_j\upgamma_j}(\tipae)$, whose value under $\wtd\eta_{j,\tipae}$ is positive. This fact, together with Prop. \ref{lb26}, implies that
\begin{align*}
\big(\fk X,\varphi:C\rightarrow\Pbb^1,\Psi'_{\upgamma_\blt'(1)}\big)
\end{align*}
is a permutation covering of $\fk P'$ associated to $\Gamma'\curvearrowright E$ and $E(g_\blt)$. 
\end{rem}

\subsection{Permutation-twisted modules; Main Theorem}\label{lb23}

\begin{df}\label{lb72}
Let $\Ubb$ be the tensor product VOA $\Ubb=\Vbb^{\otimes E}\equiv\bigotimes_{e\in E}\Vbb$. In other words, as a vector space, $\Vbb^{\otimes E}$ is a vector bundle (with possibly infinite rank) over a single point set (say $\{0\}$) such that for any bijection  $\varphi:\{1,\dots,n\}\rightarrow E$ (where $n=|E|$) we have a trivialization $\Psi_\varphi:\Vbb^{\otimes n}\xrightarrow{\simeq} \Vbb^{\otimes E}$; if $\sigma:\{1,\dots,n\}\rightarrow\{1,\dots,n\}$ is a bijection, then $\Psi_\varphi^{-1}\circ\Psi_{\varphi\circ\sigma}:\Vbb^{\otimes n}\rightarrow\Vbb^{\otimes n}$ is the unique (bijective) linear map sending $v_{\sigma(1)}\otimes\cdots\otimes v_{\sigma(n)}$ to $v_1\otimes\cdots\otimes v_n$ (for each $v_1,\dots,v_n\in\Vbb$). 

For each function $\vbf:E\rightarrow\Vbb$, a vector in $\Vbb^{\otimes E}$ is defined:
\begin{align}\label{eq104}
\bigotimes_{e\in E} \vbf(e)=\Psi_\varphi\Big(\vbf(\varphi(1))\otimes\cdots \otimes\vbf(\varphi(n))\Big)
\end{align} 
This definition is independent of the choice of $\varphi$, since $\Psi_\varphi^{-1}\circ\Psi_{\varphi\circ\sigma}$ sends $\vbf(\varphi\circ\sigma(1))\otimes\cdots \otimes\vbf(\varphi\circ\sigma(n))$ to $\vbf(\varphi(1))\otimes\cdots \vbf(\varphi(n))$. We let $\vbf$ also denote $\bigotimes_{e\in E}\vbf(e)$ when no confusion arises.

Pushing forward the VOA structure of $\Vbb^{\otimes n}$ to $\Vbb^{\otimes E}$ via $\Psi_\varphi$, then $\Vbb^{\otimes E}$ becomes a VOA. This VOA structure is independent of the choice of $\varphi$ (because $\Psi_\varphi^{-1}\circ\Psi_{\varphi\circ\sigma}$ is an automorphism of the VOA $\Vbb^{\otimes n}$). The conformal vector $\cbf_\Ubb$ of $\Ubb=\Vbb^{\otimes E}$ is $\sum_{\epsilon\in E}\cbf_\epsilon$ where $\cbf_\epsilon(e)$ equals $\cbf_\Vbb$ (the conformal vector of $\Vbb$) when $\epsilon=e$, and equals $\id$ otherwise.

The (faithful) group action of $\Perm(E)$ on $\Ubb$ is determined by the following fact: for each $g\in\Perm(E)$ and each $\vbf:E\rightarrow \Vbb$, we have
\begin{align}
g\cdot \bigotimes_{e\in E}\vbf(e)=\bigotimes_{e\in E}\vbf(g^{-1}e).
\end{align}
In short, we have $(g\cdot \vbf)(e)=\vbf(g^{-1}e)$.  \hfill\qedsymbol
\end{df}

Thus, if $E=\{1,\dots,n\}$, then $\Ubb$ can simply be the VOA $\Vbb^{\otimes n}$ understood in the usual sense.

Recall $g_1,\dots,g_N\in\Perm(E)$. For each $g_j$, we shall construct $g_j$-twisted $\mbb U$-modules. For each $\bk{g_j}$-orbit $\bk{g_j}\tipae$ we choose a $\Vbb$-module $\Wbb_{j,\tipae}$.

Let $\zeta$ be the standard coordinate of $\Cbb$. Fix any $\tipae\in E(g_j)$. We set a $2$-pointed Riemann sphere with local coordinates
\begin{align}
(\Pbb^1_{j,\tipae};0,\infty;\zeta,\zeta^{-1})	\label{eq40}
\end{align}
where $\Pbb^1_{j,\tipae}=\Pbb^1$. We associate $\Wbb_{j,\tipae}$ and its contragredient  $\Wbb_{j,\tipae}'$ to $0,\infty$ respectively. Then the standard pairing \index{zz@$\tipxphi_{j,\tipae}$}
\begin{gather}
\tipxphi_{j,\tipae}:	\Wbb_{j,\tipae}\otimes \Wbb_{j,\tipae}'\rightarrow\Cbb, \qquad w\otimes w'\mapsto \bk{w,w'}\label{eq33}
\end{gather}
is a conformal block associated to \eqref{eq40} and the two $\Vbb$-modules. Set
\begin{align}
{k}=|\bk{g_j}\tipae|.	
\end{align}
Let $\Cbb_{j,\tipae}=\Pbb_{j,\tipae}^1\setminus\{\infty\}$ and $\Cbb_{j,\tipae}^\times=\Pbb_{j,\tipae}^1\setminus\{0,\infty\}$. Identify
\begin{align}
\boxed{~\scr V_{\Cbb^\times_{j,\tipae}}=\Vbb\otimes_\Cbb\scr O_{\Cbb^\times_{j,\tipae}}\qquad \text{via }\mc U_\varrho(\zeta^{{k}})~}\label{eq35}
\end{align}
where $\zeta^k:z\mapsto z^k$. By Thm. \ref{lb7}, we have
\begin{align*}
\wr^{{k}}\tipxphi_{j,\tipae}:\Big(\Vbb\otimes\scr O(\Cbb_{j,\tipae}^\times)\Big)^{\otimes {k}}\otimes \Wbb_{j,\tipae}\otimes \Wbb_{j,\tipae}'\rightarrow\scr O\big(\Conf^{{k}}(\Cbb_{j,\tipae}^\times)\big)
\end{align*}
where all $\otimes$ are over $\Cbb$. For each $z\in\Cbb^\times$ with argument $\arg z$,  $\sqrt[k]z$ is assumed to have argument $\frac 1k\arg z$. Let \index{zz@$\tipaomega_k^{\blt-1}z^{1/k}$}
\begin{align*}
\tipaomega_k^{\blt-1}z^{1/k}=(z^{1/k},\tipaomega_kz^{1/k},\tipaomega_k^2z^{1/k},\dots,\tipaomega_k^{k-1}z^{1/k})	\qquad\in\Conf^k(\Cbb^\times_{j,\tipae}).
\end{align*}
Then for each $\vbf:E\rightarrow \Vbb$ (viewed as an element of $\Ubb$), $w\in \Wbb_{j,\tipae}$, $w'\in \Wbb_{j,\tipae}'$, and $z\in\Cbb^\times$, we set \index{zz@$\tipxcc_{j,\tipae}$}
\begin{align}
\boxed{\begin{array}{rl}
&\tipxcc_{j,\tipae}(\vbf,w,w',z)\\[0.7ex]
:=&	\wr^{{k}}\tipxphi_{j,\tipae}\Big(\vbf(\tipae),\vbf(g_j\tipae),\dots,\vbf(g_j^{{k}-1}\tipae), w\otimes w'\Big)_{\tipaomega_{{k}}^{\blt-1}z^{1/{k}}}
\end{array}}	\label{eq29}
\end{align}

We now set
\begin{align}
\mc W_j=\bigotimes_{\tipae\in E(g_j)}\Wbb_{j,\tipae}	\label{eq23}
\end{align}
with $\wtd L_0$-action
\begin{align}
\wtd L_0^{g_j} \Big(\bigotimes_{\tipae} w_{j,\tipae}\Big)=\Big(\sum_{\tipae}\frac {\wtd\wt w_{j,\tipae}}{|\bk{g_j}\tipae|}\Big)\Big(\bigotimes_{\tipae} w_{j,\tipae}\Big)\label{eq73}
\end{align}
if each $w_{j,\tipae}\in\Wbb_{j,\tipae}$ is $\wtd L_0$-homogeneous with eigenvalue $\wtd\wt w_{j,\tipae}$.

\begin{thma}\label{lb13}
\textit{There is a (necessarily unique) vertex operation $Y^{g_j}$ which makes the $\wtd L_0^{g_j}$-graded vector space $\mc W_j$ a $g_j$-twisted $\Ubb$-module and satisfies the following condition: for each $\vbf:E\rightarrow\Vbb$ (viewed as an element of $\Ubb$), $\wbf_j=\bigotimes_\tipae w_{j,\tipae}\in\otimes_{\tipae\in E(g_j)}\Wbb_{j,\tipae}=\mc W_j$, and $\wbf_j'=\bigotimes_\tipae w_{j,\tipae}'\in\otimes_{\tipae\in E(g_j)}\Wbb_{j,\tipae}'=\mc W_j'$,}
\begin{align}
\big\langle Y^{g_j}(\vbf,z)\wbf_j,\wbf_j'\big\rangle=\prod_{\tipae\in E(g_j)}{\tipxcc}_{j,\tipae}(\vbf,w_{j,\tipae},w_{j,\tipae}',z).	\label{eq30}
\end{align}
\end{thma}

Note that $Y^{g_j}$ depends not only on $g_j$ and the $\Vbb$-modules associated to the $\bk{g_j}$-orbits, but also on the set $E(g_j)$.
\begin{proof}
When there is only one $\bk{g_j}$-orbit, this was proved in \cite[Sec. 10]{Gui24b}. In the general case, this follows from the fact that if for each $1\leq l\leq m$, $h_l$ is an automorphism of a VOA $\Vbb_l$ with finite order, and if $\mc M_l$ is a $h_l$-twisted $\Vbb_l$-module with vertex operation $Y^{h_l}$, then $\mc M_1\otimes\cdots\otimes\mc M_m$ is an $h=h_1\otimes\cdots\otimes h_m$-twisted $\Vbb_1\otimes\cdots\otimes \Vbb_m$-module with vertex operation $v_1\otimes\cdots\otimes v_m\mapsto Y^{h_1}(v_1,z)\otimes\cdots\otimes Y^{h_m}(v_m,z)$ and $\wtd L^h_0$-grading defined by the sum over all $1\leq l\leq m$ of the $\wtd L^{h_l}_0$-weight on the $\mc M_l$-component.
\end{proof}

The following Proposition is \cite[Thm. 7.10-(1)]{BDM02}

\begin{pp}\label{lb41}
If for each $\tipae\in E(g_j)$, $\Wbb_{j,\tipae}$ is an irreducible $\Vbb$-module, then $\mc W_j$ is an irreducible $g_j$-twisted $\Ubb$-module.
\end{pp}

We now relate contragredient untwisted and twisted modules. Set
\begin{align}
h_j=g_j^{-1},\qquad E(h_j)=E(g_j).	
\end{align}
For each $\bk{h_j}$-orbit $\bk{h_j}\tipae$ where $\tipae\in E(h_j)$, we choose $\Wbb_{j,\tipae}'$, the contragedient module of $\Wbb_{j,\tipae}$. Then $\mc W_j'$ equals $\bigotimes_{\tipae\in E(h_j)}\Wbb_{j,\tipae}'$ as $\wtd L_0$-graded vector spaces. We use these data to construct a vertex operation $Y^{h_j}$ for $\mc W_j'$ as in Thm. \ref{lb13}. Recall the definition of contragredient twisted modules in \eqref{eq64}. The following theorem will not be used until Chapter \ref{lb35}.

\begin{thm}\label{lb42}
$(\mc W_j',Y^{h_j})$ is the contragredient twisted module of $(\mc W_j,Y^{g_j})$.
\end{thm}

\begin{proof}
Consider $\tipxphi_{j,\tipae}$ defined by \eqref{eq33} as a conformal block associated to $(\Pbb^1;\infty,0;\zeta^{-1},\zeta)$, which is equivalent to $(\Pbb^1_{j,\tipae};0,\infty;\zeta,\zeta^{-1})$ via the map $z\mapsto z^{-1}$. This equivalence sends the standard coordinate $\zeta$ to $\zeta^{-1}$, and sends any $\tipaomega_k^lz^{1/k}$ to $\tipaomega_k^{-l}z^{-1/k}$ where $z\in\Pbb,k\in\Zbb_+,l\in\Zbb$. Thus, assuming the identification $\scr V_{\Cbb^\times}=\Vbb\otimes_\Cbb\scr O_{\Cbb^\times}$ via $\mc U_\varrho(\zeta^{-k_{j,\tipae}})$, by the definition of $Y^{h_j}$, we have (noticing Thm. \ref{lb7}-(4))
\begin{align*}
&\big\langle \wbf_j,Y^{h_j}(\vbf,z)\wbf_j'\big\rangle\\
=&	\prod_{\tipae\in E(h_j)}\wr^{k_{j,\tipae}}\tipxphi_{j,\tipae}\Big(\vbf(\tipae),\vbf(g_j^{-1}\tipae),\dots,\vbf(g_j^{-k_{j,\tipae}+1}\tipae), w\otimes w'\Big)_{\tipaomega_{k_{j,\tipae}}^{-\blt+1}z^{-1/k_{j,\tipae}}}\\
=& \prod_{\tipae\in E(h_j)}\wr^{k_{j,\tipae}}\tipxphi_{j,\tipae}\Big(\vbf(\tipae),\vbf(g_j\tipae),\dots,\vbf(g_j^{k_{j,\tipae}-1}\tipae), w\otimes w'\Big)_{\tipaomega_{k_{j,\tipae}}^{\blt-1}z^{-1/k_{j,\tipae}}}
\end{align*}
where for  each $k\in\Zbb_+$ we set
\begin{align*}
\tipaomega_k^{-\blt+1}z^{-1/k}=(z^{-1/k},\tipaomega_k^{-1}z^{-1/k},\dots,\tipaomega_k^{-k+1}z^{-1/k}).	
\end{align*}

We now do not assume the identification. Then
\begin{align*}
&\Big\langle \wbf_j,Y^{h_j}\big(\mc U(\tipxgamma_z)\vbf,z^{-1}\big)\wbf_j'\Big\rangle\\
=&	\prod_{\tipae\in E(h_j)}\wr^{k_{j,\tipae}}\tipxphi_{j,\tipae}\Big(\mc U_\varrho(\zeta^{-k_{j,\tipae}})^{-1}\mc U(\tipxgamma_z)\vbf(\tipae),\dots,\\
&\qquad\mc U_\varrho(\zeta^{-k_{j,\tipae}})^{-1}\mc U(\tipxgamma_z)\vbf(g_j^{k_{j,\tipae}-1}\tipae), w\otimes w'\Big)_{\tipaomega_{k_{j,\tipae}}^{\blt-1}z^{1/k_{j,\tipae}}}.
\end{align*}
By \eqref{eq65}, for each $k\in\Zbb_+$, $\mc U_\varrho(\zeta^{-k})\mc U_\varrho(\zeta^k)^{-1}=\mc U(\varrho(\zeta^{-k}|\zeta^k))$. By Example \ref{lb36}, the value of $\varrho(\zeta^{-k}|\zeta^k)$ at each $z\in\Cbb^\times$ is $\tipxgamma_{z^k}$. Therefore, its value at $\tipaomega_k^lz^{1/k}$ (for each $l\in\Zbb$) is $\tipxgamma_z$. Thus, the value at $\tipaomega_k^lz^{1/k}$ of $\mc U_\varrho(\zeta^{-k})\mc U_\varrho(\zeta^k)^{-1}$ is $\mc U(\tipxgamma_z)$. It follows that
\begin{align*}
&\Big\langle \wbf_j,Y^{h_j}\big(\mc U(\tipxgamma_z)\vbf,z^{-1}\big)\wbf_j'\Big\rangle\\
=&	\prod_{\tipae\in E(h_j)}\wr^{k_{j,\tipae}}\tipxphi_{j,\tipae}\Big(\mc U_\varrho(\zeta^{k_{j,\tipae}})^{-1}\vbf(\tipae),\dots,\mc U_\varrho(\zeta^{k_{j,\tipae}})^{-1}\vbf(g_j^{k_{j,\tipae}-1}\tipae),w\otimes w'\Big)_{\tipaomega_{k_{j,\tipae}}^{\blt-1}z^{1/k_{j,\tipae}}},
\end{align*}
which is just $\big\langle Y^{g_j}(\vbf,z)\wbf_j,\wbf_j'\big\rangle$.
\end{proof}

The following theorem is due to \cite[Thm. 6.4]{BDM02}. 

\begin{thm}\label{lb47}
Suppose $\Vbb$ is rational (i.e., any admissible $\Vbb$-module is completely reducible). Then any $g_j$-twisted $\Ubb$-module is a direct sum of those of the form \eqref{eq23} (whose module structures are described in Thm. \ref{lb13}), where each $\Wbb_{j,\tipae}$ is an irreducible  $\Vbb$-module.
\end{thm}

Let $\fk P$ be a positively $N$-pathed Riemann sphere with local coordinates, together with a permutation covering $(\fk X,\varphi,\Psi_{\upgamma_\blt(1)})$ as in Subsection \ref{lb31}.

Associate $\mc W_1,\dots,\mc W_N$ to the marked points $x_1,\dots,x_N$ of $\fk P$. Any marked point of $\fk X$ is of the form $\Upsilon(\bk{g_j}\tipae)$ for some $1\leq j\leq N$ and $\tipae\in E(g_j)$, to which we associate the $\Vbb$-module $\Wbb_{j,\tipae}$. By \eqref{eq23}, we have
\begin{align*}
	\mc W_\blt:=\bigotimes_{1\leq j\leq N}\mc W_j=\bigotimes_{\begin{subarray}{c}
			1\leq j\leq N\\	
			\tipae\in E(g_j)
	\end{subarray}}\Wbb_{j,\tipae}=:\Wbb_{\blt,\blt}.
\end{align*}

Consider a linear functional $\upphi:\Wbb_{\blt,\blt}\rightarrow\Cbb$. This same map can be regarded as a linear functional $\uppsi:\mc W_\blt\rightarrow\Cbb$. The reason for using two different symbols for the same linear functional is due to the following reason: we can ask whether $\upphi$ is a conformal block associated to $\fk X$, and whether $\uppsi$ is a conformal block associated to $\fk P$; even when both are true, their propagations $\wr\upphi$ and $\wr\uppsi$ will have different meanings.

The theorem below is the first major result of this article, and its proof is left to the following two sections.

\begin{thm}[Main Theorem]\label{lb14}
$\upphi$ is a conformal block associated to $\fk X$ and the $\Vbb$-modules $\Wbb_{\blt,\blt}$ if and only if $\uppsi$ is a conformal block associated to $\fk P$ and the twisted $\Vbb^{\otimes 
E}$-modules $\mc W_\blt$.
\end{thm}

Consequently, the spaces of the two types of conformal blocks are isomorphic.

\subsection{From untwisted to permutation-twisted conformal blocks}\label{lb1}

In this section, we prove the ``only if" part of Thm. \ref{lb14}. Let us  make some preparations. First, note that:

\begin{rem}\label{lb15}
$\scr U_{\Pbb^1}$ can be identified with a tensor product (indexed by $E$) over $\scr O_{\Pbb^1}$ of $\scr V_{\Pbb^1}$ such that for any open $V\subset\Pbb^1$ and locally injective $\mu\in\scr O(V)$, the trivialization $\mc U_\varrho(\eta):\scr U_V\xrightarrow{\simeq}\Ubb\otimes_\Cbb\scr O_V=\Vbb^{\otimes E}\otimes_\Cbb\scr O_V$ agrees with $\mc U_\varrho(\eta)^{\otimes E}:\scr V_V^{\otimes E}\xrightarrow{\simeq}\Vbb^{\otimes E}\otimes_\Cbb\scr O_V$. To see this, one checks that the two transition functions agree. This is easy from the definition of $\cbf_\Ubb$. It follows that for each connected open $U\subsetneq\Pbb^1$, we have
\begin{align}
\scr U_{\Pbb^1}(U)=\scr V_{\Pbb^1}(U)^{\otimes E}\equiv\otimes_{e\in E}\scr V_{\Pbb^1}(U)\label{eq41}
\end{align}
where the tensor product is over $\scr O(U)$ (rather than over $\Cbb$).
\end{rem}

Another useful fact is:

\begin{rem}\label{lb16}
For any open and simply-connected $U\subset\Pbb^1\setminus\Sbf$, since $\varphi^{-1}(U)$ is a disjoint union of some open sets biholomorphic to $U$, we can use the pull back $\varphi^*:\scr O(U)\rightarrow\scr O(\varphi^{-1}(U)),f\mapsto f\circ\varphi$ to define a natural pullback map $\varphi^*:\scr V_{\Pbb^1}(U)\rightarrow\scr V_C(\varphi^{-1}(U))$, i.e. the one compatible with the restriction to open subsets of $U$ and intertwines the actions of $\scr O(U)$, and in the case that there exists any locally injective $\mu\in\scr O(U)$, the following diagram commutes
\begin{equation}\label{eq36}
	\begin{tikzcd}
		\scr V_{\Pbb^1}(U) \arrow[r, "\varphi^*"] \arrow[d,"\mc U_\varrho(\mu)"', "\simeq"] &\scr V_C(\varphi^{-1}(U)) \arrow[d, "\mc U_\varrho(\mu\circ\varphi)", "\simeq"']\\
		\Vbb\otimes_\Cbb\scr O(U)\arrow[r,"1\otimes\varphi^*"] &\Vbb\otimes_\Cbb\scr O(\varphi^{-1}(U))
	\end{tikzcd}	
\end{equation} 
\end{rem}

We now begin to prove the ``only if" part of Thm. \ref{lb14}. Assume $\upphi:\Wbb_{\blt,\blt}\rightarrow\Cbb$ is a conformal block associated to $\fk X$, which equals $\uppsi$ as a linear functional. $\zeta$ always denotes the standard coordinate of a complex plane.
 
\begin{proof}[\textbf{Step 1}] Let us construct the $\wr\uppsi$ that will satisfy Condition 2 of Def. \ref{lb8}. Choose any $\wbf=\bigotimes_{j,\tipae}w_{j,\tipae}\in \mc W_\blt$. For any open simply-connected $U\subset\Pbb^1\setminus\Sbf$ and a path $\uplambda$ in $\Pbb^1\setminus\Sbf$ from inside $U$ to $\upgamma_\blt(1)$, we define an $\scr O(U)$-module homomorphism $\wr\uppsi(\uplambda,\cdot,\wbf):\scr U_{\Pbb^1}(U)\rightarrow\scr O(U)$ as follows. By Rem. \ref{lb15}, we have identification $\scr U_{\Pbb^1}(U)=\scr V_{\Pbb^1}(U)^{\otimes E}$. This space is spanned by $\vbf=\bigotimes_{e\in E}\vbf(e)$ where each $\vbf(e)\in\scr V_{\Pbb^1}(U)$. We understand
\begin{align*}
	\varphi^*\vbf\in \scr V_C(\varphi^{-1}(U))^E	
\end{align*}
as the $|E|$-tuple labeled by $E$ (i.e. a function from $E$ to $\scr V_C(\varphi^{-1}(U))$) whose $e$-component (for each $e\in E$) is $\varphi^*\vbf(e)=\varphi^*(\vbf(e))$ defined by Rem. \ref{lb16}.

Let $\Conf^E(\varphi^{-1}(U))$ be the subset of all points $\mbf y$ of $\varphi^{-1}(U)^E$ (i.e. any function $\mbf y:E\rightarrow \varphi^{-1}(U)$) such that any two components $\mbf y(e_1),\mbf y(e_2)$ are different if $e_1\neq e_2$. For each $y\in U$, we can set 
\begin{gather}
\varphi_\uplambda^*y\in\Conf^E(\varphi^{-1}(U)),\nonumber\\
\varphi_\uplambda^*y(e)=\Psi_{l\uplambda}(e)\qquad (\text{$l$ is a path in $U$ from $y$ to $\uplambda(0)$}).	\label{eq27}
\end{gather}
Then  $y\in U\mapsto\varphi_{\uplambda}^*y$ is holomorphic, and $\varphi_\uplambda^*y\in\varphi^{-1}(y)$.

Consider the $E$-propagation $\wr^E\upphi$ of $\upphi$, i.e., the $|E|$-propagation whose components are labeled by $E$. Then
\begin{align*}
\wr^E\upphi(\varphi^*\vbf,\wbf)\in\Conf^E(\varphi^{-1}(U)).	
\end{align*}
We can then define $\wr\uppsi(\uplambda,\vbf,\wbf)\in\scr O(U)$ such that for each $y\in U$,
\begin{gather}
\boxed{~\wr\uppsi(\uplambda,\vbf,\wbf)_y=\wr^E\upphi(\varphi^*\vbf,\wbf)_{\varphi_\uplambda^*y}~}\label{eq3}
\end{gather}
One checks easily that $\wr\uppsi(\cdot,\cdot,\wbf)$ is a multivalued $\scr O_{\Pbb^1\setminus\Sbf}$-module morphism from $\scr U_{\Pbb^1\setminus\Sbf}$ to $\scr O_{\Pbb^1\setminus\Sbf}$ (cf. Lemma \ref{lb17}).

We remark that $\varphi^*\vbf$ is not uniquely determined by $\vbf$: it depends  on how $\vbf$ is factored into the tensor product over $\scr O(U)$ of elements of $\scr V_{\Pbb^1}(U)$. However, if any tensor component of $\vbf$ is multiplied by some $f\in\scr O(U)$, then $\wr^E\upphi(\varphi^*\vbf,\wbf)_{\varphi_\uplambda^*y}$ is multiplied $f(y)$ since $\wr^E\upphi$ intertwines the actions of $\scr O_C$ (cf. Thm. \ref{lb7}). Therefore \eqref{eq3} is $\scr O(U)$-multilinear with respect to the   components of $\vbf$, and hence $\wr^E\upphi(\varphi^*\vbf,\wbf)_{\varphi_\uplambda^*y}$  is uniquely determined.
\end{proof}

\begin{proof}[\textbf{Step 2}] Let us give an explicit expression of $\varphi^*_{\upgamma_j}y$ when $y$ is near $\upgamma_j(0)$. This will help us relate $\wr\uppsi$ and the expression \eqref{eq29}. 

For each $j$, identify 
\begin{align*}
\wtd W_{j,\tipae}=\wtd\eta_{j,\tipae}(\wtd W_{j,\tipae})\qquad\text{via }	\wtd\eta_{j,\tipae},
\end{align*}
which is an open disc inside $\Cbb_{j,\tipae}$ centered at $0$. Likewise, we identify 
\begin{align*}
W_j=\eta_j(W_j) \qquad\text{via }\eta_j,	
\end{align*}
which is an open disc inside $\Cbb_j:=\Cbb$ centered at $0$. Then by \eqref{eq25}, the branched covering $\varphi:\wtd W_{j,\tipae}\rightarrow W_j$ is equal to $\zeta^{k_{j,\tipae}}:z\mapsto z^{k_{j,\tipae}}$ where
\begin{align*}
k_{j,\tipae}=|\bk{g_j}\tipae|.	
\end{align*}
(Recall $\zeta$ is the standard coordinate of $\Cbb_j=\Cbb$).

Choose an open simply-connected $U\subset W_j\setminus\{x_j\}\subset\Cbb_j^\times$ containing $\upgamma_j(0)$ together with a continuous $\arg$-function as in Condition 2 of Def. \ref{lb8}. We prove in this step that for each $z\in U,m\in\Zbb,\tipae\in E(g_j)$,
\begin{align}
\varphi^*_{\upgamma_j}z(g_j^m\tipae)=\tipaomega_k^mz^{\frac 1{k}}\in\wtd W_{j,\tipae}\qquad\text{where }k=k_{j,\tipae}.\label{eq26}
\end{align}	
We note that both sides of this relation rely continuously on $z$ and is sent by the covering map $\varphi$ to $z$. Namely, both sides are lifts of the inclusion $U\hookrightarrow W_j\setminus\{x_j\}=W_j\setminus\{0\}$ to  $C\setminus\varphi^{-1}(\Sbf)$. Moreover, if $z=\upgamma_j(0)$, by \eqref{eq27}, $\varphi^*_{\upgamma_j}z(g_j^m\tipae)=\Psi_{\upgamma_j}(g_j^m\tipae)\in\wtd W_{j,\tipae}$. Therefore, both sides of \eqref{eq26} are lifts of $U\hookrightarrow W_j\setminus\{0\}$ to $\wtd W_{j,\tipae}\setminus \{0\}$. Thus, it suffices to prove \eqref{eq26} for one point in $U$.  Note that, since the left hand side of \eqref{eq26} is a lift of $z$ through $\varphi=\zeta^k$, it must be a $k$-th root of $z$. 

Let us prove \eqref{eq26} for $z=\upgamma_j(0)$. Then
\begin{align*}
\varphi^*_{\upgamma_j}z(g_j^m\tipae)\xlongequal{\eqref{eq27}} \Psi_{\upgamma_j}(g_j^m\tipae)	\xlongequal{\eqref{eq2}}\Psi_{\upgamma_j}([\upalpha_j]^m\tipae)\xlongequal{\eqref{eq7}}\Psi_{\upgamma_j\upalpha_j^m}(\tipae)\xlongequal{\eqref{eq1}}\Psi_{\upepsilon_j^m\upgamma_j}(\tipae).
\end{align*}
By \eqref{eq22}, $z=\upgamma_j(0)$ is positive with zero argument. When $m=0$, $\Psi_{\upepsilon_j^m\upgamma_j}(\tipae)$, which is a $k$-th root of $z$, is also positive due to \eqref{eq28}. So it must be $z^{\frac 1k}$. This proves \eqref{eq26} when $m=0$. For a general $m\in\Zbb$, by the definition of $\Psi$ in Thm. \ref{lb4}, $\Psi_{\upepsilon_j^m\upgamma_j}(\tipae)$ is the initial point of the lift of the path $\upepsilon_j^m$ through $\varphi=\zeta^k$ into the punctured disc $\wtd W_{j,\tipae}\setminus \{0\}$ ending at $\Psi_{\upgamma_j}(\tipae)=z^{\frac 1k}$. It must be $\tipaomega_k^mz^{\frac 1{k}}$, since $\upepsilon_j$ is the anticlockwise circle around the origin (whose lift under $\zeta^k$ goes anticlockwisely by $2\pi/k$).
\end{proof}

\begin{proof}[\textbf{Step 3}] Assume the identifications in Step 2. Let $\wbf\in\mc W_\blt$ be $\bigotimes_{1\leq j\leq N}\wbf_j\in\otimes_j\mc W_j$ where $\wbf_j=\bigotimes_{\tipae\in E(g_j)}w_{j,\tipae}$. Choose $\vbf:E\rightarrow\Vbb$ (viewed as an element of $\Ubb$).  Choose any $z\in W_j\setminus\{0\}$. 

For each $n\in\Nbb$, let $\{m_{j,\tipae}(n,\alpha):\alpha\in\mc A_{j,\tipae,n}\}$ be a finite set of basis of $\Wbb_{j,\tipae}(n)$ whose dual basis $\{\wch m_{j,\tipae}(n,\alpha):\alpha\in\mc A_{j,\tipae,n}\}$ is a basis of $\Wbb_{j,\tipae}'(n)=\Wbb_{j,\tipae}(n)^*$. For each $\mbf n\in \Nbb^{E(g_j)}$ (i.e.,  a function $E(g_j)\rightarrow\Nbb$), and for each $\bsb\alpha$ sending each $\tipae\in E(g_j)$ to an element $\bsb\alpha(\tipae) \in\mc A_{j,\tipae,\mbf n(\tipae)}$ (the set of all such $\bsb\alpha$ is denoted by $\fk A_{j,\mbf n}$), set
\begin{gather*}
\mbf m_j(\mbf n,\bsb\alpha)=\bigotimes_{\tipae\in E(g_j)} m_{j,\tipae}(\mbf n(\tipae),\bsb\alpha(\tipae))\qquad \in \bigotimes_{\tipae\in E(g_j)} \Wbb_{j,\tipae}=\mc W_j,\\
\wch{\mbf m}_j(\mbf n,\bsb\alpha)=\bigotimes_{\tipae\in E(g_j)} \wch m_{j,\tipae}(\mbf n(\tipae),\bsb\alpha(\tipae))\qquad \in \bigotimes_{\tipae\in E(g_j)} \Wbb_{j,\tipae}'=\mc W_j'.
\end{gather*}
Then  the following infinite series of $n\in\frac 1{|g_j|}\Nbb$
\begin{align}
\uppsi (\wbf_1\otimes\cdots\otimes Y^{g_j}(\vbf,z)\wbf_j\otimes\cdots\otimes\wbf_N),	\label{eq31}
\end{align}
understood in the sense of \eqref{eq17}, converges absolutely provided that the following multi-series of $\mbf n\in\Nbb^{E(g_j)}$ converges absolutely:
\begin{align}
&~~~~~\sum_{\mbf n\in\Nbb^{E(g_j)}}\sum_{\bsb\alpha\in\fk A_{j,\mbf n}}\uppsi \big(\wbf_1\otimes\cdots\otimes \mbf m_j(\mbf n,\bsb\alpha)\otimes\cdots\otimes\wbf_N\big)\cdot\big\langle Y^{g_j}(\vbf,z)\wbf_j,\wch{\mbf m}_j(\mbf n,\bsb\alpha)\big\rangle\nonumber\\
&=\sum_{\mbf n\in\Nbb^{E(g_j)}}\sum_{\bsb\alpha\in\fk A_{j,\mbf n}}\upphi \big(\wbf_1\otimes\cdots\otimes \mbf m_j(\mbf n,\bsb\alpha)\otimes\cdots\otimes\wbf_N\big)\nonumber\\
&\cdot\prod_{\tipae\in E(g_j)}\wr^{{k_{j,\tipae}}}\tipxphi_{j,\tipae}\Big(\vbf(\tipae),\vbf(g_j\tipae),\dots,\vbf(g_j^{{k_{j,\tipae}}-1}\tipae), w_{j,\tipae}\otimes \wch m_{j,\tipae}(\mbf n(\tipae),\bsb\alpha(\tipae))\Big)_{\tipaomega_{{k_{j,\tipae}}}^{\blt-1}z^{1/{k_{j,\tipae}}}}.\label{eq32}
\end{align}
(We have used \eqref{eq30} and\eqref{eq29}.) In that case, \eqref{eq31} converges absolutely to \eqref{eq32}. 

Note that inside $\wr^{{k_{j,\tipae}}}\tipxphi_{j,\tipae}$ (for each $\tipae\in E(g_j)$), we have used the identification \eqref{eq35} by setting $k=k_{i,\tipae}$. 

According to \eqref{eq11}, we know that \eqref{eq32} is the sewing of a propagation. By Thm. \ref{lb19}, it converges absolutely to the propagation of the sewing, provided that the sewing of the (unpropagated) conformal block converges $q$-absolutely, and that the marked points and the points of propagation are away from the discs to be sewn (Assumption \ref{lb20} and the statement ``...disjoint from $W_j',W_j''$..." in Thm. \ref{lb19}). 

The unpropagated pointed Riemann surface with local coordinates is 
\begin{align*}
\fk Y=\fk X\sqcup\bigsqcup_{\tipae\in E(g_j)}\fk Q_{j,\tipae}	
\end{align*}
where $\fk Q_{j,\tipae}=(\Pbb^1_{j,\tipae};0,\infty;\zeta,\zeta^{-1})$. The two marked points $0,\infty$ of $\fk Q_{j,\tipae}$ are associated with $\Vbb$-modules $\Wbb_{j,\tipae},\Wbb_{j,\tipae}'$ respectively. The (unpropagated) conformal block associated to $\fk Y$ is (recall \eqref{eq33})
\begin{align*}
\upchi:=\upphi\otimes \Big(\bigotimes_{\tipae\in E(g_j)}\tipxphi_{j,\tipae}\Big):\mc W_\blt\otimes	\Big(\bigotimes_{\tipae\in E(g_j)}\Wbb_{j,\tipae}\otimes\Wbb_{j,\tipae}'\Big)\rightarrow\Cbb.
\end{align*}
We sew $\fk Y$ along each pair $\Upsilon(\bk{g_j}\tipae)$ (i.e. the center of $\wtd W_{j,\tipae}$) and $\infty\in\Pbb^1_{j,\tipae}$  (for every $\tipae\in E(g_j)$) using their local coordinates. More precisely, we remove a small disc inside $\wtd W_{j,\tipae}$ and another disc inside
\begin{align*}
M_{j,\tipae}:=\{z\in\Pbb^1_{j,\tipae}:|z^{-1}|<R^{1/|E|}\}
\end{align*}
for sufficiently large $R>1$, and glue the remaining part using the relation $\wtd\eta_{j,\tipae}\cdot \zeta^{-1}=1$. 

The sewn data  $\scr S\fk Y$ is equal to $\fk X$.   Moreover, our sewing is compatible with the identifications in Step 2. Namely: this sewing process is just removing $\wtd W_{j,\tipae}$ (for each $\tipae$) from $C$, and filling into the holes  the equivalent open disc $\wtd W_{j,\tipae}\subset\Pbb^1_{j,\tipae}$ (associated to $\fk Q_{j,\tipae}$). It is clear that $\scr S\upchi$ converges $q$-absolutely to $\upphi$. Also, when $0<|z|<R^{-1}$, $\tipaomega_{k_{j,\tipae}-1}^\blt z^{1/{k_{j,\tipae}}}$ is disjoint from $M_{j,\tipae}$. So, by Thm. \ref{lb19}, \eqref{eq32} converges absolutely. This proves Condition 1 of Def. \ref{lb8} for all $0<|z|<R^{-1}$ (with any choice of $\arg z$), thanks to Rem. \ref{lb10}.
\end{proof}

\begin{proof}[\textbf{Step 4}] Assume the identifications in Step 2. Let $U=\{z\in\Cbb_j^\times:|z|<R^{-1}\}\setminus (-\infty,0]$ with $\arg$ function ranging in $(-\pi,\pi)$. Note $U\subset W_j\subset\Cbb_j$. Choose $z\in U$. After sewing, the point $\tipaomega_{k_{j,\tipae}}^mz^{1/k_{j,\tipae}}$ originally in $\fk Q_{j,\tipae}$ becomes the same point in $\wtd W_{j,\tipae}$, which is $\varphi^*_{\upgamma_j}z(g_j^m\tipae)$ by \eqref{eq26}. 

Note that both Rem.  \ref{lb15} and Rem. \ref{lb16} are considered in the definition of $\wr\uppsi$ in Step 1. In view of Rem. \ref{lb16}, we assume two more identifications
\begin{gather}
	\scr V_{\Pbb^1}(U)=\Vbb\otimes_\Cbb\scr O(U)	\qquad\text{via }\mc U_\varrho(\zeta)\label{eq39}\\
	\scr V_C(\varphi^{-1}(U))=\Vbb\otimes_\Cbb\scr O(\varphi^{-1}(U))\qquad\text{via }\mc U_\varrho(\zeta\circ\varphi).	\label{eq37}
\end{gather}
Since
\begin{align*}
	\zeta\circ\varphi=\zeta^{k_{j,\tipae}}\qquad\text{on }\varphi^{-1}(U)\cap \wtd W_{j,\tipae},
\end{align*}
after sewing,  the identification \eqref{eq35} (where $k=k_{j,\tipae}$ for each $\tipae$) used in the definition of $\wr^{{k_{j,\tipae}}}\tipxphi_{j,\tipae}$ is compatible with the identification \eqref{eq37}. Also, if we take Rem.\ref{lb15} into account, then \eqref{eq39} yields the identification
\begin{align}
\scr U_{\Pbb^1}(U)=\Ubb\otimes_\Cbb\scr O(U)\qquad\text{via }\mc U_\varrho(\zeta).
\end{align}


Under these identifications, for each $\vbf:E\rightarrow\Vbb$ (viewed as an element of $\Ubb\subset\Ubb\otimes_\Cbb\scr O(U)$),  $\varphi^*\vbf\in\Vbb^E\subset(\Vbb\otimes_\Cbb\scr O(\varphi^{-1}(U)))^E$ is a tuple of constant sections whose component at each $e=g_j^m\tipae$ (where $m\in\Zbb,\tipae\in E(g_j)$) is $\vbf(g_j^m\tipae)$. Thus \eqref{eq32}, which is the sewing of propagation, converges absolutely to the propagation of the sewing $\scr S\upchi=\upphi$ (by Thm. \ref{lb19}), which is $\wr^E\upphi(\varphi^*\vbf,\wbf)$ at the point of $\Conf^E(\varphi^{-1}(U))$ whose $g_j^m\tipae$-component is $\varphi^*_{\upgamma_j}z(g_j^m\tipae)$ (according to the first paragraph). This proves that \eqref{eq32} (and hence \eqref{eq31}) converge absolutely to $\wr^E\upphi(\varphi^*\vbf,\wbf)_{\varphi_{\upgamma_j}^*z}$, which is $\wr\uppsi(\upgamma_j,\vbf,\wbf)_z$ by \eqref{eq3}. Thus, by Rem. \ref{lb10}, the two conditions of Def. \ref{lb8} hold for possibly smaller discs $W_1,\dots,W_N$; the choice of $\arg z$ is not important for Condition 1, due to part (2) of that remark; the choice of the simply-connected subset $U$ can be arbitrary, due to Rem. \ref{lb9}. By Rem. \ref{lb21}, these two conditions hold for the original discs.
\end{proof}

We are done with the proof of the ``only if" part of Thm. \ref{lb14}.

\subsection{From permutation-twisted to untwisted conformal blocks}\label{lb49}

We prove the ``if" part of Thm. \ref{lb14}. Assume $\uppsi:\mc W_\blt\rightarrow\Cbb$ is a conformal block associated to $\fk P$. We have $\wr\uppsi$ as described in Def. \ref{lb8}.

\begin{proof}[\textbf{Step 1}] To show that the same linear functional $\upphi:\Wbb_{\blt,\blt}\rightarrow\Cbb$ is a conformal block associated to $\fk X$, we shall construct its propagation $\wr\upphi$.

Choose any simply-connected open $U\subset\Pbb^1\setminus\Sbf$ together with a path $\uplambda$ in $\Pbb^1\setminus\Sbf$ from inside $U$ to $\upgamma_\blt(1)$. For each $e\in E$, let $\wtd U_{\uplambda,e}$ be the connected component of $\varphi^{-1}(U)$ containing $\Psi_{\uplambda}(e)$. Then all such sets form a basis of the topology of $C$. Define
\begin{align*}
\varphi_*:\scr V_C(\wtd U_{\uplambda,e})\xrightarrow{\simeq}\scr V_{\Pbb^1}(U)	
\end{align*}
to be the inverse of $\varphi^*$ defined in Rem. \ref{lb16}. Recall $\scr U_{\Pbb^1}(U)=\scr V_{\Pbb^1}(U)^{\otimes_{\scr O(U)} E}$ by Rem. \ref{lb15}. For each $v\in\scr V_C(\wtd U_{\uplambda,e})$,  we let
\begin{align*}
(\varphi_*v)_e\otimes\id\in\scr U_{\Pbb^1}(U)	
\end{align*}
be a tensor product of elements of $\scr V_{\Pbb^1}(U)$ such that the $e$-component is $\varphi_*v$, and that the other tensor components are the vacuum section $\id$.

We now define $\wr\upphi$ on $\wtd U_{\uplambda,e}$. Choose any $\wbf=\bigotimes_{j,\tipae}w_{j,\tipae}\in \mc W_\blt$. For each $\wtd y\in\wtd U_{\uplambda,e}$ and $v\in\scr V_C(\wtd U_{\uplambda,e})$, we define
\begin{align}
\boxed{~\wr\upphi\big(v,\wbf\big)_{\wtd y}=\wr\uppsi\big(\uplambda,(\varphi_*v)_e\otimes\id,\wbf\big)_{\varphi(\wtd y)}~}\label{eq50}	
\end{align}
We need to show that this definition is independent of $\uplambda$ and $e$. Suppose $\uplambda'$ is another path from inside $U$ to $\upgamma_\blt(1)$, $e'\in E$, and $\wtd U_{\uplambda,e}=\wtd U_{\uplambda',e'}$. The above definition is clearly unchanged if $e'=e$ and $\uplambda'=l\uplambda$ where $l$ is a path in $U$ from $\uplambda'(0)$ to $\uplambda(0)$. Thus, we may assume $\uplambda$ and $\uplambda'$ have the same initial point (and end point). Choose $\upmu\in\Lambda_{\upgamma_\blt(1)}$ such that
\begin{align*}
\uplambda'=\uplambda\upmu.	
\end{align*}
We compute
\begin{align*}
\Psi_{\uplambda}(e)=\Psi_{\uplambda'}(e')=\Psi_{\uplambda\upmu}(e')\xlongequal{\eqref{eq7}}\Psi_{\uplambda}([\upmu]e').	
\end{align*}
Therefore, as $\Psi_{\uplambda}$ is one-to-one (cf. Thm. \ref{lb4}), we have
\begin{align*}
e=[\upmu]e'.	
\end{align*}

The map $[\upalpha_j]\mapsto g_j$ defines an action of the fundamental group $\Gamma$ on $\Ubb$ and hence on $\scr U_{\Pbb^1}$. It is easy to check that
\begin{align*}
[\upmu]	\big((\varphi_*v)_e\otimes\id\big)=(\varphi_*v)_{[\mu]e}\otimes\id.
\end{align*}
Therefore, by \eqref{eq43},
\begin{align*}
&\wr\uppsi\big(\uplambda',(\varphi_*v)_{e'}\otimes\id,\wbf\big)=\wr\uppsi\big(\uplambda\upmu,(\varphi_*v)_{e'}\otimes\id,\wbf\big)\\
=&	\wr\uppsi\big(\uplambda,(\varphi_*v)_{[\upmu]e'}\otimes\id,\wbf\big)=\wr\uppsi\big(\uplambda,(\varphi_*v)_e\otimes\id,\wbf\big).
\end{align*}
Thus $\wr\upphi(\cdot,\wbf)$ is a well-defined $\scr O_{C\setminus\varphi^{-1}(\Sbf)}$-module morphism $\scr V_{C\setminus\varphi^{-1}(\Sbf)}\rightarrow \scr O_{C\setminus\varphi^{-1}(\Sbf)}$.
\end{proof}

In Step 2, we verify the two conditions of Thm. \ref{lb24}.

\begin{proof}[\textbf{Step 2}]

We write $\wbf_j=\bigotimes_{\tipae\in E(g_j)}w_{j,\tipae}\in\mc W_j$ so that $\wbf=\bigotimes_j\wbf_j$. Let $U\subset W_j\setminus\{x_j\}$ be open and simply-connected. According to the notations in Step 1, 
\begin{align*}
\Psi_{\upgamma_j}(\tipae)\in\wtd U_{\upgamma_j,\tipae}\subset\wtd W_{j,\tipae}.	
\end{align*} 
(Recall that $\wtd W_{j,\tipae}$ is a disc with center $\Upsilon(\bk{g_j}\tipae)$, and hence contains the set of points $\Psi_{\upgamma_j}(\bk{g_j}\tipae)$ surrounding the center.) Choose any $v\in\Vbb\subset\Vbb\otimes_\Cbb\scr O(\wtd U_{\upgamma_j,\tipae})$, and set
\begin{align*}
v_\tipae\otimes\id\in\Ubb=\Vbb^{\otimes E}	
\end{align*}
to be a tensor product of vectors of $\Vbb$ whose $\tipae$-component is $v$ and the other components are $\id$. Then by \eqref{eq50} and Condition 2 of Def. \ref{lb8} (notice the identification there), for each $\wtd y\in\wtd U_{\upgamma_j,\tipae}$,
\begin{align*}
&\wr\upphi\big(\mc U_\varrho(\eta_j\circ\varphi)^{-1}v,\wbf\big)_{\wtd y}\\
=&\uppsi\big(\wbf_1\otimes\cdots\otimes Y^{g_j}\big(v_\tipae\otimes\id,\eta_j\circ\varphi(\wtd y)\big)\wbf_j\otimes\cdots\otimes\wbf_N \big)	
\end{align*}
where the right hand side converges a.l.u. in the sense of Def. \ref{lb8}. Using the notations of Step 3 in Sec. \ref{lb1}, we have
\begin{align*}
&\wr\upphi\big(\mc U_\varrho(\eta_j\circ\varphi)^{-1}v,\wbf\big)_{\wtd y}\\
=&\sum_{\mbf n\in\Nbb^{E(g_j)}}\sum_{\bsb\alpha\in\fk A_{j,\mbf n}}\uppsi\big(\wbf_1\otimes\cdots\otimes \mbf m_j(\mbf n,\bsb\alpha)\otimes\cdots\otimes\wbf_N \big)\nonumber\\
&\cdot \big\langle Y^{g_j}\big(v_\tipae\otimes\id,\eta_j\circ\varphi(\wtd y)\big)\wbf_j,\wch{\mbf m}_j(\mbf n,\bsb\alpha)\big\rangle.	
\end{align*}

Let us write the above $\bk{Y^{g_j}\cdots}$ in terms of $\wr\tipxphi_{j,\tipae}$. Let
\begin{align*}
	k=k_{j,\tipae}=|\bk{g_j}\tipae|.	
\end{align*}
Recall that $\zeta$ is the standard coordinate of $\Cbb$. By \eqref{eq29} and Thm. \ref{lb7}-(3,4), if $\vbf:E\rightarrow\Vbb$ is such that $\vbf(g_j\tipae)=\vbf(g_j^2\tipae)=\cdots=\id$ then
\begin{align*}
\tipxcc_{j,\tipae}(\vbf,w,w',z)=\wr\tipxphi_{j,\tipae}\big(\mc U_\varrho(\zeta^k)^{-1}\vbf(\tipae),w\otimes w'\big)_{z^{1/k}};	
\end{align*}
if also $\vbf(\tipae)=\id$ then $\tipxcc_{j,\tipae}(\vbf,w,w',z)=\bk{w,w'}$.  For each $m_{j,\tipae}\in\Wbb_{j,\tipae}$ we set
\begin{align*}
\wbf_{\setminus j,\tipae}\otimes m_{j,\tipae}\in\Wbb_{\blt,\blt}	
\end{align*}
be a tensor product of vectors whose $(j,\tipae)$-component is $m_{j,\tipae}$ and the other components agree with the corresponding ones of $\wbf$. By \eqref{eq25}, we have
\begin{align*}
(\eta_j\circ\varphi(\wtd y))^{1/k}=\wtd\eta_{j,\tipae}(\wtd y),
\end{align*}
where the argument of $\eta_j\circ\varphi(\wtd y)$ is defined  such that if $\wtd y$ changes continuously to $\Psi_{\upgamma_j}(\tipae)$ then the argument changes continuously to $0$ (notice \eqref{eq28}). Then, using the construction of  $Y^{g_j}$ in Sec. \ref{lb23} (especially, pay attention to the identification \eqref{eq35} there), we have
\begin{align}
&\wr\upphi\big(\mc U_\varrho(\eta_j\circ\varphi)^{-1}v,\wbf\big)_{\wtd y}\nonumber\\
=&\sum_{n\in\Nbb}\sum_{\alpha\in\mc A_{j,\tipae,n}}\uppsi\big(\wbf_{\setminus j,\tipae}\otimes m_{j,\tipae}(n,\alpha) \big)\cdot \wr\tipxphi_{j,\tipae}\Big(\mc U_\varrho(\zeta^k)^{-1}v,w_{j,\tipae}\otimes\wch m_{j,\tipae}(n,\alpha)\Big)_{\wtd\eta_{j,\tipae}(\wtd y)}\label{eq51}	
\end{align}
where the right hand side converges absolutely as a series of $n$.

We now identify
\begin{align*}
\wtd W_{j,\tipae}=\wtd\eta_{j,\tipae}	(\wtd W_{j,\tipae})\qquad\text{via }\wtd\eta_{j,\tipae},
\end{align*}
considered as an open disc in $\Cbb_{j,\tipae}=\Cbb$. Then $\wtd\eta_{j,\tipae}$ is the standard coordinate $\zeta$, $\eta_j\circ\varphi=\zeta^k$, and $\wtd y=\wtd\eta_{j,\tipae}(\wtd y)$. Therefore, \eqref{eq51} holds if we replace $\wtd \eta_{j,\tipae}(\wtd y)$ by $\wtd y$, and (by $\scr O(\wtd U_{j,\tipae})$-linearity) replace $\mc U_\varrho(\eta_j\circ\varphi)^{-1}v=\mc U_\varrho(\zeta^k)^{-1}v$ by any element of $\scr V_C(\wtd U_{\upgamma_j,\tipae})$. Thus, if we fix identification
\begin{align*}
\scr V_{\wtd U_{j,\tipae}}=\Vbb\otimes_\Cbb\scr O_{\wtd U_{j,\tipae}}\qquad\text{via }\mc U_\varrho(\wtd\eta_{j,\tipae})=\mc U_\varrho(\zeta),	
\end{align*}
then for each $z\in\wtd U_{j,\tipae}\subset\Cbb_{j,\tipae}$  and $v\in\Vbb\subset\Vbb\otimes_\Cbb\scr O_{\wtd U_{j,\tipae}}$,
\begin{align*}
&\wr\upphi(v,\wbf)_z\nonumber\\
=&\sum_{n\in\Nbb}\sum_{\alpha\in\mc A_{j,\tipae,n}}\uppsi\big(\wbf_{\setminus j,\tipae}\otimes m_{j,\tipae}(n,\alpha) \big)\cdot \wr\tipxphi_{j,\tipae}\Big(v,w_{j,\tipae}\otimes\wch m_{j,\tipae}(n,\alpha)\Big)_z\\
=&\sum_{n\in\Nbb}\sum_{\alpha\in\mc A_{j,\tipae,n}}\upphi\big(\wbf_{\setminus j,\tipae}\otimes m_{j,\tipae}(n,\alpha) \big)\cdot\big\langle Y(v,z)w_{j,\tipae},\wch m_{j,\tipae}(n,\alpha)\big\rangle\\
=&\sum_{n\in\Nbb}\upphi\big(\wbf_{\setminus j,\tipae}\otimes P_nY(v,z)w_{j,\tipae} \big),
\end{align*}
where the rightmost part converges absolutely. Similar to the argument in Rem. \ref{lb10}, the above equation holds and the series of $n$ converges a.l.u. for $z\in\wtd W_{j,\tipae}\setminus\{x_j\}$. So $\upphi$ satisfies the two conditions of Thm. \ref{lb24}. 
\end{proof}

\begin{subappendices}
\subsection{Dimension of the space of permutation-twisted conformal blocks}\label{lb58}

In this subsection, we assume $\Vbb$ is CFT-type, $C_2$-cofinite, and rational. Let $\mc E$ be a (finite) complete list of irreducible $\Vbb$-modules (cf. the paragraph above Thm. \ref{lb40}). By Main Theorem \ref{lb14}, the calculation of the dimension of the space of permutation-twisted conformal blocks is reduced to that of untwisted ones. For the reader's convenience, we explicitly write down the steps of calculating such dimension.

If $C$ is a connected compact Riemann surface of genus $g$, together with $N$ marked points, local coordinates, and semi-simple  $\Vbb$-modules $\Wbb_1,\dots,\Wbb_N$, we let
\begin{align*}
\Nbf(g;\Wbb_\blt)=\Nbf(g;\Wbb_1,\dots,\Wbb_N)	 
\end{align*}
be the dimension of the space of conformal blocks associated to these data. By \cite{DGT21,DGT22}, this number is finite, and  is independent of the complex structure of $C$, the position of $N$ marked points, and the local coordinates. Moreover, $\Nbf(g;\Wbb_1,\dots,\Wbb_N)$ is unchanged if we rearrange the order $\Wbb_1,\dots,\Wbb_N$. In the special case that $g=0$ and $N=3$,
\begin{align*}
\Nbf_{\Wbb_1,\Wbb_2}^{\Wbb_3}:=\Nbf(0;\Wbb_1,\Wbb_2,\Wbb_3')	
\end{align*}
is the \textbf{fusion rule} among $\Wbb_1,\Wbb_2,\Wbb_3$.

By the factorization property proved by \cite{DGT22} (cf. also Thm. \ref{lb40}), if $g\geq 1$, we have
\begin{align}
\Nbf(g;\Wbb_1,\dots,\Wbb_N)=\sum_{\Mbb\in\mc E}\Nbf(g-1;\Wbb_1,\dots,\Wbb_N,\Mbb,\Mbb').	
\end{align}
If $g=g_1+g_2$ where $g_1,g_2\geq 0$, and if $1\leq L<N$, then
\begin{align}
\Nbf(g;\Wbb_1,\dots,\Wbb_N)=\sum_{\Mbb\in\mc E}\Nbf(g_1;\Wbb_1,\dots,\Wbb_L,\Mbb)\cdot\Nbf(g_2;\Wbb_{L+1},\dots,\Wbb_N,\Mbb').	
\end{align}	
These two formulas allow us to calculate any $\Nbf(g;\Wbb_\blt)$ using the fusion rules.

Now we assume the setting of Main Theorem \ref{lb14}. Namely, $\fk P=\eqref{eq16}$ is a positively  $N$-pathed Riemann sphere with local coordinates, and to each marked point $x_j$ we associate a $g_j$-twisted  $\Ubb=\Vbb^{\otimes E}$-module $\mc W_j$ defined in Thm. \ref{lb13}. Note that $E(g_j)$ is the set of marked points of $\bk{g_j}$-orbits. We assume $g_1,\dots,g_N$ are admissible, i.e., there is an action $\Gamma\curvearrowright E$ where each $[\upalpha_j]$ acts as $g_j$. So $\Gamma$-orbits in $E$ are precisely $\bk{g_1,\dots,g_N}$-orbits. We let $\Orb(\Gamma)$ be the set of $\Gamma$-orbits in $E$. As usual, for $\Omega\in\Orb(\Gamma)$, $|\Omega|$ denotes its cardinality.

Now, by the Main Theorem \ref{lb14}, we have:

\begin{co}\label{lb50}
For each $\Omega\in\Orb(\Gamma)$, let $(\Wbb)_\Omega$ be the list of all $\Wbb_{j,\tipae}$ where $j=1,\dots,N$, and $\tipae\in E(g_j)$ is such that the $\bk{g_j}$-orbit $\bk{g_j}\tipae$ is contained in $\Omega$. Then the dimension of the space of conformal blocks associated to $\fk P$ and the twisted modules $\mc W_\blt$ equals
\begin{align*}
\prod_{\Omega\in\Orb(\Gamma)}\Nbf(g(C^\Omega);(\Wbb)_\Omega)	
\end{align*}
where $g(C^\Omega)$ is given by \eqref{eq88} or equivalently \eqref{eq89}.
\end{co}

\end{subappendices}

\section{Relating untwisted and permutation-twisted sewing and factorization}\label{lb35}

\subsection{Sewing Riemann spheres and their permutation coverings}\label{lb59}

\subsubsection{The setting}\label{lb45}

Let $\Pbb^1_a=\Pbb^1,\Pbb^1_b=\Pbb^1$. Choose two positively pathed Riemann spheres with local coordinates
\begin{gather*}
\fk P^a=\big(\Pbb^1;x_0,x_1,\dots,x_N;\eta_0,\eta_1,\dots,\eta_N;\upgamma_0,\upgamma_1,\dots,\upgamma_N\big),\\
\fk P^b=\big(\Pbb^1;y_0,y_1,\dots,y_K;\varpi_0,\varpi_1,\dots,\varpi_K;\updelta_0,\updelta_1,\dots,\updelta_K\big)	
\end{gather*}
where $N,K\geq 1$. So each $\eta_j$ (resp. $\varpi_l$) is a local coordinate at $x_j$ (resp. $y_l$). For each $1\leq j\leq N,1\leq k\leq K$, we assume $\eta_j\in\scr O(W_j)$ (resp. $\varpi_l\in\scr O(M_l)$) where $W_j\subset\Pbb^1$ (resp. $M_l\subset\Pbb^1$) is open, and $\eta_j(W_j)$ (resp. $\varpi_l(M_l)$) is an open disc centered at $0$ with radius $r_j$ (resp. $\rho_l$). We choose  $r_0,\rho_0>0$ such that
\begin{gather*}
\eta_0(W_0)=\mc D_{r_0},\qquad\varpi_0(M_0)=\mc D_{\rho_0}.	
\end{gather*}
We assume
\begin{align*}
r_0\rho_0>1.	
\end{align*}

We assume that $W_0$ does not contain $x_1,\dots,x_N$, and $M_0$ does not contain $y_1,\dots,y_K$. Let
\begin{align*}
F^a=\{x\in W_0:|\eta_0(x)|\leq 1/\rho_0\},\qquad F^b=\{y\in M_0:|\varpi_0(y)|\leq 1/r_0\}.	
\end{align*}
We assume\footnote{Some assumptions in this chapter are similar to those in the previous chapters and are reviewed here for the readers' convenience. As for some other assumptions that were not mentioned in the previous chapters, we enclose them in boxes.}
\begin{gather}\label{eq90}
\boxed{~\begin{array}{c}
F^a\text{ is disjoint from }	W_1,\dots,W_N,\upgamma_0,\upgamma_1,\dots,\upgamma_N\\[0.8ex]
F^b\text{ is disjoint from }M_1,\dots,M_K,\updelta_0,\updelta_1,\dots,\updelta_K
\end{array}~}
\end{gather}
Note that every $\eta_j(\upgamma_j(0))$ and every $\varpi_l(\updelta_l(0))$ are positive by \eqref{eq22}. We also assume
\begin{align}
\boxed{~\eta_0\big(\upgamma_0(0)\big)\varpi_0\big(\updelta_0(0)\big)=1~}	\label{eq54}
\end{align}
so that $\upgamma_0(0)$ and $\updelta_0(0)$ can be glued to the same point.

We let
\begin{align*}
\Sbf^a=\{x_0,x_1,\dots,x_N\},\qquad \Sbf^b=\{y_0,y_1,\dots,y_K\}.	
\end{align*}
Note that $\upgamma_0,\dots,\upgamma_N$ (resp. $\updelta_0,\dots,\updelta_K$) have a common end point $\upgamma_\blt(1)$ (resp. $\updelta_\blt(1)$). We set
\begin{align*}
\Gamma^a=\pi_1\big(\Pbb^1_a\setminus\Sbf^a,\upgamma_\blt(1)\big),\qquad \Gamma^b=\pi_1\big(\Pbb^1_b\setminus\Sbf^b,\updelta_\blt(1)\big).
\end{align*}
Similar to the setting of Sec. \ref{lb22}, we let $\upepsilon_j^a$ (resp. $\upepsilon_l^b$) be a small anticlockwise circle in $W_j$ (resp. $M_l$) around the center $x_j$ (resp. $y_l$) from and to $\upgamma_j(0)$ (resp. $\updelta_l(0)$), and let
\begin{align}
\upalpha_j=\upgamma_j^{-1}\upepsilon_j^a\upgamma_j,\qquad\upbeta_l=\updelta_l^{-1}\upepsilon_l^b\updelta_l.\label{eq56}
\end{align}
We assume
\begin{align}
\boxed{~\Gamma^a=\bk{[\upalpha_1],\dots,[\upalpha_N]},\quad\Gamma^b=\bk{[\upbeta_1],\dots,[\upbeta_K]}~}	\label{eq86}
\end{align}

Recall that $E$ is a finite set.  We fix an action of $\Gamma^a$ on $E$ and another one $\Gamma^b$ on $E$, namely, we fix homomorphisms $\Gamma^a,\Gamma^b\rightarrow\Perm(E)$. We assume these two homomorphisms send each $[\upalpha_j]\in\Gamma^a$ and each $[\upgamma_l]\in\Gamma^b$ to
\begin{align*}
[\upalpha_j]\mapsto g_j,\qquad [\upbeta_l]\mapsto h_l	
\end{align*}
where $g_0,g_1,\dots,g_N,h_0,h_1,\dots,h_M\in\Perm(E)$, and 
\begin{align}\label{eq55}
\boxed{~g_0h_0=1~}
\end{align}

\subsubsection{$\fk P^{a\#b}$ is the sewing of $\fk P^a$ and $\fk P^b$}\label{lb60}

We define a positively $(N+K)$-pathed Riemann sphere with local coordinates
\begin{align*}
\fk P^{a\#b}=\big(&\Pbb^1_{a\#b};x_1,\dots,x_N,y_1,\dots,y_K;\eta_1,\dots,\eta_N,\varpi_1,\dots,\varpi_K;\nonumber\\
&\upgamma_1\upgamma_0^{-1},\dots,\upgamma_N\upgamma_0^{-1},\updelta_1\updelta_0^{-1},\dots,\updelta_K\updelta_0^{-1}\big).	
\end{align*}
as follows. Without the paths, $\fk P^{a\#b}$ is the sewing of $\fk P^a\sqcup\fk P^b$ along the pair of points $x_0,y_0$ using the local coordinates $\eta_0,\varpi_0$. (Cf. Sec. \ref{lb28}.) Here are the details: we remove $F^a$ from the $\Pbb^1_a$, remove $F^b$ from $\Pbb^1_b$, and glue the remaining parts by identifying $W_0\setminus F^a$ and $M_0\setminus F^b$ using the relation
\begin{gather*}
x\in W_0\setminus F^a\text{ equals }y\in M_0\setminus F^b\\
\Updownarrow\\
\eta_0(x)\varpi_0(y)=1.
\end{gather*}
This gives us a new Riemann surface $\Pbb^1_{a\#b}\simeq\Pbb^1$. Since, by our assumption, $x_1,\dots,x_N$, $W_1,\dots,W_N$, and $\upgamma_0,\upgamma_1,\dots,\upgamma_N$ are disjoint from $F^a$, they can be regarded as sets/open subsets/paths of $\Pbb^1_{a\#b}$. The same can be said about $y_1,\dots,y_K$, $M_1,\dots,M_K$, and $\updelta_0,\updelta_1,\dots,\updelta_K$. Moreover, for each $1\leq j\leq N,1\leq k\leq K$, $\eta_j\in\scr O(W_j)$ and $\varpi_l\in\scr O(M_l)$ can be viewed as local coordinates of $\Pbb^1_{a\#b}$ at $x_j,y_l$ respectively.

By \eqref{eq54}, $\upgamma_0(0)$ and $\updelta_0(0)$ become the same point of $\Pbb^1_{a\#b}$ after gluing, which we denote by $\bigvarstar$. We record this definition:
\begin{align*}
\boxed{~\bigvarstar=\upgamma_0(0)=\updelta_0(0)\in\Pbb^1_{a\#b}~}	
\end{align*}
Set
\begin{align*}
\Sbf^{a\#b}=\{x_1,\dots,x_N,y_1,\dots,y_K\}\subset \Pbb^1_{a\#b}.	
\end{align*}
Then $\bigvarstar$ is the common end point of the paths $\upgamma_1\upgamma_0^{-1},\dots,\upgamma_N\upgamma_0^{-1},\updelta_1\updelta_0^{-1},\dots,\updelta_K\updelta_0^{-1}$ in $\Pbb^1_{a\#b}\setminus\Sbf^{a\#b}$.

We define
\begin{align*}
\Gamma^{a\#b}=\pi_1\big(\Pbb^1_{a\#b}\setminus\Sbf^{a\#b},\bigvarstar \big).	
\end{align*}
For each $1\leq j\leq N,1\leq k\leq K$, we let
\begin{gather}
\upalpha_j^\#:=\upgamma_0\upalpha_j\upgamma_0^{-1}=	\upgamma_0\upgamma_j^{-1}\upepsilon_j^a\upgamma_j\upgamma_0^{-1},\qquad\upbeta_l^\#:=\updelta_0\upbeta_l\updelta_0^{-1}=	\updelta_0\updelta_l^{-1}\upepsilon_j^b\updelta_l\updelta_0^{-1},\label{eq60}
\end{gather}
and let $[\upalpha_j^\#],[\upbeta_l^\#]\in\Gamma^{a\#b}$ be their homotopy classes (on $\Pbb^1_{a\#b}\setminus\Sbf^{a\#b}$). By \eqref{eq86}, \eqref{eq55}, and Van Kampen Theorem, $\Gamma^{a\#b}$ is generated by $[\upalpha_j^\#],[\upbeta_l^\#]$ for all $l,k\geq 1$ (note that this is required in the definition of positive pathed Riemann spheres; see \eqref{eq87}), and we can define a unique action 
\begin{align}
\boxed{~\Gamma^{a\#b}\curvearrowright E:[\upalpha_j^\#]\mapsto g_j,	[\upbeta_l^\#]\mapsto h_l~}\label{eq59}
\end{align}

\subsubsection[$\fk X^a$ and $\fk X^b$ are the permutation coverings of $\fk P^a$ and $\fk P^b$]{$\fk X^a$ and $\fk X^b$ are the permutation coverings of $\fk P^a$ and $\fk P^b$ associated to $\Gamma^a,\Gamma^b\curvearrowright E$}

As in Sec. \ref{lb23}, we can construct permutation coverings 
\begin{align*}
\big(\fk X^a,\varphi^a:C^a\rightarrow\Pbb^1_a,\Psi^a_{\upgamma_\blt(1)}\big),\qquad \big(\fk X^b,\varphi^b:C^b\rightarrow\Pbb^1_b,\Psi^b_{\updelta_\blt(1)}\big)
\end{align*}
of $\fk P^a$ and $\fk P^b$ respectively, where
\begin{gather*}
\fk X^a=	(C^a;\Upsilon^a(\bk{g_j}\tipae);\wtd\eta_{j,\tipae}\in\scr O(\wtd W_{j,\tipae})\text{ for all }0\leq j\leq N,\tipae\in E(g_j)),\\
\fk X^b=	(C^b;\Upsilon^b(\bk{h_l}\tipae);\wtd\varpi_{l,\tipae}\in\scr O(\wtd M_{l,\tipae})\text{ for all }0\leq l\leq K,\tipae\in E(h_l)).		
\end{gather*}
We recall some details for the readers' convenience.  As in Thm. \ref{lb4} and Cor. \ref{lb6},  we define $\Psi^a,\Upsilon^a$ for $\fk P^a$ and $\fk X^a$, and define $\Psi^b,\Upsilon^b$ for $\fk P^b$ and $\fk X^b$. For each $0\leq j\leq N$ (resp. $0\leq k\leq K$), $E(g_j)$ (resp. $E(h_l)$) is the set marked points of $\bk{g_j}$-orbits (resp. $\bk{h_l}$-orbits). For each $\tipae\in E(g_j)$ (resp. $\tipae\in E(h_l)$), we let $\wtd W_{j,\tipae}$ (resp. $\wtd M_{l,\tipae}$) be the connected component of $W_j$ (resp. $M_l$) containing $\Upsilon^a(\bk{g_j}\tipae)$ (resp. $\Upsilon^b(\bk{h_l}\tipae)$). Let
\begin{align*}
k_{j,\tipae}=|\bk{g_j}\tipae|\qquad\text{resp.}\qquad\tipak_{l,\tipae}=|\bk{h_l}\tipae|.	
\end{align*}
Then, as in \eqref{eq25}, $\wtd\eta_{j,\tipae}$ (resp. $\wtd\varpi_{l,\tipae}$) is determined by the fact that the diagram
\begin{equation}\label{eq57}
\begin{tikzcd}
	\wtd W_{j,\tipae} \arrow[r, "\wtd\eta_{j,\tipae}","\simeq"'] \arrow[d, "\varphi^a"]
	& \mc D_{(r_j)^{1/k_{j,\tipae}}} \arrow[d, "z^{k_{j,\tipae}}"] \\
	W_j \arrow[r,  "\eta_j","\simeq"']
	&  \mc D_{r_j}
\end{tikzcd}
\qquad \text{resp.}\qquad
\begin{tikzcd}
	\wtd M_{l,\tipae} \arrow[r, "\wtd\varpi_{j,\tipae}","\simeq"'] \arrow[d, "\varphi^b"]
	& \mc D_{(\rho_l)^{1/\tipak_{l,\tipae}}} \arrow[d, "z^{\tipak_{l,\tipae}}"] \\
	M_l \arrow[r,  "\varpi_l","\simeq"']
	&  \mc D_{\rho_l}
\end{tikzcd}	
\end{equation}
commutes for some $r_j>0$ (resp. $\rho_l>0$), and that (similar to \eqref{eq28})
\begin{align}
\wtd\eta_{j,\tipae}\circ\Psi^a_{\upgamma_j}(\tipae)>0\qquad\text{resp.}\qquad \wtd\varpi_{l,\tipae}\circ\Psi^b_{\updelta_l}(\tipae)>0.\label{eq58}	
\end{align}

In addition to the above conditions (which we have assumed before), for the purpose of sewing $\fk X^a$ and $\fk X^b$, we also assume that $\bk{g_0}$-orbits and $\bk{h_0}$-orbits have the same set of marked points, i.e.,
\begin{align}
\boxed{~E(g_0)=E(h_0)~}
\end{align} 
This is possible, since we have assumed $g_0=h_0^{-1}$ (Eq. \eqref{eq55}). Thus, for each $\tipae\in E(g_0)=E(h_0)$, we have
\begin{align*}
k_{0,\tipae}=\tipak_{0,\tipae}.	
\end{align*}

\subsubsection{$\fk X^{a\#b}$ is the sewing of $\fk X^a$ and $\fk X^b$}\label{lb43}

We now sew $\fk X^a\sqcup\fk X^b$ along all the pairs $\Upsilon^a(\bk{g_0}\tipae),\Upsilon^b(\bk{h_0}\tipae)$ (for all $\tipae\in E(g_0)=E(h_0)$) to obtain
\begin{align*}
\fk X^{a\#b}=	\big(&C^{a\#b};\Upsilon^a(\bk{g_j}\tipae),\Upsilon^b(\bk{h_l}\tipae');\wtd\eta_{j,\tipae}\text{ and }\wtd\varpi_{l,\tipae'}\nonumber\\
&\text{ for all }1\leq j\leq N,1\leq l\leq K,\tipae\in E(g_j),\tipae'\in E(h_l)\big).	
\end{align*}
Thus, for each $\tipae\in E(g_0)=E(h_0)$, we remove closed subsets
\begin{align*}
\wtd F_\tipae^a:=\big\{\wtd x\in\wtd W_{0,\tipae}:|\wtd\eta_{0,\tipae}(\wtd x)|^{k_{0,\tipae}}\equiv|\eta_0\circ\varphi^a(\wtd x)|\leq 1/\rho_0\big\}
\end{align*}
from $\wtd W_{0,\tipae}$ and 
\begin{align*}
\wtd F_\tipae^b:=\big\{\wtd y\in\wtd M_{0,\tipae}:|\wtd\varpi_{0,\tipae}(\wtd y)|^{k_{0,\tipae}}\equiv|\varpi_0\circ\varphi^b(\wtd y)|\leq 1/r_0\big\},
\end{align*}
from $\wtd M_{0,\tipae}$, and glue the remaining parts such that
\begin{gather*}
\wtd x\in \wtd W_{0,\tipae}\setminus \wtd F_\tipae^a\text{ equals }\wtd y\in \wtd M_{0,\tipae}\setminus F_\tipae^b\\
\Updownarrow\\
\wtd \eta_{0,\tipae}(\wtd x)\wtd\varpi_{0,\tipae}(\wtd y)=1.
\end{gather*}
Thus we obtained $C^{a\#b}$ as the sewing of $C^a$ and $C^b$.

This gluing process is compatible with $\varphi^a$ and $\varphi^b$. Thus, we obtain a holomorphic surjective
\begin{align*}
\varphi^{a\#b}:C^{a\#b}\rightarrow \Pbb^1_{a\#b}	
\end{align*}
such that
\begin{gather*}
\varphi^{a\#b}=\varphi^a\text{ when restricted to }C^a\Big\backslash\bigcup_{\tipae\in E(g_0)}\wtd F_\tipae^a,	\\
\varphi^{a\#b}=\varphi^b\text{ when restricted to }C^b\Big\backslash\bigcup_{\tipae\in E(h_0)}\wtd F_\tipae^b.
\end{gather*}

Since $\Upsilon^a(\bk{g_j}\tipae),\Upsilon^b(\bk{h_l}\tipae')$ (for all $j\geq 1,k\geq 1,\tipae\in E(g_j),\tipae'\in E(h_j)$) are outside any $\wtd F^a_e,\wtd F^b_e$ (where $e\in E(g_0)$), we can define these points as marked points of $\fk X^{a\#b}$. Thus, we can also view $\wtd\eta_{j,\tipae},\wtd\varpi_{l,\tipae'}$ as the local coordinates of $\fk X^{a\#b}$ at these two points points, defined on $\wtd W_{j,\tipae},\wtd M_{l,\tipae'}\subset C^{a\#b}$.

\subsubsection[$\fk X^{a\#b}$ is the permutation covering of $\fk P^{a\#b}$]{$\fk X^{a\#b}$ is the permutation covering of $\fk P^{a\#b}$ associated to $\Gamma^{a\#b}\curvearrowright E$ and $E(g_\blt),E(h_\blt)$}

Since $\upgamma_0(0)$ and $\updelta_0(0)$ are outside $F^a,F^b$ (by \eqref{eq90}), their preimages $\Psi_{\upgamma_0}^a(e)$ and $\Psi_{\updelta_0}^b(e)$ (for all $e\in E$) are outside $\wtd F^a_\tipae,\wtd F^b_\tipae$ (for every $\tipae\in E(g_0)=E(h_0)$). Moreover:

\begin{lm}
$\Psi_{\upgamma_0}^a(e)$ and $\Psi_{\updelta_0}^b(e)$ are the same point on $C^{a\#b}$.
\end{lm}

\begin{proof}
Write $e=g_0^m\tipae=h_0^{-m}\tipae$ for some $m\in\Zbb,\tipae\in E(g_0)=E(h_0)$. Let $\wtd x=\Psi_{\upgamma_0}^a(e)$ and $\wtd y=\Psi_{\updelta_0}^b(e)$. Set $k=k_{0,\tipae}$. We have
\begin{align*}
\wtd x=\Psi_{\upgamma_0}^a([\upalpha_0]^m\tipae)\xlongequal{\eqref{eq7}}\Psi^a_{\upgamma_0\upalpha_0^m}(\tipae)\xlongequal{\eqref{eq56}}\Psi^a_{(\upepsilon_0^a)^m\upgamma_0}(\tipae).	
\end{align*}
So $\wtd x$ is the initial point of the lift of $(\upepsilon_0^a)^m$ to $C^a$ ending at $\Psi^a_{\upgamma_0}(\tipae)$. Since $\eta_0\circ\upepsilon_0^a$ is  an anticlockwise circle going by $2\pi$, $\wtd\eta_{0,\tipae}$ sends the lift of $\upepsilon_0^a$ to the anticlockwise arc going by $2\pi/k$. Therefore
\begin{align*}
\wtd\eta_{0,\tipae}(\wtd x)=\tipaomega_k^m\cdot\wtd\eta_{0,\tipae}\big(\Psi_{\upgamma_0}^a(\tipae)\big).	
\end{align*}
A similar argument shows
\begin{align*}
\wtd\varpi_{0,\tipae}(\wtd y)=\tipaomega_k^{-m}\cdot\wtd\varpi_{0,\tipae}\big(\Psi_{\updelta_0}^b(\tipae)\big).	
\end{align*}

Set $d=\eta_0(\upgamma_0(0))=\varpi_0(\updelta_0(0))^{-1}$ (cf. \eqref{eq54}), which is positive. Since $\Psi_{\upgamma_0}^a(\tipae)\in(\varphi^a)^{-1}(\upgamma_0(0))$, by \eqref{eq57},
\begin{align*}
\wtd\eta_{0,\tipae}\big(\Psi_{\upgamma_0}^a(\tipae)\big)^k=\eta_0\circ\varphi^a\big(\Psi_{\upgamma_0}^a(\tipae)\big)=\eta_0\circ\upgamma_0(0)=d.
\end{align*}
Thus, by \eqref{eq58}, we have $\wtd\eta_{0,\tipae}\big(\Psi_{\upgamma_0}^a(\tipae)\big)=d^{1/k}$. A similar argument shows $\wtd\varpi_{0,\tipae}\big(\Psi_{\updelta_0}^b(\tipae)\big)=d^{-1/k}$. Therefore $\wtd\eta_{0,\tipae}(\wtd x)=1/\wtd\varpi_{0,\tipae}(\wtd y)$, which shows that $\wtd x$ and $\wtd y$ are identical after sewing.
\end{proof}

It follows that the preimage of $\bigvarstar=\upgamma_0(0)=\updelta_0(0)$ under $\varphi^{a\#b}$ consists of 
\begin{align}
\Psi^{a\#b}_\bigvarstar(e):= \Psi_{\upgamma_0}^a(e)=\Psi_{\updelta_0}^b(e)\label{eq61}	
\end{align}
for all $e\in E$.  This gives us a bijection
\begin{align*}
\Psi_\bigvarstar^{a\#b}:E\rightarrow(\varphi^{a\#b})^{-1}(\bigvarstar).	
\end{align*}
The following lemma is an easy consequence of Thm. \ref{lb4}-(b).

\begin{lm}\label{lb33}
Let $\uplambda$ be a path in $\Pbb^1_{a\#b}\setminus\Sbf^{a\#b}$ ending at $\bigvarstar$. Choose $e\in E$, and let $\wtd\uplambda$ be the lift of $\uplambda$ to $C^{a\#b}$ ending at $\Psi_\bigvarstar^{a\#b}(e)$. If $\uplambda$ is in $\Pbb^1_a\backslash F^a$ (resp. in $\Pbb^1_b\backslash F^b$), then the initial point of $\wtd\uplambda$ is $\Psi^a_{\uplambda\upgamma_0}(e)$ (resp. $\Psi^b_{\uplambda\updelta_0}(e)$).
\end{lm}

The main result of this section is the following theorem.

\begin{thm}\label{lb46}
The bijection $\Psi^{a\#b}_\bigvarstar$ is $\Gamma^{a\#b}$-covariant, and 
\begin{align*}
\big(\fk X^{a\#b},\varphi^{a\#b}:C^{a\#b}\rightarrow\Pbb^1_{a\#b},\Psi^{a\#b}_\bigvarstar\big)	
\end{align*}
is a permutation covering of $\fk P^{a\#b}$ associated to the action $\Gamma^{a\#b}\curvearrowright E$ (defined by \eqref{eq59}) and $E(g_j),E(h_l)$ (for all $1\leq j\leq N$ and $1\leq l\leq K$). 

Moreover, let $\Psi^{a\#b}$ be the trivialization determined by $\Psi^{a\#b}_\bigvarstar$, and define $\Upsilon^{a\#b}$ as in Cor. \ref{lb6}.  Then for any path $\uplambda$ in $\Pbb^1_{a\#b}\setminus\Sbf^{a\#b}$ ending at $\bigvarstar$, if $\uplambda$ is in $\Pbb^1_a\backslash F^a$ (resp. in $\Pbb^1_b\backslash F^b$), then
\begin{align}
\Psi_\uplambda^{a\#b}(e)=\Psi^a_{\uplambda\upgamma_0}(e)\qquad\text{resp.}\qquad\Psi_\uplambda^{a\#b}(e)=\Psi^b_{\uplambda\updelta_0}(e).\label{eq62}
\end{align}
For any $1\leq j\leq N$ and $1\leq l\leq K$, we have
\begin{align}
\Upsilon^{a\#b}(\bk{g_j}e)=\Upsilon^a(\bk{g_j}e),\qquad 	\Upsilon^{a\#b}(\bk{h_l}e)=\Upsilon^b(\bk{h_l}e).\label{eq63}
\end{align}
\end{thm}

\begin{proof}
Note that for any $\upalpha_j^\#$ and $\upbeta_l^\#$ defined by \eqref{eq60} (where $j,l\geq 1$), applying  Lemma \ref{lb33} to $\uplambda=\upalpha_j^\#,\upbeta_l^\#$, we see that the lift of $\upalpha_j^\#$ (resp. $\upbeta_l^\#$) to $C^{a\#b}$ ending at $\Psi_\bigvarstar^{a\#b}(e)$ must start from $\Psi^a_{\upgamma_0\upalpha_j}(e)=\Psi^a_{\upgamma_0}([\upalpha_j]e)=\Psi^{a\#b}_\bigvarstar(g_je)$ (resp. $\Psi^{a\#b}_\bigvarstar(h_le)$). Thus, for any path $\upmu$ in $\Pbb^1_{a\#b}\setminus\Sbf^{a\#b}$ from and to $\bigvarstar$, which is homotopic to a product of powers of $\upalpha_j^\#,\upbeta_l^\#$ (for all $j,l\geq 1$), the initial point of the lift of $\upmu$ to $C^{a\#b}$ ending at $\Psi^{a\#b}_\bigvarstar(e)$ is  $\Psi^{a\#b}_\bigvarstar(ge)$, where $g$ is the same product of powers of $g_j,h_l$, which is $[\upmu]$. So $\Psi^{a\#b}_\bigvarstar$ is $\Gamma^{a\#b}$-covariant. 

By the construction of local coordinates, it is clear that  $\fk X^{a\#b}$ is the permutation covering of $\fk P^{a\#b}$ associated to the action of $\Gamma^{a\#b}$ and all $E(g_j),E(h_l)$ (where $j,l\geq 1$).

For any path $\uplambda$ in $\Pbb^1_a\setminus F^a$ ending at $\bigvarstar$, by Thm. \ref{lb4}-(b) (applied to $\Psi^{a\#b}$), we know that $\Psi_\uplambda^{a\#b}(e)$ is the initial point of the lift $\wtd\uplambda$ of $\uplambda$ to $C^{a\#b}$ ending at $\Psi^{a\#b}_\bigvarstar(e)=\Psi^a_{\upgamma_0}(e)$. Since $\wtd\uplambda$ is also the lift of $\uplambda$ to $C^a$  ending at $\Psi^a_{\upgamma_0}(e)$,  $\Psi_\uplambda^{a\#b}(e)=\wtd\uplambda(0)$ must equal $\Psi^a_{\uplambda\upgamma_0}(e)$ by Thm. \ref{lb4}-(b).   This proves the first half of \eqref{eq62}. The second half follows from a similar argument. 

By Cor. \ref{lb6} and the definition of the paths of $\fk P^{a\#b}$, $\Upsilon^{a\#b}(\bk{g_j}e)$ and $\Psi^{a\#b}_{\upgamma_j\upgamma_0^{-1}}(\bk{g_j}e)$ are in the same connected component of $(\varphi^{a\#b})^{-1}(W_j)=(\varphi^a)^{-1}(W_j)$. Similarly, $\Upsilon^a(\bk{g_j}e)$ and  $\Psi^a_{\upgamma_j}(\bk{g_j}e)$ (which equals $\Psi^{a\#b}_{\upgamma_j\upgamma_0^{-1}}(\bk{g_j}e)$ by \eqref{eq62}) are in the same connected component of $(\varphi^a)^{-1}(W_j)$. This proves the first half of \eqref{eq63}. The second half can be proved in the same way.
\end{proof}

\subsection{Sewing and factorization of permutation-twisted conformal blocks}

In this section, we assume $\Vbb$ is $C_2$-cofinite. Then so is $\Ubb=\Vbb^{\otimes E}$. For each $0\leq j\leq N,0\leq l\leq K$, and for each $\tipae\in E(g_j)$ (resp. $\tipae\in E(h_l)$), we associate a \emph{finitely-generated} $\Vbb$-module $\Wbb_{j,\tipae}$ (resp. $\Mbb_{l,\tipae}$) to the marked point $\Upsilon^a(\bk{g_j}\tipae)$ of $\fk X^a$ (resp. $\Upsilon^b(\bk{h_l}\tipae)$ of $\fk X^b$). Then we have a $g_j$- resp. $h_l$-twisted $\Ubb$-module
\begin{gather*}
\mc W_j=\bigotimes_{\tipae\in E(g_j)}\Wbb_{j,\tipae}\qquad \text{resp.}\qquad 	\mc M_l=\bigotimes_{\tipae\in E(h_l)}\Mbb_{l,\tipae}
\end{gather*}
defined as in Thm. \ref{lb13} associated to the marked point $x_j$ of $\fk P^a$ (resp. $y_l$ of $\fk P^b$). As usual, we set
\begin{align*}
k_{j,\tipae}=|\bk{g_j}\tipae|,\qquad \tipak_{l,\tipae}=|\bk{h_l}\tipae|.	
\end{align*}
According to \eqref{eq73}, the $\wtd L_0^{g_j}$- (resp. $\wtd L_0^{h_j}$-)grading and the $\wtd L_0$-grading are related by the fact that for each $n\in\frac 1{|g_j|}\Nbb$,
\begin{gather}\label{eq74}
\begin{array}{c}
\mc W_j(n)=\bigoplus\limits_{\sum_{\tipae\in E(g_j)}n_{\tipae}/k_{j,\tipae}=n}\bigg(\bigotimes_{\tipae\in E(g_j)}\Wbb_{j,\tipae}(n_\tipae)\bigg)\\[1.5ex]
\mc M_l(n)=\bigoplus\limits_{\sum_{\tipae\in E(h_l)}n_{\tipae}/\tipak_{l,\tipae}=n}\bigg(\bigotimes_{\tipae\in E(h_l)}\Mbb_{l,\tipae}(n_\tipae)\bigg)
\end{array}	
\end{gather}
where all $n_\tipae$ are in $\Nbb$.

We assume that for each $\tipae\in E(g_0)=E(h_0)$,
\begin{align*}
\boxed{~\Mbb_{0,\tipae}=\Wbb_{0,\tipae}'~}
\end{align*}
i.e., $\Wbb_{0,\tipae}$ and $\Mbb_{0,\tipae}$ are the contragredient $\Vbb$-modules of each other.  Then, by Thm. \ref{lb42}, $\mc W_0$ and $\mc M_0$ are the contragredient twisted $\Ubb$-modules of each other.

Let 
\begin{gather*}
\mc W_\blt=\bigotimes_{0\leq j\leq N}\mc W_j,\qquad 	\mc M_\blt=\bigotimes_{0\leq j\leq N}\mc M_j,\\
\mc W_{\blt\backslash 0}=\bigotimes_{1\leq j\leq N}\mc W_j,\qquad 	\mc M_{\blt\backslash 0}=\bigotimes_{1\leq j\leq N}\mc M_j.
\end{gather*}
Let $\CB_{\fk P^a}(\mc W_\blt)$ resp. $\CB_{\fk P^b}(\mc M_\blt)$ be the space of conformal blocks associated to $\fk P^a$ and the twisted $\Ubb$-modules $\mc W_\blt$ (resp. $\fk P^b$ and $\mc M_\blt$).  For each $\uppsi^a\in\CB_{\fk P^a}(\mc W_\blt)$ and $\uppsi^b\in\CB_{\fk P^b}(\mc M_\blt)$, we define, for every $n\in \frac 1{|g_0|}\Nbb=\frac{1}{|h_0|}\Nbb$, a linear functional
\begin{align*}
\uppsi^{a\#b}_n:	\mc W_{\blt\backslash 0}\otimes\mc M_{\blt\backslash 0}\rightarrow \Cbb
\end{align*}
such that for each $w\in \mc W_{\blt\backslash 0}$ and $m\in \mc M_{\blt\backslash 0}$, $\uppsi^{a\#b}_n(w\otimes m)$ is the contraction of the linear functional (noting $\mc W_0'(n)=\mc W_0(n)^*$)
\begin{align*}
\uppsi^a(\cdot \otimes w)\uppsi^b(\cdot\otimes m)\Big|_{\mc W_0(n)\otimes\mc W_0(n)^*}:	\mc W_0(n)\otimes\mc W_0(n)^*\rightarrow\Cbb.
\end{align*}
We say that  \textbf{the sewing $\uppsi^{a\#b}$ converges $q$-absolutely} if there exists $R>1$ such that for each $w\in\mc W_{\blt\backslash 0},m\in\mc M_{\blt\backslash 0}$, 
\begin{align*}
\sum_{n\in\frac 1{|g_0|}\Nbb}\big|\uppsi_n^{a\#b}(w\otimes m)\big|R^n<+\infty.
\end{align*} 
If so, we define
\begin{gather*}
\uppsi^{a\#b}:	\mc W_{\blt\backslash 0}\otimes\mc M_{\blt\backslash 0}\rightarrow \Cbb\\
w\otimes m\mapsto \sum_{\frac 1{|g_0|}\Nbb}\uppsi^{a\#b}_n(w\otimes m).
\end{gather*}

We have $\mc W_\blt=\Wbb_{\blt,\blt}$, $\mc M_\blt=\Mbb_{\blt,\blt}$, $\mc W_{\blt\backslash0}=\Wbb_{\blt\backslash0,\blt}$, $\mc M_{\blt\backslash0}=\Mbb_{\blt\backslash0,\blt}$ where
\begin{gather*}
\Wbb_{\blt,\blt}=\bigotimes_{0\leq j\leq N}\bigotimes_{\tipae\in E(g_j)}\Wbb_{j,\tipae},\qquad 	\Mbb_{\blt,\blt}=\bigotimes_{0\leq l\leq N}\bigotimes_{\tipae\in E(h_l)}\Wbb_{l,\tipae},\\
\Wbb_{\blt\backslash0,\blt}=\bigotimes_{1\leq j\leq N}\bigotimes_{\tipae\in E(g_j)}\Wbb_{j,\tipae},\qquad 	\Mbb_{\blt\backslash0,\blt}=\bigotimes_{1\leq l\leq N}\bigotimes_{\tipae\in E(h_l)}\Wbb_{l,\tipae}.
\end{gather*}
Let $\CB_{\fk X^a}(\Wbb_{\blt,\blt})$ resp. $\CB_{\fk X^b}(\Mbb_{\blt,\blt})$ be the space of conformal blocks associated to $\fk X^a$ and the $\Vbb$-modules $\Wbb_{\blt,\blt}$ (resp. $\fk X^b$ and $\Mbb_{\blt,\blt}$). Then by Thm. \ref{lb14},
\begin{gather*}
\CB_{\fk X^a}(\Wbb_{\blt,\blt})=\CB_{\fk P^a}(\mc W_\blt),\qquad \CB_{\fk X^b}(\Mbb_{\blt,\blt})=\CB_{\fk P^b}(\mc M_\blt).	
\end{gather*}
Consider $\uppsi^a,\uppsi^b$ as elements of $\CB_{\fk P^a}(\mc W_\blt),\CB_{\fk P^b}(\mc M_\blt)$, which we denote by $\upphi^a,\upphi^b$ in order to follow the notation in Thm. \ref{lb14}. Then $\upphi^a\cdot\upphi^b:\Wbb_{\blt,\blt}\otimes\Mbb_{\blt,\blt}\rightarrow\Cbb$ is a conformal block associated to $\fk X^a\sqcup\fk X^b$. Corresponding to the geometric sewing process in Subsec. \ref{lb43}, we can define the algebraic sewing
\begin{align*}
\scr S(\upphi^a\cdot \upphi^b):\Wbb_{\blt\setminus0,\blt}\otimes\Mbb_{\blt\backslash0,\blt}\rightarrow\Cbb	
\end{align*}
as in Subsec. \ref{lb44}, which converges $q$-absolutely by Thm. \ref{lb39}. Moreover, let
\begin{align*}
\CB_{\fk X^{a\#b}}(\Wbb_{\blt\backslash0,\blt}\otimes\Mbb_{\blt\backslash0,\blt})	
\end{align*}
be the space of conformal blocks associated to $\fk X^{a\#b}$ and the corresponding $\Vbb$-modules. For each $1\leq j\leq N$ and $\tipae\in E(g_j)$,   $\Wbb_{j,\tipae}$ is associated to $\Upsilon^{a\#b}(\bk{g_j}\tipae)=\Upsilon^a(\bk{g_j}\tipae)$, and for each $1\leq l\leq K$ and $\tipae\in E(h_l)$, $\Mbb_{l,\tipae}$ is associated to $\Upsilon^{a\#b}(\bk{h_l}\tipae)=\Upsilon^b(\bk{h_l}\tipae)$. (Recall \eqref{eq63}.) Then $\scr S(\upphi^a\cdot \upphi^b)\in\CB_{\fk X^{a\#b}}(\Wbb_{\blt\backslash0,\blt}\otimes\Mbb_{\blt\backslash0,\blt})$ by Thm. \ref{lb37}.

Let $\CB_{\fk P^{a\#b}}(\mc W_{\blt\backslash0}\otimes\mc M_{\blt\backslash0})$ denote the space of conformal blocks associated to $\fk P^{a\#b}$ and the corresponding twisted $\Ubb$-modules. $g_j,\mc W_j$ are associated to $x_j$, and $h_l,\mc M_j$ to $y_l$ (where $1\leq j\leq N,1\leq l\leq K$).

\begin{thm}\label{lb48}
Assume the setting of Sec. \ref{lb59}. Assume that $\Vbb$ is $C_2$-cofinite. Assume that $\Wbb_{j,\tipae}$ (for each $0\leq j\leq N,\tipae\in E(g_j)$) and $\Mbb_{l,\tipae}$ (for each $0\leq l\leq K,\tipae\in E(h_l)$) are finitely-generated $\Vbb$-modules. Choose $\uppsi^a\in\CB_{\fk P^a}(\mc W_\blt)$ and $\uppsi^b\in\CB_{\fk P^b}(\mc W_\blt)$, where the same linear functionals are denoted by $\upphi^a\in\CB_{\fk X^a}(\Wbb_{\blt,\blt})$ and $\upphi^b\in\CB_{\fk X^b}(\Mbb_{\blt,\blt})$. Then 
\begin{enumerate}[label=(\arabic*)]
\item $\uppsi^{a\#b}$ converges $q$-absolutely to an element of $\CB_{\fk P^{a\#b}}(\mc W_{\blt\backslash0}\otimes\mc M_{\blt\backslash0})$.
\item $\uppsi^{a\#b}$ equals $\scr S(\upphi^a\cdot \upphi^b)$ as linear functionals.
\end{enumerate}
\end{thm}

\begin{proof}
By the definition of $\uppsi^{a\#b}$ and $\scr S(\upphi^a\cdot \upphi^b)$, and by \eqref{eq74}, $\uppsi^{a\#b}$ converges $q$-absolutely to $\scr S(\upphi^a\cdot \upphi^b)$, which is an element of $\CB_{\fk X^{a\#b}}(\Wbb_{\blt\backslash0,\blt}\otimes\Mbb_{\blt\backslash0,\blt})$. By Thm. \ref{lb46}, $\fk X^{a\#b}$ is a permutation covering of $\fk P^{a\#b}$ associated to the action \eqref{eq59} and the sets $E(g_j),E(h_l)$ (where $j,l>0$) of marked points of $\bk{g_j}$- and $\bk{h_l}$-orbits. By Thm. \ref{lb14}, we have
\begin{align}
\CB_{\fk X^{a\#b}}(\Wbb_{\blt\backslash0,\blt}\otimes\Mbb_{\blt\backslash0,\blt})=\CB_{\fk P^{a\#b}}(\mc W_{\blt\backslash0}\otimes\mc M_{\blt\backslash0}).\label{eq75}
\end{align}
This finishes the proof.
\end{proof}

\begin{rem}
Relations \eqref{eq63} are necessary for the above results: they tell us that the association of $\Vbb$-modules to the marked points of $\fk X^{a\#b}$ determined by sewing $\fk X^a,\fk X^b$ agrees with the one determined by the permutation covering of $\fk X^{a\#b}$.
\end{rem}

As an application of Thm. \ref{lb48}, we prove a sewing-factorization theorem for genus-$0$ permutation-twisted conformal blocks.

Let $\mc E$ be a complete list of irreducible $\Vbb$-modules. (Cf. the paragraph containing \eqref{eq71}.) Define $\mc E_{0,\blt}$ to be the set consisting of
\begin{align*}
\mc X_0=\otimes_{\tipae\in E(g_0)}\Xbb_{0,\tipae}	
\end{align*}
where each $\Xbb_{0,\tipae}$ is in $\mc E$. Set
\begin{align*}
\mc X_0'=\otimes_{\tipae\in E(g_0)}\Xbb_{0,\tipae}'.	
\end{align*}	
Consider $\mc X_0$ and $\mc X_0'$  as mutually contragredient twisted $\Ubb$-modules (cf. Thm. \ref{lb42}). Note that by Prop. \ref{lb41} and Thm. \ref{lb47}, if $\Vbb$ is also rational, then $\mc E_{0,\blt}$ is a finite and complete list of irreducible $g_0$-twisted $\Ubb$-modules.

\begin{thm}\label{lb62}
Assume that $\Vbb$ is CFT-type (i.e. $\Vbb(0)=\Cbb\id$), $C_2$-cofinite, and rational. Choose finitely-generated $\Vbb$-modules $\Wbb_{j,\tipae}$ (for each $1\leq j\leq N,\tipae\in E(g_j)$) and $\Mbb_{l,\tipae}$ (for each $1\leq l\leq K,\tipae\in E(h_l)$). Then the linear map
\begin{gather*}
\bigoplus_{\mc X_0\in\mc E_{0,\blt}}	\CB_{\fk P^a}(\mc X_0\otimes\mc W_{\blt\backslash0})\otimes 	\CB_{\fk P^b}(\mc X_0'\otimes\mc M_{\blt\backslash0})\rightarrow\CB_{\fk P^{a\#b}}(\mc W_{\blt\backslash0}\otimes \mc M_{\blt\backslash0})\\
\bigoplus_{\mc X_0} \uppsi^a_{\mc X_0}\otimes \uppsi^b_{\mc X_0}\mapsto \sum_{\mc X_0} \uppsi^{a\#b}_{\mc X_0}
\end{gather*}
(where $\uppsi^{a\#b}_{\mc X_0}$ is the sewing of $\uppsi^a_{\mc X_0}$ and $\uppsi^b_{\mc X_0}$) is bijective.
\end{thm}

\begin{proof}
By Thm. \ref{lb14} and \ref{lb48}, the above linear map is equivalent to a sewing map of spaces of untwisted conformal blocks. This linear map is bijective by Thm. \ref{lb40}.
\end{proof}

\section{Applications}

\subsection{Twisted intertwining operators}

In this section, we assume the setting at the beginning of Sec. \ref{lb22}, namely, $G$ is a general finite automorphism automorphism group of a VOA $\Ubb$. We assume that $1$ has argument $0$, and that in general $e^{\im t}$ has argument $t$. If the arguments of $z_1,z_2\in\Cbb^\times$ are chosen, we assume the argument of $z_1z_2$ is $\arg(z_1z_2)=\arg z_1+\arg z_2$.

\begin{df}\label{lb51}
Let $g_1,g_2\in G$ and $g_3=g_1g_2$. Let $\mc W_1,\mc W_2,\mc W_3$ be respectively $g_1$-,  $g_2$-, $g_3$-twisted $\Vbb$-modules. A type $\mc W_3\choose\mc W_1\mc W_2$-\textbf{intertwining operator} is an operation $\mc Y$ that associates to each $w_1\in\mc W_1,w_2\in\mc W_2,w_3'\in\mc W_3'$ a multivalued homolorphic function
\begin{align*}
z\in\Cbb^\times\mapsto \bk{\mc Y(w_1,z)w_2,w_3'}	
\end{align*}
(i.e. a holomorphic function which depends on $z$ as well as its argument $\arg z$) depending linearly on $w_1,w_2,w_3'$, such that the following conditions are satisfied for every $w_1\in\mc W_1,w_2\in\mc W_2,w_3'\in\mc W_3'$. Consider $\mc Y(w_1,z)$ as a linear map from $\mc W_2$ to $(\mc W_3')^*$.
\begin{enumerate}
\item ($L_{-1}$-derivative)	For each $z\in\Cbb^\times$,
\begin{align}
\frac d{dz}\bk{\mc Y(w_1,z)w_2,w_3'}=\bk{\mc Y(L_{-1}w_1,z)w_2,w_3'}	\label{eq81}
\end{align}
\item (Analytic Jacobi identity)\footnote{A similar property is called the duality property in \cite[Def. 4.1]{Hua18}.}  For each $u\in\Ubb$, and for each $z\neq \tipaz$ in $\Cbb^\times$ with chosen $\arg z,\arg \tipaz,\arg(z-\tipaz)$, the following series of single-valued functions of $\log z,\log z,\log(z-\tipaz)$
\begin{gather}
	\bk{Y^{g_3}(u,z)\mc Y(w_1,\tipaz)w_2,w_3'}:=\sum_{n\in\frac 1{|g_3|}\Nbb}	\bk{Y^{g_3}(u,z)P_n^{g_3}\mc Y(w_1,\tipaz)w_2,w_3'}\label{eq76}\\
	\bk{\mc Y(w_1,\tipaz)Y^{g_2}(u,z)w_2,w_3'}:=\sum_{n\in\frac 1{|g_2|}\Nbb}	\bk{\mc Y(w_1,\tipaz)P_n^{g_2}Y^{g_2}(u,z)w_2,w_3'}\label{eq77}\\
	\bk{\mc Y(Y^{g_1}(u,z-\tipaz)w_1,\tipaz)w_2,w'_3}:=\sum_{n\in\frac 1{|g_1|}\Nbb}	\bk{\mc Y(P_n^{g_1}Y^{g_1}(u,z-\tipaz)w_1,\tipaz)w_2,w'_3}\label{eq78}
\end{gather}
converge a.l.u.  on $|z|>|\tipaz|$, $|z|<|\tipaz|$, $|z-\tipaz|<|\tipaz|$ respectively (in the sense of \eqref{eq10}). Moreover, for any fixed $\tipaz\in\Cbb^\times$ with chosen argument $\arg \tipaz$, let $R_{\tipaz}=\{t\tipaz:t\in(0,1)\cup (1,+\infty)\}$. For any $z\in R_{\tipaz}$, we assume that 
\begin{gather}
\begin{array}{c}
\arg z=\arg\tipaz,\\[0.7ex]
\arg(z-\tipaz)=\arg \tipaz\text{ or }\arg \tipaz-\pi,
\end{array}	\label{eq82}
\end{gather}
where second equality depends on whether $|z|>|\tipaz|$ or $|z|<|\tipaz|$. Then the above three expressions \eqref{eq76}-\eqref{eq78}, considered as functions of $z$ defined on $R_{\tipaz}$ satisfying the three mentioned inequalities respectively, can be analytically continued to the same holomorphic function on the simply-connected open set
\begin{align*}
\Upsigma_\tipaz=\Cbb\setminus\{\im t\tipaz,\tipaz+\im t\tipaz:t\geq 0\},	
\end{align*}
which can furthermore be extended to a multivalued holomorphic function $f_{\tipaz}(z)$ on $\Cbb^\times\setminus\{\tipaz\}$ (i.e., a holomorphic function on the universal cover of $\Cbb^\times\setminus\{\tipaz\}$).
\end{enumerate}
\end{df}

\begin{rem}\label{lb55}
In the above analytic Jacobi identity, if we assume $u$ is fixed by $G$, then $f_\tipaz(z)$ is single-valued. In particular, if we choose $u$ to be the conformal vector of $\Ubb$, and apply the residue theorem to $zf_\tipaz(z)dz$, we obtain
\begin{align*}
	&\bk{\mc Y(w_1,\tipaz)w_2,L_0w_3'}-\bk{\mc Y(w_1,\tipaz)L_0w_2,w_3'}\nonumber\\
	=&	\bk{\mc Y(L_0w_1,\tipaz)w_2,w_3'}+\tipaz\bk{\mc Y(L_{-1}w_1,\tipaz)w_2,w_3'}.
\end{align*}
Thus, assuming the analytic Jacobi identity, condition 1  of Def. \ref{lb51}  (the $L_{-1}$-derivative) is equivalent to
\begin{align}
	&\bk{\mc Y(w_1,\tipaz)w_2,L_0w_3'}-\bk{\mc Y(w_1,\tipaz)L_0w_2,w_3'}\nonumber\\
	=&	\bk{\mc Y(L_0w_1,\tipaz)w_2,w_3'}+\tipaz\frac d{d\tipaz}\bk{\mc Y(w_1,\tipaz)w_2,w_3'}\label{eq79}.
\end{align}	
If we apply the residue theorem to $f_\tipaz(z)dz$, we get
\begin{align}
\bk{L_{-1}\mc Y(w_1,\tipaz)w_2,w_3'}-\bk{\mc Y(w_1,\tipaz)L_{-1}w_2,w_3'}=\bk{\mc Y(L_{-1}w_1,\tipaz)w_2,w_3'}.\label{eq98}	
\end{align}
\end{rem}

\begin{rem}
If $\mc W$ is a $g$-twisted $\Ubb$-module, then its vertex operator $Y^g$ clearly defines a type $\mc W\choose\Ubb\mc W$-intertwining operator.
\end{rem}

\begin{rem}
Since $[L_0,\wtd L_0^g]=0$ by \eqref{eq12}, $L_0$ preserves each $\mc W(n)$. Since $\mc W(n)$ is finite-dimensional, we have a decomposition $L_0=\Lss+\Lni$ where $[\Lss,\Lni]=0$, $\Lss$ is diagonal, and for every $w\in\mc W$, $\Lni^kw=0$ for sufficiently large $k$. Thus we can define $z^{L_0}w\in\mc W[\log z]\{z\}$	by
\begin{align*}
z^{L_0}w=z^\Lss\sum_{k\in\Nbb}\frac{\Lni^k w}{k!}(\log z)^k.
\end{align*}
\end{rem}

\begin{pp}\label{lb56}
Let $\mc Y$ be as in Def. \ref{lb51}. Then for each $\xi,z\in\Cbb^\times$ with chosen $\arg\xi,\arg z$, and for each $w_1\in\mc W_1,w_2\in\mc W_2,w_3'\in\mc W_3'$,
\begin{align}
\bk{\mc Y(w_1,z\xi)w_2,w_3'}=\bk{\mc Y(z^{-L_0}w_1,\xi)z^{-L_0}w_2,z^{L_0}w_3'}.	\label{eq84}
\end{align}
\end{pp}

\begin{proof}
$f(z)=\bk{\mc Y(z^{L_0}w_1,z\xi)z^{L_0}w_2,z^{-L_0}w_3'}$ is a multivalued holomorphic function of $z\in\Cbb^\times$. Using \eqref{eq79}, it is easy to see that $\frac{d}{dz}f(z)=0$. Thus, $f(z)=f(1)$, which is equivalent to \eqref{eq84}.
\end{proof}

\begin{rem}
Set $\xi=1$ in \eqref{eq84}. If we assume $w_1,w_2,w_3'$ are eigenvectors of $\Lss$ with eigenvalues $\wt w_1,\wt w_2,\wt w_3'$, then \eqref{eq84} is clearly an element of $\Cbb\cdot z^{\wt w_3'-\wt w_1-\wt w_2}[\log z]$. Suppose that each eigenspace of $\Lss$ on $\mc W_3$ is finite-dimensional. Then for each $n\in\Cbb, k\in\Nbb$ we can find a unique $\Lss$-eigenvector $\mc Y(w_1)_{n,k}w_2$ of $\mc W_3$ with eigenvalue $\wt w_1+\wt w_2-n-1$ such that for each $\Lss$-eigenvector $w_3'\in\mc W_3$ of $\mc W_3'$ with the same eigenvalue, $z^{n+1}\bk{\mc Y(w_1,z)w_2,w_3'}=\sum_{k\in\Nbb}\bk{\mc Y(w_1)_{n,k}w_2,w_3'}(\log z)^k$. 

It follows that for any $\Lss$-eigenvalues $w_1\in\mc W_1,w_2\in\mc W_2$ (and hence, for any non-necessarily eigenvalues), we have expansion
\begin{gather}
	\mc Y(w_1,z)w_2=\sum_{n\in\Cbb}\sum_{k\in\Nbb}\mc Y(w_1)_{n,k}w_2\cdot z^{-n-1}(\log z)^k,\nonumber\\	
\mc Y(w_1)_{n,k}w_2\in\mc W_3,	\label{eq103}
\end{gather}
and we have
\begin{align}
[\Lss,\mc Y(w_1)_{n,k}]=\mc Y(\Lss w_1)_{n,k}-(n+1)\mc Y(w_1)_{n,k}.	
\end{align}

Consequently,  if we assume that each $\Lss$-eigenspace of $\mc W_3$ is finite-dimensional, and that the real parts of the eigenvalues of $\Lss$ on $\mc W_3$ are bounded below, then we have the expansion \eqref{eq103}, and $\mc Y$ satisfies the lower truncation property: $\mc Y(w_1)_{n,k}w_2=0$ for sufficiently large $n$. Thus, $\mc Y$ is an intertwining operator of the fixed point subalgebra $\Ubb^G$ in the usual sense as \cite{HLZ10}. Note that by \cite[Lemma 2.4]{Miy04}, our assumption on $\mc W_3$ automatically holds when $\Ubb^G$ is $C_2$-cofinite and $\mc W_3$ is $\Ubb^G$-generated by finitely many vectors (equivalently, $\mc W_3$ is a grading-restricted generalized $\Ubb^G$-module \cite[Cor. 3.16]{Hua09}).
\end{rem}

The following property says that in order to show that $\mc Y$ is an intertwining operator, it suffices to verify the analytic Jacobi identity for  $\tipaz$ in a small region. Therefore, at least in the case that $\Ubb^G$ is $C_2$-cofinite and the modules are $\Ubb^G$-finitely generated, our definition of intertwining operators agrees with that in \cite{McR21} (cf. the paragraph above Rem. 4.16).

\begin{pp}\label{lb57}
Let $U\subset\Cbb^\times$ be non-empty and simply-connected with a continuous $\arg$ function $\arg_U$. Let $\mc Y$ be as in Def. \ref{lb51}, but satisfies the two conditions only when $\tipaz\in U$ and $\arg z=\arg_U(z)$ is defined by that of $U$. Then $\mc Y$ is a type $\mc W_3\choose\mc W_1\mc W_2$ intertwining operator.
\end{pp}

\begin{proof}
Note that \eqref{eq84} applies to $Y^{g_j}$ since $Y^{g_j}$ is an intertwining operator. By the argument in Rem. \ref{lb55}, \eqref{eq79} holds  for all $\tipaz\in U,\arg(\tipaz)=\arg_U(\tipaz)$. Hence, it holds for any $\tipaz\in\Cbb^\times$ and any compatible $\arg\tipaz$. Thus, the same argument as in Prop. \ref{lb56} proves \eqref{eq84} for $\mc Y$ and any $w_1,w_2,w_3'$.

Fix $\tipaz_0\in U$ with $\arg\tipaz_0=\arg_U(\tipaz_0)$. For each $\tipaz\in\Cbb^\times$, write $\tipaz=\xi\tipaz_0$ and, in particular, $\arg\tipaz=\arg\xi+\arg\tipaz_0$. Choose $\arg(\xi^{-1})=-\arg\xi$. Since $[L_0,\wtd L_0^g]=0$, we see that $L_0$ commutes with $P_n^{g_j}$. Therefore $\xi^{-L_0}$ commutes with each $P_n^{g_j}$. Using this fact and \eqref{eq84}, we obtain three equations about series of $n$:
\begin{gather*}
\bk{Y^{g_3}(u,z)\mc Y(w_1,\tipaz)w_2,w_3'}=\bk{Y^{g_3}(\xi^{-L_0}u,\xi^{-1}z)\mc Y(\xi^{-L_0}w_1,\tipaz_0)\xi^{-L_0}w_2,\xi^{L_0}w_3'}\\
\bk{\mc Y(w_1,\tipaz)Y^{g_2}(u,z)w_2,w_3'}=\bk{\mc Y(\xi^{-L_0}w_1,\tipaz_0)Y^{g_2}(\xi^{-L_0}u,\xi^{-1}z)\xi^{-L_0}w_2,\xi^{L_0}w_3'}\\
\bk{\mc Y(Y^{g_1}(u,z-\tipaz)w_1,\tipaz)w_2,w'_3}=\bk{\mc Y(Y^{g_1}(\xi^{-L_0}u,\xi^{-1}(z-\tipaz))\xi^{-L_0}w_1,\tipaz_0)\xi^{-L_0}w_2,\xi^{L_0}w'_3}	
\end{gather*}
which are defined in the similar way as \eqref{eq76}, \eqref{eq77}, \eqref{eq78}. Since the analytic Jacobi identity holds for $\tipaz_0$, we know that: (a) The above three series converges a.l.u. when $|\xi^{-1}z|>|\tipaz_0|,|\xi^{-1}z|<|\tipaz_0|,|\xi^{-1}(z-\tipaz)|<|\tipaz_0|$ respectively. (b) If we assume that $\xi^{-1}z\in R_{\tipaz_0}$, that $\arg(\xi^{-1}z)=\arg z-\arg\xi$ equals $\arg\tipaz_0$, and that $\arg(\xi^{-1}z-\tipaz_0):=\arg(\xi^{-1}(z-\tipaz))=\arg(z-\tipaz)-\arg\xi$ equals either $\arg \tipaz_0$ or $\arg\tipaz_0-\pi$, then the above three functions of $\xi^{-1}z$ can be extended to  the same holomorphic function on $\Upsigma_{\tipaz_0}$ which can furthermore be extended to a multi-valued holomorphic function on $\Cbb^\times\setminus\{\tipaz_0\}$. The analytic Jacobi identity for $\tipaz$ follows immediately.
\end{proof}

We now relate twisted intertwining operators and twisted conformal blocks. Choose any $r>0$. Then
\begin{align*}
\Upsigma_r=\Cbb\setminus\{\im t,r+\im t:t\geq 0\}.	
\end{align*}
Define a positively $3$-pathed Riemann sphere with local coordinates
\begin{align*}
\fk P_r=\{\Pbb^1;0,r,\infty;\zeta,\zeta-r,1/\zeta;\upgamma_0,\upgamma_r,\upgamma_\infty\}	
\end{align*}
where $\zeta$ is the standard coordinate of $\Cbb$, $\upgamma_0,\upgamma_r,\upgamma_\infty$ are paths in $\Upsigma_r$ with common end point $\bigstar$, and their initial points are on the real line satisfying
\begin{align*}
0<\upgamma_0(0)<r<\upgamma_r(0)<2r<\upgamma_\infty(0).	
\end{align*}
The equivalence classes of  $(\upgamma_0,\upgamma_r,\upgamma_\infty)$ (in the sense of Rem. \ref{lb2}) subject to such condition are clearly unique. Thus, the  twisted conformal blocks and the branched coverings associated to $\fk P_r$ are independent of the choice of such paths, thanks to Rem. \ref{lb2} and \ref{lb27}.

Set $\Sbf=\{x_1,x_2,x_3\}$ as usual. For $x=0,r,\infty$, let $\upepsilon_x$ be anticlockwise circle (defined by the give local coordinate at $x$) from and to $\upgamma_x(0)$, and set  $\upalpha_x=\upgamma_x^{-1}\upepsilon_x\upgamma_x$ (cf. \eqref{eq1}). Then $[\upalpha_0],[\upalpha_r],[\upalpha_\infty]$ generate $\Gamma=\pi_1(\Pbb^1\setminus\Sbf;\bigstar)$ (which is required in the definition of positively pathed Riemann spheres) since any two of these three elements are clearly free generators  of $\Gamma$. Recall $g_3=g_1g_2$
\begin{lm}
$g_1,g_2,g_3^{-1}$ are admissible, i.e., there is an action $\Gamma\curvearrowright E$ sending
\begin{align*}
[\upalpha_r]\mapsto g_1,\quad [\upalpha_0]\mapsto g_2,\quad [\upalpha_\infty]\mapsto g_3^{-1}.	
\end{align*}
\end{lm}

\begin{proof}
We have $[\upalpha_r\upalpha_0]=[\upalpha_\infty^{-1}]$, and $[\upalpha_0],[\upalpha_r]$ are free generators of $\Gamma$.
\end{proof}

We shall also choose open discs $W_0,W_r,W_\infty$, each of which contains only one of $0,r,\infty$. We set
\begin{align*}
W_0=\{z\in\Cbb:|z|<r\},\quad W_r=\{z\in\Cbb:|z-r|<r\},\quad W_\infty=\{z\in\Pbb^1:|z|>r\}.
\end{align*}

\begin{pp}\label{lb54}
Let $\mc Y(\cdot,r)$ denote a linear map
\begin{gather*}
\mc W_1\otimes \mc W_2\rightarrow (\mc W_3')^*,\qquad w_1\otimes w_2\mapsto\mc Y(w_1,r)w_2,	
\end{gather*}
also regarded as a linear functional on $\mc W_1\otimes\mc W_2\otimes \mc W_3$. Then the following two statements are equivalent.
\begin{enumerate}[label=(\alph*)]
\item $\mc Y(\cdot,r)$ satisfies the analytic Jacobi identity in Def. \ref{lb51} (for the $g_1$-, $g_2$-, and $g_3$-twisted modules $\mc W_1,\mc W_2,\mc W_3$) in the special case that $\tipaz=r$ and $\arg\tipaz=0$.	
\item $\mc Y(\cdot,r)$ is a conformal block associated to $\fk P_r$ and the $g_1$-, $g_2$-, and $g_3^{-1}$-twisted modules $\mc W_1,\mc W_2,\mc W_3'$.
\end{enumerate}	
\end{pp}

\begin{proof}
Step 1. Each of the two statements consists of two parts: the ``convergence part" and the ``extension" part. Let us first verify the equivalence of the ``convergence parts". This is obvious ``near $0$ and $r$". So we only need to focus on the convergence  near $\infty$, namely, the equivalence of convergence of \eqref{eq76} and of \eqref{eq17} (when $j=3$). On the side of \eqref{eq17}, noting Rem. \ref{lb10}, we have the absolute  convergence of the series of $z^{1/|g_3|}$
\begin{align*}
\bk{\mc Y(w_1,r)w_2,Y^{g_3^{-1}}(u,z)w_3'}:=\sum_{n\in\frac 1{|g_3|}\Nbb}	\bk{\mc Y(w_1,r)w_2,Y^{g_3^{-1}}(u)_nw_3'}z^{-n-1}	
\end{align*}
on $0<|z^{1/|g_3|}|<|r^{-1/|g_3|}|$. This is equivalent, by linearity, to the convergence  of the series of $z^{-1/|g_3|}$
\begin{align*}
&\bk{\mc Y(w_1,r)w_2,Y^{g_3^{-1}}(\mc U(\tipxgamma_z)u,z^{-1})w_3'}\\
:=&\sum_{n\in\frac 1{|g_3|}\Nbb}	\bk{\mc Y(w_1,r)w_2,Y^{g_3^{-1}}(e^{zL_1}(-z^{-2})^{L_0}u)_nw_3'}z^{n+1}	
\end{align*}
on $0<|z^{-1/|g_3|}|<|r^{-1/|g_3|}|$, where $\mc U(\tipxgamma_z)$ equals $e^{zL_1}(-z^{-2})^{L_0}$ on $\Ubb$. By \eqref{eq64}, the above series is equivalent to the convergence of the series of \eqref{eq76} (by setting $\tipaz=r$ with zero arg) on the same domain, and we have
\begin{align}
\bk{\mc Y(w_1,r)w_2,Y^{g_3^{-1}}(\mc U(\tipxgamma_z)u,z^{-1})w_3'}=	\bk{Y^{g_3}(u,z)\mc Y(w_1,r)w_2,w_3'}\label{eq83}
\end{align}
where the right hand side is defined by \eqref{eq76}. This proves the equivalence of the ``convergence" part.

Step 2. Recall that $\zeta$ is the standard coordinate of $\Cbb$. We set $\fk U=\{\mc U_\varrho(\zeta)^{-1}u:u\in\Ubb\}$, which is a subspace of $\scr U_{\Pbb^1}(\Pbb^1\setminus\Sbf)$ containing a subset that generates freely $\scr U_{\Pbb^1\setminus\Sbf}$ (cf. the proof of Lemma. \ref{lb53}). Set $\uppsi=\mc Y(\cdot,r)$. We shall verify the equivalence of: (a) the two conditions of Prop. \ref{lb11} (stated for the linear functional $\mc Y(\cdot,r)$), (b) the analytic Jacobi identity in Def. \ref{lb51}.

Assume (b). Part 1 of Prop. \ref{lb11} is already proved. For each $u\in\Ubb,w_1\in\mc W_1,w_2\in\mc W_2,w_3'\in\mc W_3'$, set $w_\blt=w_1\otimes w_2\otimes w_3'$, let $f_{r,u,w_\blt}$ be the multivalued holomorphic function $f_r=f_\tipaz$ in the analytic Jacobi identity of Def. \ref{lb51}, which becomes single-valued on any open simply-connected  $U\subset\Pbb^1\setminus\Sbf$ if we specify a path $\uplambda$ in $\Pbb^1\setminus\Sbf$ from inside $U$ to $\bigstar$.  We define this function on $U$ to be $\wr\uppsi(\uplambda,\mc U_\varrho(\zeta)^{-1}u,w_\blt)\big|_U$. Moreover, this multivalued function can be chosen such that when $U=\Upsigma_r$ and $\uplambda$ is any path in $\Upsigma_r$ with initial point $\bigstar$, then $\wr\uppsi(\uplambda,\mc U_\varrho(\zeta)^{-1}u,w_\blt)\big|_{\Upsigma_r}$ agrees with the common (single-valued) holomorphic function on $\Upsigma_r$ mentioned in the analytic Jacobi identity of Def. \ref{lb51}. It is now easy to check that condition 2 of Prop. \ref{lb11} holds for $W_j$ being $W_0$ or $W_r$ (notice Rem. \ref{lb9}). As for $W_\infty$, condition 2 also holds due to \eqref{eq83} and the fact that $\mc U_\varrho(\zeta)^{-1}u$ equals $\mc U(\tipxgamma_z)u$ under the trivialization $\mc U_\varrho(\zeta^{-1})$  (cf. Ex. \ref{lb36}). The proof of (a) is complete.

Assume (a). Let $\wr\uppsi$ be as in Prop. \ref{lb11}. Then it is easy to check that $\wr\uppsi(\uplambda,\mc U_\varrho(\zeta^{-1})u,w_\blt)\big|_{\Upsigma_r}$ restricts to \eqref{eq76}, \eqref{eq77}, \eqref{eq78} in the required regions. This verifies the analytic Jacobi identity in Def. \ref{lb51}, hence proves (b).
\end{proof}

We are now ready to prove the main result of this subsection. Let
\begin{align*}
\mc I{\mc W_3\choose \mc W_1\mc W_3}=\bigg\{\text{Type }{\mc W_3\choose \mc W_1\mc W_2}\text{ intertwining operators of }\Ubb\bigg\}.	
\end{align*}
Recall that $\mc W_1,\mc W_2,\mc W_3$ are  $g_1$-, $g_2$-, $g_3$-twisted $\Ubb$-modules. Let $\CB_{\fk P_r}(\mc W_1,\mc W_2,\mc W_3')$ be the space of conformal blocks associated to $\fk P_r$ and $\mc W_1,\mc W_2,\mc W_3'$. For each $r>0$, define the restriction map
\begin{gather*}
\nu_r:	\mc I{\mc W_3\choose \mc W_1\mc W_2}\rightarrow \CB_{\fk P_r}(\mc W_1,\mc W_2,\mc W_3'),\\
\mc Y\mapsto \mc Y(\cdot,r).
\end{gather*}
That the range of $\nu_r$ is inside $ \CB_{\fk P_r}(\mc W_1,\mc W_2,\mc W_3')$ is due to Prop. \ref{lb54}.

\begin{thm}
For each $r>0$, the linear map $\nu_r$ is bijective.
\end{thm}

\begin{proof}
By \eqref{eq81} and Lemma \ref{lb68}, $\nu_r$ is injective. To prove that $\nu_r$ is surjective, we choose any $\mc Y(\cdot,r)$ in the codomain of $\nu_r$, which satisfies the analytic Jacobi identity (in Def. \ref{lb51}) for $\tipaz=r$ and $\arg\tipaz=\arg r=0$. For a general $\tipaz\in\Cbb^\times$ with argument, we set
\begin{align*}
\bk{\mc Y(w_1,\tipaz)w_2,w_3'}=\bk{\mc Y(\xi^{-L_0}w_1,r)\xi^{-L_0}w_2,\xi^{L_0}w_3}	
\end{align*}
where $\xi=r^{-1}\tipaz$ and $\arg\xi=\arg\tipaz$. We may use the same method as in the proof of Prop. \ref{lb57} to prove the analytic Jacobi identity for all $\tipaz$. Finally, it is straightforward to check that \eqref{eq79} holds, which is equivalent to condition 1 of Def. \ref{lb51} due to Rem. \ref{lb55}. This proves that $\mc Y$ is an intertwining operator whose value at $r$ is the given $\mc Y(\cdot,r)$.
\end{proof}

\begin{lm}\label{lb68}
Let $\mc V$ be a vector space, and let $f(t_1,\dots,t_N)=\sum_{n_1,\dots,n_N\in\Nbb}f_{n_1,\dots,n_N}t_1^{n_1}\cdots t_N^{n_N}$ be a formal power series of $t_1,\dots,t_N$ where  $f_{n_1,\dots,n_N}\in\mc V$. Suppose that for each $i$ there is a linear operator $A_i$ on $\mc V$ such that $\partial_{t_i}f=A_if$. Then $f$ is determined by $f(0,\dots,0)=f_{0,\dots,0}$.
\end{lm}

\begin{proof}
	$\partial_{t_i}f=A_if$ shows $f_{n_1,\dots,n_i+1,\dots,n_N}=(n_i+1)^{-1}A_if_{n_1,\dots,n_i,\dots,n_N}$. Therefore all $f_{n_1,\dots,n_N}$ are determined by $f_{0,\dots,0}$.
\end{proof}

\subsection{OPE for permutation-twisted intertwining operators}\label{lb71}

Choose $0<r_2<r_1$ satisfying $r_1-r_2<r_2$. Set $\arg r_1=\arg r_2=\arg(r_2-r_1)=0$. Set
\begin{align*}
\Upsigma_{r_1,r_2}=\Cbb\backslash\{\im t,r_1+\im t,r_2+\im t:t\geq 0\}.	
\end{align*}
We can then define a positively $4$-pathed Riemann sphere
\begin{align*}
\fk P_{r_1,r_2}=(\Pbb^1;0,r_1,r_2,\infty;\zeta,\zeta-r_1,\zeta-r_2,\zeta^{-1};\upsigma_1,\upsigma_2,\upsigma_3,\upsigma_4)	
\end{align*}
where the four paths $\upsigma_1,\dots,\upsigma_4$ with common end point are all inside $\Upsigma_{r_1,r_2}$. Then, by replacing $\upsigma_\blt$ by equivalent paths (cf. \ref{lb2}), $\fk P_{r_1,r_2}$ has the following two decompositions:
\begin{enumerate}[label=(\alph*)]
\item $\fk P_{r_1,r_2}$ is the sewing of  $\fk P_{r_1}$ along the marked point $0$ (and its local coordinate $\zeta$), and $\fk P_{r_2}^{(a)}\simeq\fk P_{r_2}$ along the marked point $\infty$ (with local coordinate $\zeta^{-1}$). To perform this sewing, we choose $r_2<\rho_2<\rho_1<r_1$, remove a small closed disk from $\{z\in\Cbb:|z|<\rho_1\}$ inside $\fk P_{r_1}$, remove and one from $\{z\in\Pbb^1:|z|^{-1}<\rho_2^{-1}\}$ inside $\fk P_{r_2}^{(a)}$ (the two discs play the role of $W_0,M_0$ in Sec. \ref{lb51}), and glue the remaining part to obtain $\fk P_{r_1,r_2}$. After gluing, any point $|z|\geq r_1$ of $\fk P_{r_1}$ and any point $|z|\leq r_2$ of $\fk P_{r_2}^{(b)}$ become the point $z$ of $\fk P_{r_1,r_2}$.

\item $\fk P_{r_1,r_2}$ is the sewing of  $\fk P_{r_2}^{(b)}\simeq\fk P_{r_2}$ along the marked point $r_2$ and $\fk P_{r_1-r_2}$ along the marked point $\infty$. Similar to (a), one removes closed discs from two open discs and glue the remaining part to get $\fk P_{r_1,r_2}$. After gluing, any point $|z|\geq r_2$ of $\fk P^{(b)}_{r_2}$ becomes $z$ of $\fk P_{r_1,r_2}$, and any point $|z|\leq r_1-r_2$ of $\fk P_{r_1-r_2}$ becomes the point $z+r_2$ of $\fk P_{r_1,r_2}$.
\end{enumerate}

Note that in both (a) and (b), we need to replace the paths in $\fk P_{r_1},\fk P_{r_1-r_2},\fk P_{r_2}^{(a)},\fk P_{r_2}^{(b)}$ by equivalent paths such that \eqref{eq90} and \eqref{eq54} hold. We record the result
\begin{align}
\fk P_{r_1,r_2}=\fk P_{r_1}\# \fk P_{r_2}^{(a)}=\fk P_{r_2}^{(b)}\#\fk P_{r_1-r_2}.	\label{eq97}
\end{align}

Assume $E$ is a finite set.  Choose $g_1,g_2,g_3\in\Perm(E)$ and set $g_4=(g_1g_2g_3)^{-1}$, and assign group elements to the marked points as follows
\begin{gather}\label{eq91}
\begin{array}{c}
\fk P_{r_1}: g_2g_3\rightsquigarrow 0,\quad g_1\rightsquigarrow	r_1,\quad (g_1g_2g_3)^{-1}\rightsquigarrow \infty\\[0.7ex]
\fk P_{r_2}^{(a)}: g_3\rightsquigarrow 0,\quad g_2\rightsquigarrow	r_2,\quad (g_2g_3)^{-1}\rightsquigarrow \infty
\end{array}
\end{gather}
and also
\begin{gather}\label{eq92}
\begin{array}{c}
\fk P_{r_2}^{(b)}: g_3\rightsquigarrow 0,\quad g_1g_2\rightsquigarrow	r_2,\quad (g_1g_2g_3)^{-1}\rightsquigarrow \infty\\[0.7ex]
\fk P_{r_1-r_2}: g_2\rightsquigarrow 0,\quad g_1\rightsquigarrow	r_1-r_2,\quad (g_1g_2)^{-1}\rightsquigarrow \infty
\end{array}
\end{gather}
In each of the four cases, the group elements are admissible (cf. Def. \ref{lb61}). We can use either of the  above two sets of data to define an action of $\Gamma_{r_1,r_2}=\pi_1(\Pbb^1\setminus\{0,r_1,r_2,\infty\},\bigvarstar)$ on $E$ (where $\bigvarstar$ is the common end point of the four paths of $\fk P_{r_1,r_2}$) as in Subsec. \ref{lb60}, and the results are the same: let $\upepsilon_j$ be the circle around $0,r_1,r_2,\infty$ respectively when $j=1,2,3,4$, then $[\upsigma_j^{-1}\upepsilon_j\upsigma_j]$ acts as $g_j$. (We set $g_4=(g_1g_2g_3)^{-1}$.)

Now, let $\fk X_{r_1},\fk X_{r_2}^{(a)},\fk X_{r_2}^{(b)},\fk X_{r_1-r_2}$ be, respectively, the permutation branched coverings of $\fk P_{r_1},\fk P_{r_2}^{(a)},\fk P_{r_2}^{(b)},\fk P_{r_1-r_2}$ and the fundamental group actions defined by \eqref{eq91} and \eqref{eq92}. Let $\fk X_{r_1,r_2}$ be the permutation branched covering of $\fk P_{r_1,r_2}$ defined by the action described previously. Then, by Thm. \ref{lb46}, 
\begin{align}
\fk X_{r_1,r_2}\simeq \fk X_{r_1}\#\fk X_{r_2}^{(a)}\simeq \fk X_{r_2}^{(b)}\#\fk X_{r_1-r_2}	\label{eq99}
\end{align}
where the two sewings are defined with respect the sewings of $\fk P_{r_1}$ with $\fk P_{r_2}^{(a)}$ and $\fk P_{r_2}^{(b)}$ with $\fk P_{r_1-r_2}$ (cf. Subsec. \ref{lb43}).

Now we assume $\Vbb$ is CFT-type, $C_2$-cofinite, and rational. Let $\Ubb=\Vbb^{\otimes E}$.  Associate semi-simple  permutation twisted $\Ubb$-modules (``semi-simple" means that it is a finite direct sum of irreducible twisted $\Ubb$-modules) to marked points
\begin{gather}\label{eq93}
	\begin{array}{c}
		\fk P_{r_1}: \mc W_\pi\rightsquigarrow 0,\quad \mc W_1\rightsquigarrow	r_1,\quad \mc W_4'\rightsquigarrow \infty\\[0.7ex]
		\fk P_{r_2}^{(a)}: \mc W_3\rightsquigarrow 0,\quad \mc W_2\rightsquigarrow	r_2,\quad \mc W_\pi'\rightsquigarrow \infty
	\end{array}
\end{gather}
and also
\begin{gather}\label{eq94}
	\begin{array}{c}
		\fk P_{r_2}^{(b)}: \mc W_3\rightsquigarrow 0,\quad \mc W_\iota\rightsquigarrow	r_2,\quad \mc W_4'\rightsquigarrow \infty\\[0.7ex]
		\fk P_{r_1-r_2}: \mc W_2\rightsquigarrow 0,\quad \mc W_1\rightsquigarrow	r_1-r_2,\quad \mc W_\iota'\rightsquigarrow \infty
	\end{array}
\end{gather}
whose types correspond to the group elements in \eqref{eq91}, \eqref{eq92}. By Thm. \ref{lb47}, all these twisted $\Ubb$-modules arise from untwisted  semi-simple $\Vbb$-modules as described in Thm. \ref{lb14}. The following is the main result of this section. We do not assume $0<r_1-r_2<r_2<r_1$. 

\begin{thm}\label{lb63}
The following are true.
\begin{enumerate}
\item For any $\mc Y_\alpha\in\mc I{\mc W_4\choose \mc W_1\mc W_\pi}$, $\mc Y_\beta\in\mc I{\mc W_\pi\choose\mc W_2\mc W_3}$, $\mc Y_\gamma\in\mc I{\mc W_4\choose\mc W_\iota\mc W_3}$, $\mc Y_\delta{\mc W_\iota\choose\mc W_1\mc W_2}$, and for any $w_1\in\mc W_1,w_2\in\mc W_2,w_3\in\mc W_3,w_4'\in\mc W_4$, the series
\begin{align}
&\bk{\mc Y_\alpha(w_1,r_1)\mc Y_\beta(w_2,r_2)w_3,w_4'}\nonumber\\
:=&\sum_{n\in\frac 1{|g_2g_3|}\Nbb}\bk{\mc Y_\alpha(w_1,r_1)P_n^{g_2g_3}\mc Y_\beta(w_2,r_2)w_3,w_4'}\label{eq95}
\end{align}
and
\begin{align}
&\bk{\mc Y_\gamma(\mc Y_\delta(w_1,r_1-r_2)w_2,r_2)w_3,w_4'}\nonumber\\
:=&\sum_{n\in\frac 1{|g_1g_2|}\Nbb}\bk{\mc Y_\gamma(P_n^{g_1g_2}\mc Y_\delta(w_1,r_1-r_2)w_2,r_2)w_3,w_4'}\label{eq96}
\end{align}
converge absolutely on $I_1=\{(r_1,r_2):0<r_2<r_1\}$ and  $I_2=\{(r_1,r_2):0<r_1-r_2<r_1\}$ respectively. Moreover, if we vary $r_1,r_2$ and assume the $\arg$ of $r_1,r_2,r_1-r_2$ are all $0$, then these two functions are real analytic functions of  $r_1,r_2$, namely, they can be extended to holomorphic functions on neighborhoods of $I_1$ and $I_2$ respectively.

\item Let $\mc W_1,\mc W_2,\mc W_3,\mc W_4$ be semi-simple $g_1,g_2,g_3,g_1g_2g_3$-twisted $\Ubb=\Vbb^{\otimes E}$-modules. Then for each semi-simple $g_2g_3$-twisted module $\mc W_\pi$ (resp. $g_1g_2$-twisted module $\mc W_\iota$) and each $\mc Y_\alpha\in\mc I{\mc W_4\choose \mc W_1\mc W_\pi}$, $\mc Y_\beta\in\mc I{\mc W_\pi\choose\mc W_2\mc W_3}$ (resp. $\mc Y_\gamma\in\mc I{\mc W_4\choose\mc W_\iota\mc W_3}$, $\mc Y_\delta{\mc W_\iota\choose\mc W_1\mc W_2}$), there exists a $g_1g_2$-twisted module $\mc W_\iota$ (resp. $g_2g_3$-twisted module $\mc W_\pi$) and $\mc Y_\gamma\in\mc I{\mc W_4\choose\mc W_\iota\mc W_3}$, $\mc Y_\delta{\mc W_\iota\choose\mc W_1\mc W_2}$ (resp. $\mc Y_\alpha\in\mc I{\mc W_4\choose \mc W_1\mc W_\pi}$, $\mc Y_\beta\in\mc I{\mc W_\pi\choose\mc W_2\mc W_3}$) such that for any $w_1\in\mc W_1,w_2\in\mc W_2,w_3\in\mc W_3,w_4'\in\mc W_4$, \eqref{eq95} and \eqref{eq96} agree on $I_1\cap I_2$, assuming the $\arg$ of $r_1,r_2,r_1-r_2$ are all $0$.
\end{enumerate}
\end{thm}

Note that if $g_1,g_2,g_3\in G\leq\Perm(E)$ where $G$ is solvable, $\Ubb^G$ is $C_2$-cofinite and rational by \cite{Miy15,CM16}. Then Thm. \ref{lb63} follows from \cite{McR21}.  The $C_2$-cofiniteness  of $\Ubb^G$ is conjectured to be true for any finite group $G$. If this were proved, then \cite{McR21} would imply that $\Ubb^G$ is also rational. Then Thm. \ref{lb63} would also follow from \cite{McR21}. Here, we provide a proof without knowing $\Ubb^G$ to be $C_2$-cofinite.

\begin{proof}
\eqref{eq95} and \eqref{eq96} are the sewing of permutation-twisted conformal blocks corresponding to the geometric sewing \eqref{eq97}. Therefore, by Thm. \ref{lb48}, when $r_1,r_2$ are in $I_1$ or $I_2$ respectively, \eqref{eq95} or \eqref{eq96} converges absolutely to a conformal block associated to $\fk P_{r_1,r_2}$ and the twisted modules $\mc W_1,\mc W_2,\mc W_3,\mc W_4'$. Moreover, due to Thm. \ref{lb62}, for fixed $r_1,r_2$, any such conformal block can be expressed either as \eqref{eq95} or \eqref{eq96}. Therefore, statement 2 holds for any fixed $(r_1,r_2)\in I_1\cap I_2$.

Note that $L_0$ commutes with $P_n^g$ (since $[L_0,\wtd L_0^g]=0$). Choose $\rho\in (0,1)$. Then by \eqref{eq84}, the series \eqref{eq95} equals $\sum_{n\in\frac 1{|g_2g_3|}\Nbb}f_n(r_1,r_2)$ where
\begin{align*}
f_n(r_1,r_2)=\bk{\mc Y_\alpha(r_1^{-L_0}w_1,1)\Big(\frac{r_2}{r_1\rho}\Big)^{L_0}P_n^{g_2g_3}\mc Y_\beta((\rho/r_2)^{L_0}w_2,\rho)(\rho/r_2)^{L_0}w_3,r_1^{L_0}w_4'}.
\end{align*}
To prove the analyticity of \eqref{eq95}, it suffices to restrict to each simple submodule of $\mc W_\pi$ such that  on this submodule $L_0$ and $\wtd L_0^{g_2g_3}$ differ by a constant $\lambda$. It follows that $\big(\frac{r_2}{r_1\rho}\Big)^{L_0}P_n^{g_2g_3}=\big(\frac{r_2}{r_1\rho}\Big)^{n+\lambda}P_n^{g_2g_3}$. It is clear that for each $(r_1,r_2)$ in a compact subset $K$ of $O_1=\{(z_1,z_2)\in\Cbb^2:0<|z_2|<|z_1|,\mathrm{Re}(z_1)>0,\mathrm{Re}(z_2)>0\}$, one can choose $\rho$ greater than every $r_2/r_1$. Hence, by the absolute convergence proved in the first paragraph, we have $\sup_{(r_1,r_2)\in K}\sum_n |f_n(r_1,r_2)|<+\infty$. Therefore, as each $f_n(r_1,r_2)$ is analytic over $r_1,r_2$, the sum of $f_n$, which is just \eqref{eq95}, must be analytic on $O_1$ and hence analytic on $I_1$. The same method proves that \eqref{eq96} is analytic on $I_2$.

Finally, we assume that \eqref{eq95} equals \eqref{eq96} for one $(r_1,r_2)\in I_1\cap I_2$, and show that they are equal for all $(r_1+t_1,r_2+t_2)\in I_1\cap I_2$. By \eqref{eq81}  (the $L_{-1}$-derivative) and \eqref{eq98} (applied to $\mc Y_\delta$), if we write \eqref{eq95} (resp. \eqref{eq96})  as $f_i(t_1,t_2,w_1\otimes w_2\otimes w_3\otimes w_4')$ where $i=1$ (resp. $i=2$), then 
\begin{gather*}
\partial_{t_1}f_i(t_1,t_2,w_1\otimes w_2\otimes w_3\otimes w_4')=f_i(t_1,t_2,L_{-1}w_1\otimes w_2\otimes w_3\otimes w_4'),\\
\partial_{t_2}f_i(t_1,t_2,w_1\otimes w_2\otimes w_3\otimes w_4')=f_i(t_1,t_2,w_1\otimes L_{-1}w_2\otimes w_3\otimes w_4').	
\end{gather*}
The proof is thus finished by taking power series expansions of $t_1,t_2$ and applying Lemma \ref{lb68}.
\end{proof}

\begin{rem}
By the Main Theorem \ref{lb14}, $\mc Y_\alpha,\mc Y_\beta,\mc Y_\gamma,\mc Y_\delta$ can be viewed as untwisted conformal blocks associated to $\fk X_{r_1},\fk X_{r_2}^{(a)},\fk X_{r_2}^{(b)},\fk X_{r_1-r_2}$ (and suitable $\Vbb$-modules) respectively. Thus, Thm. \ref{lb63} can be viewed as the equivalence of sewing untwisted conformal blocks for $\Vbb$-modules associated to the two geometric sewing procedures described in \eqref{eq99}. Namely, Thm. \ref{lb63}, which describes the operator product expansion (OPE) of permutation-twisted intertwining operators, describes equivalently the OPE of certain untwisted conformal blocks associated to permutation coverings of $\Pbb^1$. Therefore, it relates the associativity isomorphism for tensor products in the crossed braided fusion category of  $\Perm(E)$-twisted $\Vbb^{\otimes E}$-modules and the OPE of untwisted $\Vbb$-conformal blocks associated to (possibly non-zero genera) compact Riemann surfaces. See Figure \ref{fig2} in the Introduction.
\end{rem}

\printindex

\noindent {\small \sc Yau Mathematical Sciences Center, Tsinghua University, Beijing, China.}

\noindent {\textit{E-mail}}: binguimath@gmail.com\qquad bingui@tsinghua.edu.cn

\begin{thebibliography}{999999}
		\footnotesize	
		
		
		
		


\bibitem[BDH17]{BDH17}
Bartels, A., Douglas, C.L. and Henriques, A., 2017. Conformal nets II: Conformal blocks. Communications in Mathematical Physics, 354(1), pp.393-458.


\bibitem[BDM02]{BDM02}
Barron, K., Dong, C. and Mason, G., 2002. Twisted sectors for tensor product vertex operator algebras associated to permutation groups. Communications in mathematical physics, 227(2), pp.349-384.

\bibitem[BHS98]{BHS98}
Borisov, L., Halpern, M.B. and Schweigert, C., 1998. Systematic approach to cyclic orbifolds. International Journal of Modern Physics A, 13(01), pp.125-168.


\bibitem[BJ19]{BJ19}
Bischoff, M. and Jones, C., 2019. Computing fusion rules for spherical G-extensions of fusion categories. arXiv preprint arXiv:1909.02816.


\bibitem[BKL15]{BKL15}
Bischoff, M., Kawahigashi, Y. and Longo, R., 2015. Characterization of 2D rational local conformal nets and its boundary conditions: the maximal case. Documenta Math. 20 (2015) 1137-1184.


\bibitem[BS11]{BS11}
Barmeier, T. and Schweigert, C., 2011. A geometric construction for permutation equivariant categories from modular functors. Transformation groups, 16(2), pp.287-337.


\bibitem[Ban98]{Ban98}
Bantay, P., 1998. Characters and modular properties of permutation orbifolds. Physics Letters B, 419(1-4), pp.175-178.

\bibitem[Ban02]{Ban02}
Bantay, P., 2002. Permutation orbifolds. Nuclear Physics B, 633(3), pp.365-378.

\bibitem[Bis20]{Bis20}
Bischoff, M., 2020. A remark about the anomalies of cyclic holomorphic permutation orbifolds. International Journal of Mathematics, 31(10), p.2050080.

\bibitem[CKLW18]{CKLW18}
Carpi, S., Kawahigashi, Y., Longo, R. and Weiner, M., 2018. From vertex operator algebras to conformal nets and back (Vol. 254, No. 1213). Memoirs of the American Mathematical Society



\bibitem[CM16]{CM16}
Carnahan, S. and Miyamoto, M., 2016. Regularity of fixed-point vertex operator subalgebras. arXiv preprint arXiv:1603.05645.


\bibitem[DGT21]{DGT21}
Damiolini, C., Gibney, A., and Tarasca, N. Conformal blocks from vertex algebras and their connections on $\ovl{\mc M}_{g,n}$.  Geom. Topol., 25(5):2235–2286, 2021.

\bibitem[DGT22]{DGT22}
Chiara Damiolini, Angela Gibney, and Nicola Tarasca. On factorization and vector bundles of conformal blocks from vertex algebras. Ann. Sci. \'Ec. Norm. Sup\'er., 2022.




		
\bibitem[DLM97]{DLM97}
Dong, C., Li, H. and Mason, G., 1997. Regularity of Rational Vertex Operator Algebras. Advances in Mathematics, 132(1), pp.148-166.	





\bibitem[DLXY24]{DLXY24}
Dong, C., Li, H., Xu, F., \& Yu, N. (2024). Fusion products of twisted modules in permutation orbifolds. Transactions of the American Mathematical Society, 377(03), 1717-1760.

\bibitem[DXY22]{DXY22}
Dong, C., Xu, F. and Yu, N. Permutation orbifolds and associative algebras. Science China Mathematics, 65 (2022) 259-268.


\bibitem[DXY24]{DXY24}
Dong, C., Xu, F. and Yu, N., 2024. Fusion Products of Twisted Modules in Permutation Orbifolds: II. arXiv preprint arXiv:2411.15751


\bibitem[Del19]{Del19}
Delaney, C., 2019. Fusion rules for permutation extensions of modular tensor categories. arXiv preprint arXiv:1909.03003.


\bibitem[Don]{Don}
Donaldson, S., 2011. Riemann surfaces. Oxford University Press.

\bibitem[EG18]{EG18}
Evans, D.E. and Gannon, T., 2018. Reconstruction and local extensions for twisted group doubles, and permutation orbifolds. arXiv preprint arXiv:1804.11145.


\bibitem[ENO10]{ENO10}
Etingof, P., Nikshych, D. and Ostrik, V., 2010. Fusion categories and homotopy theory. Quantum topology, 1(3), pp.209-273.


\bibitem[FB04]{FB04}
Frenkel, E. and Ben-Zvi, D., 2004. Vertex algebras and algebraic curves (No. 88). American Mathematical Soc..


\bibitem[FHL93]{FHL93}
Frenkel, I., Huang, Y. Z., \& Lepowsky, J. (1993). On axiomatic approaches to vertex operator algebras and modules (Vol. 494). American Mathematical Soc..


\bibitem[FRS89]{FRS89}
Fredenhagen, K., Rehren, K.H. and Schroer, B., 1989. Superselection sectors with braid group statistics and exchange algebras. Communications in Mathematical Physics, 125(2), pp.201-226.

\bibitem[FRS92]{FRS92}
Fredenhagen, K., Rehren, K.H. and Schroer, B., 1992. Superselection sectors with braid group statistics and exchange algebras II: Geometric aspects and conformal covariance. Reviews in Mathematical Physics, 4(spec01), pp.113-157.

\bibitem[FS04]{FS04}
Frenkel, E. and Szczesny, M., 2004. Twisted modules over vertex algebras on algebraic curves. Advances in Mathematics, 187(1), pp.195-227.

	


\bibitem[For]{For}
Forster, O., 2012. Lectures on Riemann surfaces (Vol. 81). Springer Science \& Business Media.

\bibitem[Ful]{Ful}
Fulton, W., 2013. Algebraic topology: a first course (Vol. 153). Springer Science \& Business Media.


\bibitem[Gui20a]{Gui20a}
Gui, B., 2020. Unbounded Field Operators in Categorical Extensions of Conformal Nets,  arXiv:2001.03095


\bibitem[Gui21]{Gui21}
Gui, B., 2021. Categorical extensions of conformal nets. Comm. Math. Phys. 383, 763-839 (2021).



\bibitem[Gui24a]{Gui24a}
Gui, B. (2024). Convergence of sewing conformal blocks. Commun. Contemp. Math., Vol. 26, No. 03.

\bibitem[Gui24b]{Gui24b}
Gui, B., 2024. Sewing and Propagation of Conformal Blocks. New York J. Math., 30 (2024) 187–230.

\bibitem[HLZ10]{HLZ10}
Huang, Y.Z., Lepowsky, J. and Zhang, L., 2010. Logarithmic tensor category theory, II: Logarithmic formal calculus and properties of logarithmic intertwining operators. arXiv preprint arXiv:1012.4196.

\bibitem[Hen14]{Hen14}
Henriques, A., 2014. Course notes for Conformal field theory (math 290).


\bibitem[Hua95]{Hua95}
Huang, Y.Z., 1995. A theory of tensor products for module categories for a vertex operator algebra, IV. Journal of Pure and Applied Algebra, 100(1-3), pp.173-216.

\bibitem[Hua97]{Hua97}
Huang, Y.Z., 1997. Two-dimensional conformal geometry and vertex operator algebras (Vol. 148). Springer Science \& Business Media.

\bibitem[Hua05]{Hua05}
Huang, Y.Z., 2005. Differential equations, duality and modular invariance. Communications in Contemporary Mathematics, 7(05), pp.649-706.


\bibitem[Hua08a]{Hua08a}
Huang, Y.Z., 2008. Vertex operator algebras and the Verlinde conjecture. Communications in Contemporary Mathematics, 10(01), pp.103-154.

\bibitem[Hua08b]{Hua08b}
Huang, Y.Z., 2008. Rigidity and modularity of vertex tensor categories. Communications in contemporary mathematics, 10(supp01), pp.871-911.

\bibitem[Hua09]{Hua09}
Huang, Y.Z., 2009. Cofiniteness conditions, projective covers and the logarithmic tensor product theory. Journal of Pure and Applied Algebra, 213(4), pp.458-475.



\bibitem[Hua18]{Hua18}
Huang, Y.Z., 2018. Intertwining operators among twisted modules associated to not-necessarily-commuting automorphisms. Journal of Algebra, 493, pp.346-380.

\bibitem[KLM01]{KLM01}
Kawahigashi, Y., Longo, R. and Müger, M., 2001. Multi-Interval Subfactors and Modularity of Representations in Conformal Field Theory. Communications in Mathematical Physics, 219(3), pp.631-669.

\bibitem[KLX05]{KLX05}
Kac, V.G., Longo, R. and Xu, F., 2005. Solitons in affine and permutation orbifolds. Communications in mathematical physics, 253(3), pp.723-764.

\bibitem[KR09a]{KR09a}
Kong, L. and Runkel, I., 2009. Cardy algebras and sewing constraints, I. Communications in Mathematical Physics, 292(3), pp.871-912.

\bibitem[KR09b]{KR09b}
Kong, L. and Runkel, I., 2009. Algebraic structures in Euclidean and Minkowskian two-dimensional conformal field theory. arXiv preprint arXiv:0902.3829.


\bibitem[Kong08]{Kong08}
Kong, L., 2008. Cardy condition for open-closed field algebras. Communications in mathematical physics, 283(1), pp.25-92.


\bibitem[LX04]{LX04}
Longo, R. and Xu, F., 2004. Topological sectors and a dichotomy in conformal field theory. Communications in mathematical physics, 251(2), pp.321-364.	

\bibitem[LX19]{LX19}
Liu, Z. and Xu, F., 2019. Jones-Wassermann subfactors for modular tensor categories. Advances in Mathematics, 355, p.106775.	


\bibitem[McR21]{McR21}
McRae, R., 2021. Twisted modules and G-equivariantization in logarithmic conformal field theory. Communications in Mathematical Physics, 383(3), pp.1939-2019.

\bibitem[Miy04]{Miy04}
Miyamoto, M., 2004. Modular invariance of vertex operator algebras satisfying C2-cofiniteness. Duke Mathematical Journal, 122(1), pp.51-91.

\bibitem[Miy15]{Miy15}
Miyamoto, M., 2015. $C_2$-cofiniteness of Cyclic-Orbifold Models. Communications in Mathematical Physics, 335(3), pp.1279-1286.



\bibitem[Ten19a]{Ten19a}
Tener, J.E., 2019. Geometric realization of algebraic conformal field theories. Advances in Mathematics, 349, pp.488-563.


\bibitem[Ten19b]{Ten19b}
Tener, J.E., 2019. Representation theory in chiral conformal field theory: from fields to
observables. Selecta Math. (N.S.) (2019), 25:76

\bibitem[Ten19c]{Ten19c}
Tener, J.E., 2019. Fusion and positivity in chiral conformal field theory. arXiv preprint arXiv:1910.08257.


\bibitem[Was94]{Was94}
Wassermann, A.J., 1995. Operator algebras and conformal field theory. In Proceedings of the International Congress of Mathematicians (pp. 966-979). Birkhäuser, Basel.


\bibitem[Was98]{Was98}
Wassermann, A., 1998. Operator algebras and conformal field theory III. Fusion of positive energy representations of $LSU(N)$ using bounded operators. Inventiones mathematicae, 133(3), pp.467-538.


\bibitem[Zhu96]{Zhu96}
Zhu, Y., 1996. Modular invariance of characters of vertex operator algebras. Journal of the American Mathematical Society, 9(1), pp.237-302.	
		
		
\end{thebibliography}
\end{document}